\documentclass[11pt]{article}

\usepackage{xcolor}
\usepackage{amsmath,soul,bm}
\usepackage{mathtools}
\usepackage{amssymb,mathrsfs,euscript}
\usepackage{enumerate}
\usepackage[normalem]{ulem}
\usepackage{cancel}
\usepackage{enumitem,todonotes}
\usepackage{tikz}
\usepackage{dsfont}
\usepackage{verbatim,authblk,amsthm}
\usepackage[round]{natbib}
\usepackage{float,graphicx,subcaption}
\usepackage{tabularx}
\newcommand*\circled[1]{\tikz[baseline=(char.base)]{
            \node[shape=circle,draw,inner sep=2pt] (char) {#1};}}
            
\newcommand{\inner}[2]{{\left\langle #1, #2 \right\rangle}}            
\newcommand{\norm}[1]{\left\|#1\right\|}

\newcommand{\X}{\mathcal X}

\newcommand{\R}{\mathbb R}

\newcommand{\E}{\mathbb E}
\newcommand{\h}{\mathscr H}
\allowdisplaybreaks

\newcommand{\Id}{\boldsymbol{I}}
\newcommand{\one}{\boldsymbol{1}}

\newcommand{\B}{\mathcal B}
\newcommand{\M}{\mathcal M}
\newcommand{\W}{\mathcal W}

\newcommand{\V}{\mathcal V}

\newcommand{\id}{\mathfrak J}
\newcommand{\ep}{\Upsilon}
\newcommand{\T}{\mathcal T}

\newcommand{\A}{\mathcal A}
\newcommand{\PP}{\mathcal P}
\newcommand{\s}{\mathcal S}
\newcommand{\D}{D}
\newcommand{\Ntl}{\mathcal{N}_{2}(\lambda)}

\newcommand{\Ntlsq}{\mathcal{N}^2_{2}(\lambda)}

\newcommand{\Nol}{\mathcal{N}_{1}(\lambda)}

\newcommand{\Cl}{C_{\lambda}}

\newcommand{\K}{\kappa}
\newcommand{\kk}{K}

\newcommand{\op}{\EuScript{L}^\infty(\h_{K})}

\newcommand{\opS}{\EuScript{L}^\infty(\h_{K_0})}

\newcommand{\opl}{\EuScript{L}^\infty(\Lp)}

\newcommand{\hs}{\EuScript{L}^2(\h_K)}
\newcommand{\hsS}{\EuScript{L}^2(\h_{K_0})}

\newcommand{\Lp}{L^{2}(P_0)}
\newcommand{\range}{\text{Ran}} 

\newcommand{\SgLP}{\A_{P,\lambda}^{-1/2}}
\newcommand{\SgL}{\A_{P_0,\lambda}^{-1/2}}

\newcommand{\SL}{\A_{P_0,\lambda}}
\newcommand{\SLP}{\A_{P,\lambda}}

\newcommand{\gsP}{g^{-1/2}_{\lambda}(\A_{P})}

\newcommand{\gShh}{g^{1/2}_{\lambda}(\hat{\A}_{P_0})}
\newcommand{\gSh}{g^{1/2}_{\lambda}(\hat{\A}_{P})}
\newcommand{\gS}{g_{\lambda}(\hat{\A}_{P})}

\newcommand{\gl}{g_{\lambda}}

\newcommand{\U}{u}

\newcommand{\stat}{\hat{\mathbb S}_{\lambda}}
\newcommand{\statt}{\mathbb S_{\lambda}}
\newcommand{\htens}{\otimes_{\h_{K_0}}}
\newcommand{\hhtens}{\otimes_{\h_{K}}}
\newcommand{\ltens}{\otimes_{\Lp}}

\newcommand{\PQ}{P_0}

\newcommand{\cd} {|\Lambda|}

\renewcommand{\epsilon}{\varepsilon}

\DeclarePairedDelimiter{\floor}{\lfloor}{\rfloor}

\newtheorem{theorem}{Theorem}

\newtheorem{corollary}[theorem]{Corollary}
\newtheorem{proposition}[theorem]{Proposition}
\theoremstyle{rem}
\newtheorem{remark}{Remark}

\newtheorem{assump}{Assumption}

\newenvironment{myassump}[2][]
  {\renewcommand\theassump{#2}\begin{assump}[#1]}
  {\end{assump}}

\theoremstyle{example} 
\newtheorem{example}{Example}

\newtheorem{appxthm}{Theorem}[section]
\newtheorem{appxlem}[appxthm]{Lemma}

\newtheorem{appxcoro}[appxthm]{Corollary}

\theoremstyle{definition}

\theoremstyle{remark}

\begin{document}
\title{Minimax Optimal Goodness-of-Fit Testing with \\ Kernel Stein Discrepancy }



\author[1]{Omar Hagrass}
\author[2]{Bharath Sriperumbudur}
\author[3]{Krishnakumar Balasubramanian}
\affil[1]{Department of Operations Research \& Financial Engineering, Princeton University}
\affil[2]{Department of Statistics, Pennsylvania State University}
\affil[3]{Department of Statistics, University of California, Davis}
\affil[1]{\texttt{oh2588@princeton.edu}}
\affil[2]{\texttt{\{bks18\}}@psu.edu}\affil[3]{\texttt{kbala}@ucdavis.edu}
\date{}
\maketitle

\begin{abstract}
We explore the minimax optimality of goodness-of-fit tests on general domains using the kernelized Stein discrepancy (KSD). The KSD framework offers a flexible approach for goodness-of-fit testing, avoiding strong distributional assumptions, accommodating diverse data structures beyond Euclidean spaces, and relying only on partial knowledge of the reference distribution, while maintaining computational efficiency. Although KSD is a powerful framework for goodness-of-fit testing, only the consistency of the corresponding tests has been established so far, and their statistical optimality remains largely unexplored. In this paper, we develop a general framework and an operator-theoretic representation of the KSD, encompassing many existing KSD tests in the literature, which vary depending on the domain. Building on this representation, we propose a modified discrepancy by applying the concept of spectral regularization to the KSD framework. We establish the minimax optimality of the proposed regularized test for a wide range of the smoothness parameter $\theta$ under a specific alternative space, defined over general domains, using the $\chi^2$-divergence as the separation metric. In contrast, we demonstrate that the unregularized KSD test fails to achieve the minimax separation rate for the considered alternative space. 
Additionally, we introduce an adaptive test capable of achieving minimax optimality up to a logarithmic factor by adapting to unknown parameters. Through numerical experiments, we illustrate the superior performance of our proposed tests across various domains compared to their unregularized counterparts.
\end{abstract}
\setlength{\parskip}{4pt}
\section{Introduction} 
Given $\mathbb{X}_n:=(X_i)_{i=1}^n\stackrel{i.i.d.}{\sim} P$, where $P$ is defined on a measurable space $\mathcal{X}$, a goodness-of-fit test involves evaluating the hypotheses $$
H_0 : P=P_0 \quad \text{vs.}\quad H_1: P \neq P_0,
$$ where $P_0$ denotes a fixed known distribution. See, for example,~\cite{lehmann,ingster2012nonparametric} for a textbook-level treatment of the subject. From a methodological perspective, several approaches exist for constructing goodness-of-fit tests, a few of which we discuss next. Approaches based on the concept of nearest-neighbor statistics have been explored when $\mathcal{X}$ is a $d$-dimensional Euclidean space~\citep{bickel1983sums, schilling1983goodness,schilling1983infinite}, or a  manifold~\citep{ebner2018multivariate}. Other tests for the case of manifolds include works by~\cite{mardia1984goodness},~\cite{jupp2008data} and \cite{boente2014goodness}, to name a few, which are invariably limited to the case of $P_0$ being uniform 
or $P_0$ being the von Mises--Fisher distribution on the sphere. Yet another approach for goodness-of-fit testing involves using the innovation process~\citep{khmaladze1988innovation,khmaladze1993goodness,khmaladze2016unitary}, which, to the best of our knowledge, is limited to Euclidean spaces. Recently, techniques based on optimal transport have been used by~\cite{hallin2021multivariate} for goodness-of-fit testing in the Euclidean setting. However, such approaches could potentially be extended for certain non-Euclidean spaces with additional computational hardships. Goodness-of-fit tests based on variations of $\chi^2$ tests have been analyzed, for example, in~\cite{khmaladze2013note} and~\cite{balakrishnan2019hypothesis} for discrete distributions. On the other hand, tests based on reproducing kernels have been developed for goodness-of-fit (and two-sample) testing problems~(e.g., see~\citealt{gretton12a,rkhs,Smola}), which are applicable on non-Euclidean domains.

Despite several decades of work on this classical problem, existing tests suffer from at least one of the following drawbacks: \vspace{-2mm}
\begin{itemize}
    \item[\textsf{(I)}] They rely on strong distributional assumptions, e.g., $P_0$ is Gaussian or belongs to some classical parametric family.
    \item[\textsf{(II)}] They cannot handle many non-Euclidean data, such as graphs, strings, functions, etc., which is prevalent in many modern applications such as social networks, bioinformatics, etc.
    \item[\textsf{(III)}] They require complete knowledge of $P_0$ and cannot handle distributions that are known only up to a normalization constant, even if $\mathcal{X}$ is Euclidean, for example, a restricted Boltzmann machine \citep{rbm}.
    \item[\textsf{(IV)}] They are not in general efficiently computable and may need additional samples from $P_0$ to make them computable even if the complete knowledge of $P_0$ is available.
\end{itemize}

\vspace{-1mm}

The above-mentioned issues could all be simultaneously addressed by appealing to the framework of \emph{Stein operators} for characterizing distributions, in combination with reproducing kernels. Let $\h$ be a Hilbert space on $\X$ and $\psi_{Q}$ be defined on $\X \times \X$ such that $\psi_{Q}(\cdot,x)\in\h,\,\forall x\in\mathcal{X}$ and $\E_{Q}\psi_{Q}(y,X)=0,\,\forall\,y\in\X$.
Define the Stein operator $\s_{Q}$ with respect to a probability measure $Q$ acting on $f \in \h$ as $$(\s_{Q}f)(x):= \inner{\psi_{Q}(\cdot,x)}{f}_{\h},$$ so that $\E_{Q}(\s_{Q}f)(X)=\inner{\E_{Q}\psi_{Q}(\cdot,X)}{f}_{\h}=0$ if $\E_{Q}\Vert \psi_{Q}(\cdot,X)\Vert_\h<\infty$. Then, the kernel Stein discrepancy between two distributions $P$ and $P_0$ is defined as 
\begin{align}
\D_{\mathrm{KSD}}(P,P_0):=
\sup_{f \in \h : \norm{f}_{\h}\leq 1} |\E_{P}(\s_{P_0}f)(X)|
= \norm{\E_P\psi_{P_0}(\cdot,X)}_{\h}=\sqrt{\E_{P\times P}K_0(X,Y)}, \label{eq:general KSD}
\end{align} 
where \begin{equation*}K_0(x,y):= \inner{\psi_{P_0}(\cdot,x)}{\psi_{P_0}(\cdot,y)}_{\h}
\end{equation*} is called the \emph{Stein kernel}. Since $K_0$ is a symmetric positive definite kernel, it follows from Moore-Aronszajn theorem \citep[Theorem 3, Chapter 1]{Berlinet} that there exists a unique reproducing kernel Hilbert space (RKHS), denoted as $\h_{K_0}$ as the reproducing kernel. Moreover, by the reproducing property of $K_0$, we have $K_0(x,y)=\inner{K_0(\cdot,x)}{K_0(\cdot,y)}_{\h_{K_0}}.$ We refer to $\h_{K_0}$ as the \emph{Stein RKHS}. It is easy to verify that if $P\mapsto \int_\mathcal{X} \psi_{P_0}(\cdot,x)\,dP(x)$ is injective 
\citep{universal}, then $D_{\mathrm{KSD}}(P,P_0)=0$ implies $P=P_0$, i.e., the goodness-of-fit testing problem can equivalently be posed as $H_0:D_{\mathrm{KSD}}(P,P_0)=0$ vs.~$H_1:D_{\mathrm{KSD}}(P,P_0)\ne 0$. Therefore, based on $(X_i)^n_{i=1}$, KSD-based goodness-of-fit test can be constructed using the $U$-statistic \begin{equation}\frac{1}{n(n-1)}\sum^n_{i\ne j}K_0(X_i,X_j)\label{Eq:test}\end{equation} with the test threshold being the $1-\alpha$ quantile of the asymptotic distribution of the test statistic under $H_0$---which can be obtained from the standard $U$-statistic theory---or that of i.i.d.~weighted bootstrap \cite{DEHLING1994392}. 
Note that this test addresses all the drawbacks \textsf{(I)}--\textsf{(IV)} if $K_0$ can be evaluated by only knowing $P_0$ up to a normalization. 

In the following, we provide some examples in the literature that specialize the above-mentioned general framework for specific domains $\mathcal{X}$. \vspace{-2mm}

\begin{example}[KSD on $\mathcal{X}=\mathbb{R}^d$]\label{example:ksd}
Suppose $\X=\R^d$ and $P_0$ have a smooth density $p_0$ w.r.t.~the Lebesgue measure. Let $\mathscr{H}_K$ be the RKHS associated with the reproducing kernel, $K$. By defining $\psi_{P_0}(\cdot,x)= K(\cdot,x)\nabla_x\log p_0(x)+\nabla_{x}K(\cdot,x)\in\mathscr{H}^d_K=:\mathscr{H}$ yields the KSD as proposed by \citet{chwialkowski16} and \citet{liub16} 
with
    \begin{align*}
     K_0(x,y) & = \nabla_{x} \log p_0(x)^{\top}\nabla_{y} \log p_0(y) K(x,y)+\nabla_{y} \log p_0(y)^{\top} \nabla_{x}K(x,y)\\
     & \qquad +\nabla_{x} \log p_0(x)^{\top} \nabla_{y}K(x,y)+\mathrm{Tr}(\nabla_{x}\nabla_{y}K(x,y)).
     \end{align*}
In fact, the Stein operator in this setting can be given as $(\mathcal{S}_{P_0}f)(x)=\langle K(\cdot,x)\nabla_x\log p_0(x)+\nabla_{x}K(\cdot,x),\underline{f}\rangle_{\mathscr{H}^d_K}=\nabla^\top_x \log p_0(x) \underline{f}(x)+\bm{1}^\top\nabla_x f(x)$, which is the well-known Langevin-Stein operator, where $\underline{f}:=(f,\stackrel{d}{\ldots},f)^\top$ and $\bm{1}:=(1,\stackrel{d}{\ldots},1)^\top$. The KSD test based on the aforementioned $U$-statistic was developed in these works, which addresses the issues \textsf{(I)}, \textsf{(III)}, and \textsf{(IV)} since $K_0$ can be evaluated even if $p_0$ is known only up to normalization.\vspace{-2mm}
\end{example}
\begin{example}[KSD on a Riemannian manifold]\label{example:manifold}
Suppose $\X=(\mathcal{M},g)$, where $\M$ is a smooth $d$-dimensional Riemannian manifold (with either empty boundary or the kernel $K$ defined on $\X\times\X$ vanishes on the boundary) with Riemannian metric $g$ and $P_0$ has a smooth density $p_0$. The choice of $\psi_{P_0}(\cdot,x)=K(\cdot,x)\nabla_{\theta_x}\log (p_0(x)J(\theta))+\nabla_{\theta_x}K(\cdot,x)\in\mathscr{H}^d_K,\,x\in\mathcal{M},$
 yields the \emph{manifold kernel Stein discrepancy} ($\mathrm{mKSD}^{(1)}$) \citep{xu2021Manifold} which also covers the \emph{directional Stein discrepancy (dKSD)} \citep{xu2020directional} as a special case when $\M$ is the unit hypersphere $\mathbb{S}^{d-1}$. Here, $\theta=(\theta^1,\ldots,\theta^d)$ is a coordinate system on $\mathcal{M}$ that covers $\mathcal{M}$ almost everywhere,  $J=\sqrt{\text{det}\, g}$ is the volume element, and $\theta_x$ is the coordinate representation of $x\in\mathcal{M}$. Other variations of \emph{mKSD} discussed in \citep{xu2021Manifold} are covered by choosing the appropriate $\psi_{P_0}$. Note that this setting covers $P_0$ defined on a torus, sphere (e.g., uniform, von Mises-Fisher, Fisher-Bingham distribution), rotation group (e.g., Fisher distribution), etc., thereby addressing all the issues \textsf{(I)}--\textsf{(IV)} raised before.\vspace{-2mm}
\end{example}
\begin{example}[KSD on Hilbert spaces]\label{example:hilbert}
For a Gibbs measure, $P_0$ defined on a separable Hilbert space, $\mathcal{X}$, i.e., $dP_0/dN_C$ exists on $\mathcal{X}$ with $dP_0/dN_C \propto \exp(-U)$, where $N_C$ is a Gaussian measure with mean zero and covariance operator $C$, the Stein operator on $\mathcal{X}$ is defined through $\psi_{P_0}(\cdot,x)=\emph{Tr}[CD^2k(\cdot,x)]-\langle Dk(\cdot,x),x+CDU(x)\rangle_\mathcal{X},\,\,x\in\mathcal{X},$
where $D$ denotes the Fr\'{e}chet differential~\citep[Definition 3.2]{InfiniteDimKSD}. Examples of distributions that can be tested in this setting include Gaussian processes on a separable Hilbert space.
\end{example}
\vspace{-4mm}
Similar concrete Stein operators have been defined on 
exponential random graph models~\citep{randomgraph,xu2022agrasst}, point processes~\citep{pointprocesses}, Lie groups~\citep{qu2024kernel}, discrete distributions~\citep{discretedata}, latent variable models~\citep{latentmodels}, and censored data~\citep{censored-data}, which can be used to construct KSD-based goodness-of-fit tests on appropriate domains, beyond $\mathbb{R}^d$, thereby addressing the issues \textsf{(I)}--\textsf{(IV)}. To avoid the technical notation required to define these Stein operators, we do not introduce them here; however, the proposed framework subsumes all these non-Euclidean scenarios involving graph data, point process data, discrete data, censored data, and functional data.

Although the above discussion highlights the strength of the KSD-based approach, to the best of our knowledge, the statistical optimality of the corresponding KSD tests is not known. 
The central goal of this paper is to 
investigate and develop minimax optimal goodness-of-fit tests based on KSD that are devoid of the issues \textsf{(I)}--\textsf{(IV)}.
\subsection{Minimax framework}\label{subsec:minmaxframe}
Since the goal of our work is to investigate the minimax optimality of the KSD test, we first introduce the minimax framework for hypothesis testing pioneered by \citet{Burnashev} and \citet{Ingester1, Ingester2}, which is essential to understand our contributions. 

Let $\phi(\mathbb{X}_n)$ be any test that rejects $H_0$ when $\phi=1$ and fails to reject $H_0$ when $\phi=0$. Denote the class of all such asymptotic $\alpha$-level tests to be $\Phi_\alpha$. Let $\mathcal{C}$ be a set of probability measures on $\mathcal{X}$. The Type-II error of a test $\phi\in \Phi_\alpha$ w.r.t.~$\mathcal{P}_\Delta$ is defined as
$R_\Delta(\phi)=\sup_{P\in\mathcal{P}_\Delta}\mathbb{E}_{P^n}(1-\phi),$ 
where $$\mathcal{P}_\Delta:=\left\{P\in \mathcal{C}:\rho^2(P,P_0)\ge \Delta\right\},$$ is the class of $\Delta$-separated alternatives in probability metric $\rho$, with $\Delta$ being referred to as the \emph{separation boundary} or \emph{contiguity radius}. Of course, the interest is in letting $\Delta\rightarrow 0$ as $n\rightarrow \infty$ (i.e., shrinking alternatives) and analyzing $R_\Delta$ for a given test, $\phi$, i.e., whether $R_\Delta(\phi)\rightarrow 0$. In the asymptotic setting, the \emph{minimax separation} or \emph{critical radius} $\Delta^*$ is the fastest possible order at which $\Delta\rightarrow 0$ such that $\lim\inf_{n\rightarrow\infty}\inf_{\phi\in\Phi_\alpha}R_{\Delta^*}(\phi)\rightarrow 0$, i.e., for any $\Delta$ such that $\Delta/\Delta^*\rightarrow\infty$, there is no test $\phi\in\Phi_\alpha$ that is consistent over $\mathcal{P}_\Delta$. A test is \emph{asymptotically minimax optimal} if it is consistent over $\mathcal{P}_\Delta$ with $\Delta \asymp\Delta^*$.

In this work, we consider a class of $\Delta$-separated alternatives in $\chi^2$-divergence, defined in \eqref{Eq:alternative-theta}, of the form 
\begin{equation*}
\PP:=\PP_{\theta,\Delta}:= \left\{P: \frac{dP}{dP_0}-1 \in \range (\ep^{\theta}_{P_0}), 
 \ \chi^2(P,P_0)=\norm{\frac{dP}{dP_0}-1}^2_{L^2(P_0)} \geq \Delta\right\},\,\,\theta > 0, 
\end{equation*} where $\ep_{P_0}:L^2(P_0)\rightarrow L^2(P_0),\,f\mapsto \int_\X K_0(\cdot,x)f(x)\,dP_0(x)$ is the integral operator induced by $K_0$, and investigate the optimality of the KSD test w.r.t.~$\PP$. The condition $dP/dP_0-1\in \text{Ran}(\ep^\theta_{P_0})$ captures the smoothness of $dP/dP_0-1$ and the degree of smoothness is controlled by $\theta$, i.e., the larger $\theta$ is, more smooth is $dP/dP_0-1$. $\text{Ran}(\ep^\theta_{P_0})$---range space of $\ep^\theta_{P_0}$---for $\theta\in(0,\frac{1}{2}]$ can be interpreted as an interpolation space obtained by the real interpolation of $\h_{K_0}$ (corresponds to $\theta=\frac{1}{2}$) and $L^2(P_0)$ (corresponds to $\theta=0$) at scale $\theta$ \citep[Theorem 4.6]{Steinwart2012MercersTO}.


We would like to highlight that unlike the classical smoothness classes such as Sobolev or H\"{o}lder which are defined for functions defined on $\mathcal{X}=\mathbb{R}^d$, the smoothness class we considered above is defined through the spectrum of the operator $\ep_{P_0}$ that takes into account the interaction between $K_0$ and $P_0$, and applies to domains that are not necessarily Euclidean. Such a smoothness condition is standard in kernel-based methods like kernel regression \citep{Caponnetto-07} and has also been used in MMD-based two-sample and goodness-of-fit tests \citep{Krishna2021Gof, twosampletest, hagrass2023Gof}. 

\subsection{Contributions}\label{subsec:contributions}
Having established the importance of KSD as a discrepancy measure for goodness-of-fit testing, in this work, we investigate the minimax optimality of the KSD-based goodness-of-fit test.
To the best of our knowledge, only consistency of the test based on the test statistic in \eqref{Eq:test} has been established for $\X=\R^d$ with $\psi_{P_0}$ as defined in Example~\ref{example:ksd}.  
To this end, the following are the main contributions of our work.\vspace{1mm}\\
\emph{(i)} 
First, in Section~\ref{sec:optimal}, we provide an operator-theoretic representation of $\D_{\textrm{KSD}}$ (defined in \eqref{eq:general KSD}) in Proposition~\ref{pro:ksd} by showing that 
$\D^2_{\textrm{KSD}}(P,P_0)=\inner{\ep_{P_0} u}{u}_{\Lp}= \sum_{i\geq 1} \lambda_i \langle u,\tilde{\phi}_i\rangle^2_{\Lp}, $
where $(\lambda_i, \tilde{\phi}_i)_i$ are the eigenvalues and eigenfunctions of $\ep_{P_0}$, and $u:=dP/dP_0 - 1$. It follows from the above representation 
that the Fourier coefficients of $u$, i.e., $(\langle u, \tilde{\phi}_i\rangle_{L^2(P_0)})_i$, are downweighted by the eigenvalues of $\ep_{P_0}$, which decay to zero as $i \rightarrow \infty$. This implies that if the differences between $P$ and $P_0$ are hidden in the high-frequency (i.e., large $i$) coefficients of $u$, these differences may not be sufficiently captured by the empirical version of $D^2_{\textrm{KSD}}(P,P_0)$---one choice is shown in \eqref{Eq:test}---, which serves as the test statistic for the KSD test. To address this issue, we modify $D^2_{\textrm{KSD}}(P,P_0)$ to 
$
D^2_\lambda(P,P_0):=\sum_i \lambda_i g_\lambda(\lambda_i)\langle u, \tilde{\phi}_i\rangle^2_{L^2(P_0)},
$
where $g_\lambda: (0,\infty) \rightarrow (0,\infty)$ is a spectral regularizer satisfying $x g_\lambda(x) \asymp 1$, i.e., $g_\lambda(x)$ approximates $x^{-1}$ as $\lambda \rightarrow 0$ (an example of such a regularizer is the Tikhonov regularizer, $g_\lambda(x) = (x + \lambda)^{-1}$). This property of the spectral regularizer ensures that $D^2_\lambda(P,P_0) \rightarrow \sum_i \langle u, \tilde{\phi}_i\rangle^2_{L^2(P_0)}$ as $\lambda \rightarrow 0$, thereby mitigating the aforementioned limitation of KSD, i.e., the down-weighting of the Fourier coefficients of $u$. A similar idea has been explored in MMD-based goodness-of-fit testing \citep{Krishna2021Gof,hagrass2023Gof}. However, as aforementioned, such MMD-based tests suffer from issues \textsf{(III)} and \textsf{(IV)}. Interestingly, we observe that a direct application of spectral regularization to the KSD, as described above, still faces issue \textsf{(IV)}, as detailed in the following paragraph. Therefore, one of the key contributions of this work is in proposing a test statistic by employing spectral regularization within the KSD framework, while addressing the challenges \textsf{(I)}-\textsf{(IV)}. We show that the spectral regularized version of KSD, $D^2_\lambda(P,P_0)$, is equal to $\Vert g^{1/2}_\lambda(\mathcal{A}_{P_0})\Psi_{P}\Vert^2_{\mathscr{H}_{K_0}}$, where $\Psi_{P} := \int K_0(\cdot,x) \, dP(x)$ and $\mathcal{A}_{P_0} := \int K_0(\cdot,x)\otimes_{\mathscr{H}_{K_0}} K_0(\cdot,x) \, dP_0(x)$ are the mean element of $P$ and the covariance operator of $P_0$ with respect to $K_0$. If an empirical version of $D^2_\lambda(P,P_0)$ is used as the test statistic, it still suffers from issue $\textsf{(IV)}$, since $g^{1/2}_\lambda(\mathcal{A}_{P_0})$ is generally computable for only certain $(K_0,P_0)$ pairs and may require samples from $P_0$ to estimate $\mathcal{A}_{P_0}$. This, in turn, defeats the purpose of using KSD. Therefore, we consider a modification of $D^2_\lambda$ as 
$
\tilde{D}^2_\lambda := \Vert g^{1/2}_\lambda(\mathcal{A}_{P})\Psi_P\Vert^2_{\mathscr{H}_{K_0}},
$
which can be empirically estimated based on the given samples $(X_1,\ldots,X_n)$ and does not require additional samples from $P_0$ or complete knowledge of $P_0$, but only information about $P_0$ up to a normalization constant. Thus this contribution overcomes all the shortcomings raised in $\textsf{(I)}$--$\textsf{(IV)}$. 
Using a $U$-statistic estimator of $\tilde{D}_\lambda$ as the test statistic, we propose a goodness-of-fit test using wild bootstrap and compute its separation boundary with respect to $\mathcal{P}$ in Corollaries~\ref{coro:poly} and \ref{coro:exp} (also see Corollaries \ref{coro:general:poly}, \ref{coro:general:exp}, and Theorem~\ref{thm:sep-bound-suff-conditions}). Furthermore, we establish the minimax optimality of the proposed test for a wide range of $\theta$ by deriving the minimax separation radius of the space $\mathcal{P}$ in Theorem~\ref{thm:minimax}, showing it to be $\Delta_n = o(n^{-4\theta\beta/(4\theta\beta+1)})$ as $n \rightarrow \infty$ if $\lambda_i \asymp i^{-\beta}, \, \beta > 1$ (i.e., polynomial decay of eigenvalues), and $\Delta_n = o(\sqrt{\log n}/n)$ as $n \rightarrow \infty$ if $\lambda_i \asymp e^{-i}$ (i.e., exponential decay of eigenvalues). This matches the separation boundary of the proposed test across a wide range of $\theta$. Finally, in Theorem~\ref{thm:non-optimality-KSD}, we show that the KSD test (based on the test statistic in \eqref{Eq:test}) cannot achieve a separation radius better than $\Delta_n = o(n^{-2\theta/(2\theta+1)})$ as $n \rightarrow \infty$, irrespective of the decay rate of the eigenvalues of $\ep_{P_0}$—the eigenvalue decay rate determines the smoothness of functions in $\mathscr{H}_{K_0}$---, thus establishing its non-optimality.\vspace{1mm}\\
\emph{(ii)} 
Note that the regularized test presented in Section~\ref{sec:optimal} depends on the regularization parameter $\lambda>0$ which has to be chosen appropriately. The optimal choice of $\lambda$ that yields a minimax optimal test depends on the unknown parameters $\theta$ and $\beta$ (in the case of polynomial decay of eigenvalues). Therefore, it is important to construct a minimax optimal test that does not require information about $\theta$ and $\beta$. To this end, by aggregating tests constructed for different $\lambda$, in Section~\ref{subsec:adaptation}, we construct an adaptive test (precisely, a union test) that is minimax optimal w.r.t.~$\mathcal{P}$ up to a $\log\log$ factor.\vspace{1mm}\\
\emph{(iii)} 
Since the class of alternatives, $\mathcal{P}$ as defined in Section~\ref{subsec:minmaxframe} is abstract, to develop a deeper understanding, in Section~\ref{sec:range}, we provide concrete examples to elaborate the range space condition that appears in $\mathcal{P}$. \vspace{1mm}
\\
\emph{(iv)} 
In Section~\ref{sec:experiments}, we provide numerical experiments on $\mathbb{R}^d$, $\mathbb{S}^d$ (sphere), and a separable Hilbert space, comparing the performance of KSD and regularized KSD tests. We show that the regularized KSD tests have superior empirical performance over the KSD test in all these scenarios.
\vspace{2mm}\\
\noindent The proofs of all the results of this paper are provided in Section \ref{sec:proofs}. 
\section{Definitions \& notation}
For a topological space $\X$, $L^r(\X,\mu)$ denotes the Banach space of $r$-power $(r\geq 1)$ $\mu$-integrable function, where $\mu$ is a finite non-negative Borel measure on $\X$. For $f \in L^r(\X,\mu)=:L^r(\mu)$, $\norm{f}_{L^r(\mu)}:=(\int_{\X}|f|^r\,d\mu)^{1/r}$ denotes the $L^r$-norm of $f$. $\mu^n := \mu \times \stackrel{n}{...} \times \mu$ is the $n$-fold product measure. $\h_K$ denotes a reproducing kernel Hilbert space with a reproducing kernel $K: \X \times \X \to \R$. $[f]_{\sim}$ denotes the equivalence class associated with $f\in L^r(\X,\mu)$, that is the collection of functions $g \in L^r(\X,\mu)$ such that $\norm{f-g}_{L^r(\mu)}=0$. For two measures $P$ and $Q$, $P \ll Q$ denotes that $P$ is dominated by $Q$ which means, if $Q(A)=0$ for some measurable set $A$, then $P(A)=0$. Let $H_1$ and $H_2$ be abstract Hilbert spaces. $\EuScript{L}(H_1,H_2)$ denotes the space of bounded linear operators from $H_1$ to $H_2$. For $S \in \EuScript{L}(H_1,H_2)$, $S^*$ denotes the adjoint of $S$. $S \in \EuScript{L}(H) := \EuScript{L}(H,H)$ is called self-adjoint if $S^*=S$. For $S \in \EuScript{L}(H)$, $\text{Tr}(S)$, $\norm{S}_{\EuScript{L}^2(H)}$, and $\norm{S}_{\EuScript{L}^{\infty}(H)}$ denote the trace, Hilbert-Schmidt and operator norms of $S$, respectively. For $x,y \in H$, $x \otimes_{H} y$ is an element of the tensor product space of $H \otimes H$ which can also be seen as an operator from $H \to H$ as $(x \otimes_{H}y)z=x\inner{y}{z}_{H}$ for any $z \in H$. For constants $a$ and $b$, $a \lesssim b$ (resp. $a \gtrsim b$) denotes that there exists a positive constant $c$ (\emph{resp.} $c'$) such that $a\leq cb$ (\emph{resp.} $a \geq c' b)$. $a \asymp b$ denotes that there exists positive constants $c$ and $c'$ such that  $cb \leq a \leq c' b$. We denote $[\ell]$ for $\{1,\ldots,\ell\}$.

\section{KSD-based tests via spectral regularization}\label{sec:optimal}
We begin by examining the operating characteristics of the kernel Stein discrepancy (KSD) based goodness-of-fit tests, by viewing KSD from an operator perspective. This examination will assist in demonstrating the non-optimality of such tests and in introducing the concept of spectral regularization. Throughout the paper, we maintain the following assumption. 

\begin{myassump}{$\mathbf{A_0}$}\label{assump:a0}
$(\mathcal{X},\mathcal{B})$ is a second countable (i.e., completely separable) space endowed with Borel $\sigma$-algebra $\mathcal{B}$. $(\h_{K_0},K_0)$ is an RKHS of real-valued functions on $\X$ with a continuous reproducing kernel $K_0$ satisfying $\int_{\X} K_0(x,x)\, dP_0(x) < \infty.$ 
\end{myassump} 
The following result presents an operator-theoretic representation of KSD.
\begin{proposition}\label{pro:ksd}
Define the inclusion operator $\id : \h_{K_0} \to \Lp$, $f \mapsto [f]_{\sim}$. Then under \ref{assump:a0},
\begin{equation}
  \D_{\mathrm{KSD}}^2(P,P_0) = \inner{\id^*u}{\id^*u}_{\h_{K_0}}= \inner{\ep_{P_0} u}{u}_{\Lp}= \sum_{i\geq 1} \lambda_i \langle u,\Tilde{\phi_i}\rangle^2_{\Lp}, \label{KSD-rep}  
\end{equation}
where $u:=\frac{dP}{dP_0}-1$, $\id^*$ is the adjoint of $\id$, defined as $\id^* : \Lp \to \h_{K_0}, \ f \mapsto \int K_0(\cdot,x)f(x)\,dP_0(x),$ $\ep_{P_0}:=\id\id^*$ is the integral operator with $(\lambda_i,\tilde{\phi}_i)_i$ being its eigensystem. \vspace{-2mm}
\end{proposition}  
\begin{proof}
From \eqref{eq:general KSD}, we have 
\begin{align}
\D_{\mathrm{KSD}}^2(P,P_0) &= \E_{P \times P} (K_0(X,Y)) 
= \E_{P \times P}\inner{K_0(\cdot,X)}{K_0(\cdot,Y)}_{\h_{K_0}} \nonumber\\  
&\stackrel{(*)}{=} \inner{\E_{P} K_0(\cdot,X)-\E_{P_0} K_0(\cdot,X)}{\E_{P} K_0(\cdot,Y)-\E_{P_0} K_0(\cdot,Y)}_{\h_{K_0}} \nonumber\\ 
&= \inner{\E_{P_0}\left(K_0(\cdot,X)u(X)\right)}{\E_{P_0}\left(K_0(\cdot,Y)u(Y)\right)}_{\h_{K_0}}\nonumber\\
&=\inner{\id^*u}{\id^*u}_{\h_{K_0}}=\inner{\id\id^*u}{u}_{\Lp}= \inner{\ep_{P_0} u}{u}_{\Lp}.\label{eq:temp}
\end{align}
In the above computation, $(*)$ is derived by noticing that $\E_{P_0}K_0(x,Y)=\inner{\psi_{P_0}(\cdot,x)}{\E_{P_0}\psi_{P_0}(\cdot,Y)}_{\h}=0$ for all $x\in\mathcal{X}$. 
Under~\ref{assump:a0}, it is easy to verify that $\ep_{P_0}$ is a positive self-adjoint trace class operator since $\text{Tr}(\ep_{P_0})=\text{Tr}(\id\id^*)=\text{Tr}(\id^*\id)=\text{Tr}(\int K_0(\cdot,x)\otimes_{\h_{K_0}}K_0(\cdot,x)\,dP_0(x))=\int K_0(x,x)\,dP_0(x)<\infty$. Thus, the spectral theorem \citep[Theorems VI.16, VI.17]{Reed} yields that 
$\ep_{P_0} = \sum_{i \in I} \lambda_i \Tilde{\phi_i} \ltens \Tilde{\phi_i},$
where $(\lambda_i)_i \subset \R^+ $ are the eigenvalues and $(\Tilde{\phi}_i)_i$ are the orthonormal system of eigenfunctions of $\ep_{P_0}$ that span $\overline{\range(\ep_{P_0})}$ with the index set $I$ being either countable in which case $\lambda_i \to 0$ or finite. Therefore, the result follows by using the spectral representation of $\ep_{P_0}$ in \eqref{eq:temp}. \vspace{-1mm}
\end{proof}
\begin{remark} Since $\Tilde{\phi_i}$ represents an equivalence class in $\Lp$, by defining $\phi_i:= \id^* \Tilde{\phi_i}/\lambda_i$, it is clear that $\id\phi_i=[\phi_i]_{\sim}=\Tilde{\phi_i}$ and $\phi_i \in \h_{K_0}$. Throughout the paper, $\phi_i$ refers to this definition. \vspace{-1mm}
\end{remark}

\eqref{KSD-rep} shows that KSD is a weighted sum of the squared Fourier components of $u$ with the weights being the eigenvalues of the operator $\ep_{P_0}$. 
This representation of KSD in \eqref{KSD-rep} suggests that the KSD could overlook discrepancies between $P$ and $P_0$ occurring at higher frequency components as $\lambda_i \to 0$. 
To address this issue, we use the idea of spectral regularization as considered in \citet{Krishna2021Gof} and \citet{hagrass2023Gof,twosampletest} for the MMD test by using a regularization function designed to avoid neglecting the higher frequency components. For that matter, let 
\begin{equation}
\D_{\lambda}^2(P,P_0):= \inner{\ep_{P_0}g_{\lambda}(\ep_{P_0})u}{u}_{\Lp}=\sum_{i\geq 1} \lambda_i g_{\lambda}(\lambda_i) \langle u,\Tilde{\phi_i}\rangle^2_{\Lp}, \label{eq: reg-KSD}    
\end{equation}
 where $g_\lambda:(0,\infty)\rightarrow (0,\infty)$, $\lambda>0$ is a \emph{spectral regularizer} (satisfies $\lim_{\lambda\rightarrow 0} xg_\lambda(x)\asymp 1$, i.e., $g_\lambda(x)$ is an approximation to $x^{-1}$ with the approximation being controlled by $\lambda$), defined as
$$g_{\lambda}(\mathcal{C}) : = \sum_{i\geq 1} g_{\lambda}(\tau_i) (\alpha_i \otimes_H \alpha_i) + g_{\lambda}(0)\left(\Id - \sum_{i\geq 1} \alpha_i \otimes_H \alpha_i\right),$$
with $\mathcal{C}$ being any compact, self-adjoint operator defined on a separable Hilbert space, $H$ and $(\tau_i,\alpha_i)_i$ are the eigenvalues and eigenfunctions of $\mathcal{C}$. A common choice for regularization is the \emph{Tikhonov regularizer}, $g_{\lambda}(x)=(x+\lambda)^{-1},$ where $\lambda>0$ is regularization parameter. As $\lambda \to 0$,\, $\lambda_i/(\lambda_i+\lambda) \to 1$, hence uniformly weighting the frequency components. Another choice we will use is by setting $g_{\lambda}(x)=\max \left\{1,(x+\lambda)^{-1}\right\}$---we refer to this regularizer as \emph{TikMax regularizer}---which only reweighs frequency components corresponding to $\lambda_i < 1- \lambda$ and keeps the same weights corresponding to lower frequency components when $\lambda_i > 1- \lambda$. 



The results we present below hold for any spectral regularization function that satisfies the following sufficient conditions: \vspace{-2mm}
\begin{align*}
    &(E_1) \quad \sup_{x \in \Gamma}|x \gl(x)| \leq C_1, &&(E_2) \quad \sup_{x \in \Gamma}|\lambda\gl(x)| \leq C_2, \\ 
    &(E_3) \quad \sup_{\{x\in \Gamma: x\gl(x)<B_3\}} |B_3-x\gl(x)|x^{2\varphi} \leq C_3 \lambda^{2\varphi},
    & &(E_4) \quad \inf_{x \in \Gamma} \gl(x)(x+\lambda) \geq C_4, \vspace{-2mm}
\end{align*}
where $\Gamma : = [0,\K]$, $\varphi \in (0,\xi]$, the constant $\xi$ is called the \emph{qualification} of $\gl$, and $(C_i)^4_{i=1}$ and $B_3$ are finite positive constants (all independent of $\lambda>0$). These are common conditions on the spectral regularizer employed in the inverse problem literature (see, for instance, \citealp{Engl.et.al}) and have also been employed in kernel ridge regression \citep{BAUER200752} and kernel testing \citep{hagrass2023Gof,twosampletest}. The key condition is $(E_3)$ which provides a quantitative estimate on the amplified version of the approximation error in approximating $x^{-1}$ by $g_\lambda(x)$, and the qualification $\xi$ imposes a bound on the amplification factor. It is easy to verify that both Tikhonov and TikMax regularizers have $\xi=\infty$. We would like to mention that the popular spectral cut-off regularizer defined as $g_\lambda(x)=x^{-1}\mathds{1}_{\{x\ge \lambda\}}$ also has $\xi=\infty$ but does not satisfy $(E_4)$.



The regularized KSD in \eqref{eq: reg-KSD} can be written in a form that is more suitable for estimation as \vspace{-1mm}
\begin{align}
\D^2_\lambda(P,P_0)&=\langle \ep_{P_0} g_\lambda(\ep_{P_0})u,u\rangle_{\Lp}=\langle \id\id^*g_\lambda(\id\id^*)u,u\rangle_{\Lp}=\langle g_\lambda(\A_{P_0}) \Psi_P,\Psi_P\rangle_{\mathscr{H}_{K_0}}\nonumber\\
&=\norm{g^{1/2}_{\lambda}(\A_{P_0}) \Psi_P}_{\h_{K_0}}^2,\label{eq:reg-discrp} \vspace{-2mm}
\end{align}
where \vspace{-3mm} $$\A_{P_0}:=\int_{\X} K_0(\cdot,x) \htens K_0(\cdot,x) \, dP_0(x), \hspace{2mm}\text{and}\,\, \Psi_P := \int_{\X} K_0(\cdot,x) \, dP(x).$$

Based on the form in \eqref{eq:reg-discrp}, note that a test statistic would involve estimating $\Psi_P$ from $(X_i)^n_{i=1}$ and then would need either the complete knowledge of $P_0$ so that $\mathcal{A}_{P_0}$ could be computed or access to independent samples $\mathbb Y_m:=(Y^0_i)_{i=1}^m$ from $P_0$ so that $\mathcal{A}_{P_0}$ can be estimated. However, this defeats the purpose of working with KSD. 
Therefore, to exploit the power of KSD, we propose a variation to \eqref{eq:reg-discrp} so that this variation can be completely estimated based on $(X_i)^n_{i=1}$. To this end, define
\begin{equation}\tilde{D}^2_\lambda(P,P_0):=\left \Vert g^{1/2}_\lambda(\mathcal{A}_P)\Psi_P\right\Vert^2_{\mathscr{H}_{K_0}}.\label{Eq:ksd-approx}\end{equation}
We argue that \eqref{Eq:ksd-approx} is a valid approximation to \eqref{eq:reg-discrp} since if $P$ is closer to $P_0$, then it is clearly a good approximation. On the other hand, if $P$ is far from $P_0$, then the testing problem is easy and so $P$ and $P_0$ could be distinguished without a significant influence from the regularization term. 


To this end, we propose a test based on the following estimator of $\tilde{D}^2_\lambda$ as the test statistic,
 \begin{equation}
  \stat^{P}(\mathbb X_n) = \frac{1}{n_1(n_1-1)}\sum_{1\leq i\neq j \leq n_1} \inner{\gSh K_0(\cdot,X_i)}{\gSh K_0(\cdot,X_j)}_{\h_{K_0}},  
  \label{eq:test-stat}   
 \end{equation}
with 
\begin{equation*}
\hat{\A}_{P} :=\frac{1}{n_2}\sum_{i=1}^{n_2} K_0(\cdot,Z_i) \htens K_0(\cdot,Z_i),
\end{equation*}
where the given samples $\mathbb{X}_n:=(X_i)_{i=1}^n$ are divided into  $(X_i)_{i=1}^{n_1}$ and $\mathbb Z_{n_2}:=(Z_i)_{i=1}^{n_2}:=(X_i)_{i=n_1+1}^n$, $n_2=n-n_1$, with $(X_i)^{n_1}_{i=1}$ being used to estimate $\Psi_P$ and $\mathbb{Z}_{n_2}$ being used to estimate $\mathcal{A}_P$. 
The advantage of the test statistic in \eqref{eq:test-stat} is that it is completely computable just using the given samples $\mathbb{X}_n$ and knowing $P_0$ up to a normalization (as in Examples~\ref{example:ksd}--\ref{example:hilbert}). 
Observe that under the null, we have $\E(\stat^P)=\E(\stat^{P_0})=0$. When $P \neq P_0$, but $P$ closer to $P_0$, we expect $\stat^{P}(\mathbb{X}_n)$ to be closer to $\stat^{P_0}(\mathbb{X}_n)$ (defined by replacing $\hat{\mathcal{A}}_P$ by $\mathcal{A}_{P_0}$ or its estimator $\hat{\mathcal{A}}_{P_0}$) and therefore will inhert the minimax optimality of $\stat^{P_0}(\mathbb{X}_n)$, as shown in Corollaries \ref{coro:poly} and \ref{coro:exp}.

The following result provides the computational form of the statistic in \eqref{eq:test-stat} and shows that it can be computed by solving a finite-dimensional eigensystem.  

\begin{proposition}\label{thm: computation}
Let $(\hat{\lambda}_i, \hat{\alpha_i})_i$ be the eigensystem of $ K_{n_2}/n_2$ where $[K_0(Z_i,Z_j)]_{i,j \in [n_2]}=:K_{n_2}$. Define $G := \sum_{i}\hat{\lambda}_i^{-1}\left( g_{\lambda}(\hat{\lambda}_i)-g_{\lambda}(0)\right)\hat{\alpha}_i\hat{\alpha}_i^\top.$ Then
\begin{equation*}
\stat^{P}(\mathbb{X}_n) = \frac{1}{n_1(n_1-1)} \left(\one_{n_1}^\top M \one_{n_1}-\mathrm{Tr}(M)\right), \ \ \textit{where} \ \ M=g_{\lambda}(0)K_{n_1}+\frac{1}{n_2}K_{n_{1},n_{2}} G K_{n_{1},n_{2}}^\top,
\end{equation*}
with $K_{n_1}:=[K_0(X_i,X_j)]_{i,j \in [n_1]}$ and $K_{n_{1},n_{2}}:=[K_0(X_i,Z_j)]_{i \in [n_1],j \in [n_2]}$. 
\end{proposition}
The complexity of computing $\stat^P(\mathbb{X}_n)$ is $O(n_2^3+n_1^2+n_2^{2}n_1)$. It is worth noting that this represents a savings of a factor of order $m^2+m n_2^2$ compared to the regularized MMD test statistic of \citet[Theorem 7]{hagrass2023Gof}, which involves sampling $m$ i.i.d.~samples from $P_0$,
and it also eliminates the time required for sampling those $m$ samples.

To compute the testing threshold,  we adopt a bootstrap approach \citep{Arcones92,Huskva, DEHLING1994392}, specifically the i.i.d. weighted bootstrap \citep{DEHLING1994392}, also known as the wild bootstrap to compute the test threshold. Let
$$\stat^{\epsilon}(\mathbb{X}_n):=\frac{1}{n_1(n_1-1)} \sum_{i\neq j} \epsilon_i \epsilon_j \inner{\gSh K_0(\cdot,X_i)}{\gSh K_0(\cdot,X_j)}_{\h_{K_0}},$$
where $\epsilon \in \{-1,1\}^{n_1}$ is a random vector of $n_1$ i.i.d Rademacher random variables. The following result shows that the distribution of $\stat^\epsilon(\mathbb{X}_n)$ can be used to approximate the distribution of $\stat^P(\mathbb{X}_n).$ Thus, for the test threshold, we use the $(1-\alpha)$ quantile of $\stat^{\epsilon}(\mathbb{X}_n)$, defined as
$$q^{\lambda}_{1-\alpha}:= \inf\left\{q \in \R: P_{\epsilon}\{\stat^{\epsilon} (\mathbb{X}_n)\leq q | \mathbb{X}_n\} \geq 1-\alpha \right\}.$$ 

\begin{remark} \label{rem:practical quantile}
In practice, the calculation of the exact quantile $q^{\lambda}_{1-\alpha}$ is computationally expensive. Thus, a Monte Carlo approximation is often used to approximate the exact quantile using $B$ independent realizations of $\stat^{\epsilon}$ defined as \vspace{-3mm}
$$q^{\lambda,B}_{1-\alpha}:= \inf\left\{q \in \mathbb{R}: \frac{1}{B} \sum_{k=1}^{B} \mathbb{I}(\stat^{\epsilon_k}(\mathbb{X}_n) \leq q) \geq 1-\alpha \right\}, \vspace{-3mm}
$$
where $(\epsilon_k)_{k=1}^B$ are $B$ randomly selected Rademacher vectors out of $2^n$ total possibilities. This can be justified using the Dvoretzky-Kiefer-Wolfowitz inequality (see \citealt{DKW}, \citealt{Massart}) to show that the approximation error between $q^{\lambda}_{1-\alpha}$ and $q^{\lambda,B}_{1-\alpha}$ gets arbitrarily small for sufficiently large $B$. Hence, all the results involving $q^{\lambda}_{1-\alpha}$ can be extended to $q^{\lambda,B}_{1-\alpha}$ for sufficiently large $B$.
\end{remark}
\begin{theorem}\label{Type-I error}
    Suppose $P=P_0$ and \ref{assump:a0} holds.  Let $\mathcal{D}_1$ be the distribution of $n_1\stat^{\epsilon}(\mathbb{X}_n)$ conditioned on $\mathbb{X}_{n}$, and $\mathcal{D}_2$ be the distribution of $n_1 \stat^P(\mathbb{X}_n)$ conditioned on $\mathbb{Z}_{n_2}$. Then for any $n_2>0$, $\lambda>0$, $\mathcal{D}_1$ and $\mathcal{D}_2$ converge to the same distribution as $n_1 \to \infty.$  \vspace{-2mm}
\end{theorem}

Using the critical level $q^\lambda_{1-\alpha}$, in the next result, we investigate the conditions on the separation boundary of the test based on $\stat^P(\mathbb{X}_n)$ so that the power is asymptotically one w.r.t.~the class of local alternatives, $\mathcal{P}$ defined as 
\begin{equation}
\PP:=\PP_{\theta,\Delta}:= \left\{P \in \mathcal{S}: \frac{dP}{dP_0}-1 \in \range (\ep_{P_0}^{\theta}), 
 \ \chi^2(P,P_0)=\norm{\frac{dP}{dP_0}-1}^2_{L^2(P_0)} \geq \Delta\right\}, \label{Eq:alternative-theta}
\end{equation}
for $\theta > 0$, where \vspace{-2mm}
$$\mathcal{S}:=\left\{P : \E_{P} [K_0(X,X)^r] \leq c r!\kappa^r, \, \forall r>2,\,\, \text{and}\,\, \  \E_{P} [K_0(X,X)] < \infty\right\},$$
for some $c,\kappa>0$. To this end, we impose the following assumption on the pair 
$(K_0,P_0)$.
\vspace{-2mm}
\begin{myassump}{$\mathbf{A_1}$}\label{assump:a2}
$\E_{P_0}[K_0(X,X)^v] \leq \tilde{\kappa}^v$ for some $v > 2,$ $\tilde{\kappa}>0.$ 
\end{myassump}

Define \vspace{-4mm} $$\Nol := \text{Tr}(\SgL\A_{P_0}\SgL)\,\,  \text{and} \,\,  \Ntl := \norm{\SgL\A_{P_0}\SgL}_{\hsS},$$which quantify the intrinsic dimensionality or degrees of freedom of $\h_{K_0}$. Notably, $\Nol$ plays a significant role in the analysis of kernel ridge regression, as discussed in works such as \citet{Caponnetto-07}.

Under Assumption \ref{assump:a2}, we provide general sufficient conditions on the separation boundary for any decay rate of the eigenvalues, as stated in Theorem \ref{thm:sep-bound-suff-conditions}. The following corollaries to Theorem \ref{thm:sep-bound-suff-conditions} offer explicit expressions for the separation boundary in the cases of polynomially and exponentially decaying eigenvalues of the operator $\A_{P_0}$. For simplified presentation, we will provide the results under the following assumption \ref{assump:a2'} that is stronger than \ref{assump:a2}, and the more general versions under \ref{assump:a2} are deferred to Section~\ref{sec:general-cor}. \vspace{-1mm}
\begin{myassump}{$\mathbf{A'_1}$}\label{assump:a2'}
$\E_{P_0}[K_0(X,X)^v] \leq \tilde{\kappa}^v$ for all $v > 2,$ and some $\tilde{\kappa}>0.$  \vspace{0mm}
\end{myassump}
\begin{remark} 
   The difference between the assumptions \ref{assump:a2'} and \ref{assump:a2} lies in assuming that the required condition holds for all $v>2$, rather than only for some $v$ as assumed in \ref{assump:a2}. Thus the result presented below is just the limiting result as $v$ goes to infinity of the general result in Corollaries \ref{coro:general:poly} and \ref{coro:general:exp}. 
\vspace{-3mm}
\end{remark}

\begin{corollary} \label{coro:poly}
Suppose $\lambda_i \lesssim i^{-\beta},$ $\beta>1$. Let $\lambda_n \asymp \Delta_n^{\frac{1}{2\tilde{\theta}}},$ with $\Delta_n \to 0$. Then for $(P_n)_n \subset \PP,$ we have  
$$\lim_{n \to \infty}P\{\stat^{P_n}(\mathbb{X}_n) \geq q^{\lambda}_{1-\alpha}\} = 1,$$
if one of the following holds:
\begin{enumerate}
   \item  \ref{assump:a2'} holds and $\Delta_n B_n^{-1} \to \infty,$ where 
\begin{align*}
&B_{n}  =
\left\{
	\begin{array}{ll}
	n^{\frac{-4\tilde{\theta} \beta}{4\tilde{\theta} \beta+1}},  &  \ \ \Tilde{\theta}> \frac{2}{3} \\
		 \left(\frac{\log n}{n}\right)^{\frac{4\tilde{\theta}\beta}{1+4\beta-2\tilde{\theta}\beta}}, & \ \  \frac{1}{4\beta}+\frac{1}{2}\leq \Tilde{\theta} \leq  \frac{2}{3}  \\
      n^{-\tilde{\theta}}, & \ \ \text{otherwise}
	\end{array}
\right., 
& &B_{n}  =
\left\{
	\begin{array}{ll}
	n^{\frac{-4\tilde{\theta} \beta}{4\tilde{\theta} \beta+1}},  &  \ \ \Tilde{\theta}> \frac{1}{2}+\frac{1}{4\beta} \\
		 \left(\frac{\log n}{n}\right)^{\frac{2\tilde{\theta}\beta}{\beta+1}}, & \ \  1-\frac{1}{2\beta}\leq \Tilde{\theta} \leq  \frac{1}{2}+\frac{1}{4\beta}  \\
      n^{-\tilde{\theta}}, & \ \ \text{otherwise}
	\end{array}
\right., 
\end{align*} 
for $\beta \geq \frac{3}{2}$ and $\beta < \frac{3}{2}$, respectively, with $\tilde{\theta}=\min\{\theta,\xi\}.$
\item $\sup_{i}\norm{\phi_i}_{\infty} < \infty$ and $\Delta_n B_n^{-1} \to \infty,$ where
$$B_{n}  =
\left\{
	\begin{array}{ll}
	n^{\frac{-4\tilde{\theta} \beta}{4\tilde{\theta} \beta+1}},  &  \ \ \tilde{\theta} \geq \frac{1}{2\beta} \\
	n^{\frac{-4\tilde{\theta} \beta}{2\tilde{\theta}\beta+\beta+1}}	, & \ \  \frac{1}{2}-\frac{1}{2\beta} \leq \tilde{\theta} \leq \frac{1}{2\beta} \\
    \left(\frac{\log n}{n}\right)^{2\tilde{\theta}}, & \ \ \text{otherwise}
	\end{array}
\right..$$
\end{enumerate}

\end{corollary}

\begin{corollary}\label{coro:exp}
   Suppose $\lambda_i \lesssim e^{-\tau i},$ $\tau>0$. Let $\lambda_n \asymp \Delta_n^{\frac{1}{2\tilde{\theta}}},$ with $\Delta_n \to 0$. Then for any $(P_n)_n \subset \PP,$ we have  
$$\lim_{n \to \infty}P\{\stat^{P_n}(\mathbb{X}_n) \geq q^{\lambda}_{1-\alpha}\} = 1,$$
if one of the following holds:
\begin{enumerate}
    \item  \ref{assump:a2'} holds and $\Delta_n B_n^{-1} \to \infty,$ where 
$$B_{n}  =
\left\{
	\begin{array}{ll}
	\frac{\sqrt{\log n}}{n},  &  \ \ \Tilde{\theta}> \frac{2}{3} \\
		\left(\frac{(\log n)^2}{n}\right)^{\frac{2\tilde{\theta}}{2-\tilde{\theta}}} , & \ \  \frac{1}{2} \leq \tilde{\theta} \leq \frac{2}{3} \\
      n^{-\tilde{\theta}}, & \ \ \text{otherwise}
	\end{array}
\right.,$$
with $\tilde{\theta}=\min\{\theta,\xi\}.$\vspace{1mm}
\item  $\sup_{i}\norm{\phi_i}_{\infty} < \infty$ and $\Delta_n B_n^{-1} \to \infty,$ where
$B_{n}  = \frac{\sqrt{\log n}}{n}.$
\end{enumerate} 
\end{corollary}


\begin{remark} \label{rem:bounded-kernel}
 (i) The uniform boundedness assumption $\sup_{i}\norm{\phi_i}_{\infty} < \infty$ implies that the kernel $K_0$ is bounded---follows from Mercer's theorem (see \citealt[Lemma 2.6]{Steinwart2012MercersTO})---, which in turn is stronger than \ref{assump:a2} since the latter is only a moment condition on $K_0$. Assuming boundedness of $K_0$ is too strong, since for example, if $K_0$ is chosen as in Example~\ref{example:ksd} with $K$ being a bounded translation-invariant kernel on $\mathbb{R}^d$, then it imposes strong assumptions on $P_0$, i.e., $p_0$ cannot decay faster than $e^{-\Vert x\Vert_1}$. On the other hand, Assumption \ref{assump:a2} imposes the condition $\int |\nabla \log p_0|^{2v} p_0(x)\,dx\le \tilde{\kappa}^v$ for some $v>2$ $\tilde{\kappa}>0$, which of course is satisfied by $p_0$ that decays faster than $e^{-\Vert x\Vert_1}$. However, we present the result with both \ref{assump:a2} and the uniform bounded assumptions to compare and contrast the separation rates (see Corollaries~\ref{coro:poly} and \ref{coro:exp} for details).
 \vspace{1mm}\\
 (ii) 
 Note that the assumption $\sup_{i}\norm{\phi_i}_{\infty} < \infty$ does  
 not hold in general. For example, it holds in any $d$ for Gaussian kernels on $\mathbb{R}^d$ while it does not hold for $d\ge 3$ for Gaussian kernel defined on $\mathbb{S}^{d-1}$. See (\citealp[Theorem 5]{unibound}) and (\citealp[Remark 6]{twosampletest}) for details. \vspace{-2mm}
\end{remark}

Next, we will show in Theorem~\ref{thm:minimax} that the above separation boundary is optimal for a wide range of $\theta$. We refer the reader to Section~\ref{subsec:minmaxframe} for the notation used in the following result. Before presenting the result, we introduce the following assumption.\vspace{-6mm}\\
\begin{myassump}{$\mathbf{A_2}$}\label{assump:a1}
The pair $(P_0,K_0)$  satisfies $\E_{P_0} [K_0(X,X)^r] \leq c r!\kappa^r, \, \forall r>2,$ for some $c,\kappa>0.$  
\end{myassump}
\vspace{-6mm}

\begin{theorem}[Minimax separation]\label{thm:minimax}
If $\lambda_i \asymp i^{-\beta}$, $\beta>1$, and $n^{\frac{4\theta\beta}{4\theta\beta+1}}\Delta_{n} \to 0,$ then \\  $\liminf_{n\rightarrow\infty}\inf_{\phi\in\Phi_\alpha}R_{\Delta_n}(\phi)>0,$ provided 
one of the following holds: (a) $\theta \geq 1$ under~\ref{assump:a1}, and $\E_{P_0}[K_0(x,Y)^2] < \infty;$ $\forall x \in \X$; (b)  $\theta \geq \frac{1}{4\beta}$ under  $\sup_{i}\norm{\phi_i}_{\infty} < \infty$. Suppose $\lambda_i \asymp e^{-\tau i}$, $\tau>0$, $\theta > 0$ and $\frac{n}{\sqrt{\log n}}\Delta_{n} \to 0.$ Then $\liminf_{n\rightarrow\infty}\inf_{\phi\in\Phi_\alpha}R_{\Delta_n}(\phi)>0.$
\end{theorem}

\begin{remark}
(i) Assuming $g_\lambda$ possesses infinite qualification (examples include Tikhonov and TikMax regularizers), denoted by $\xi=\infty$, then $\tilde{\theta}=\theta$. Under \ref{assump:a2'}, a comparison between Corollary~\ref{coro:poly} (or Corollary~\ref{coro:exp}) and Theorem~\ref{thm:minimax} indicates that the spectral regularized test, based on $\stat^P(\mathbb{X}_n)$, achieves minimax optimality w.r.t.~$\PP$ within the ranges of $\theta$ where the separation boundary of the test coincides with the minimax separation. For a more general scenario, we can align Corollary \ref{coro:general:poly} (or Corollary~\ref{coro:general:exp}) with Theorem~\ref{thm:minimax} under \ref{assump:a1} and observe that the test is again minimax optimal w.r.t.~$\mathcal{P}$.\vspace{1mm}\\
\noindent (ii) Note that although \ref{assump:a2} appears very similar to \ref{assump:a1}, it is weaker than \ref{assump:a1}. This means the behavior of the regularized KSD test captured in Theorem~\ref{thm:sep-bound-suff-conditions} is under the weaker condition of \ref{assump:a2} in contrast to Theorem~\ref{thm:minimax} which provides the minimax rates under \ref{assump:a1}. Also, observe that \ref{assump:a2'} is stronger than \ref{assump:a1}, which is stronger than \ref{assump:a2}.\vspace{1mm}\\
\noindent(iii) If we refrain from assuming a specific decay rate of the eigenvalues $\lambda_i$ and instead work with $\sum_{i}\lambda_i < \infty$ (which holds since $\ep_{P_0}$ is a trace class operator), then the general condition on the minimax separation boundary becomes $n^{4\theta/(4\theta+1)}\Delta_{n} \to 0$.\vspace{1mm}\\
\noindent(iv) We would like to highlight that when the kernel associated with the integral operator $\ep_{P_0}$ is assumed to be bounded, the minimax separation result for the space $\PP$ has been established in the literature (see \citealt{Krishna2021Gof} and \citealt{hagrass2023Gof}). However, as discussed in Remark \ref{rem:bounded-kernel}(i), this assumption is too restrictive to impose on the Stein kernel $K_0$. Therefore, the result in Theorem \ref{thm:minimax} extends this minimax separation result to the case of an unbounded kernel. \vspace{-2mm}
\end{remark}
The above results demonstrate that the regularized KSD test achieves the minimax separation for the range of $\theta$ that aligns with the minimax rates established in Theorem \ref{thm:minimax}. Finally, the next result shows that the separation boundary for the unregularized KSD test fails to achieve the minimax separation, demonstrating its non-optimality. To this end, consider the $U$-statistic estimator of  \eqref{eq:general KSD}, given in \eqref{Eq:test} as \vspace{-.5mm}\begin{align*}
S:=\hat{D}^2_{\mathrm{KSD}} &= \frac{1}{n(n-1)} \sum_{i \neq j} K_0(X_i,X_j), \vspace{-.5mm}
\end{align*} 
which can be used as a test statistic in a goodness-of-fit test \citep{liub16,schrab2023ksd}. Similar to the regularized KSD, we consider a testing threshold based on wild bootstrap. 
Let 
$$S_{\epsilon}:=\frac{1}{n(n-1)} \sum_{i\neq j} \epsilon_i \epsilon_j K_0(X_i,X_j),$$
where $\epsilon \in \{-1,1\}^n$ is a random vector of $n$ i.i.d Rademacher random variables. Then it can be shown that under $H_0$, the conditional distribution of $S_{\epsilon}$ given the samples $\mathbb{X}_n$,  converges to the same distribution as $S$  \citep[Theorem 3.1]{DEHLING1994392} if $\E_{P_0}K_0^2(X,Y)<\infty$---this condition is satisfied in our case since $\E_{P_0}K_0^2(X,Y) \leq [\E_{P_0}K_0(X,X)^2] < \infty.$ Hence, for the test threshold, we use  the $(1-\alpha)$ quantile of $S_{\epsilon}$ defined as 
\begin{equation*}
    q_{1-\alpha}:= \inf\left\{q \in \R: P_{\epsilon}\{S_{\epsilon} \leq q | \mathbb{X}_n\} \geq 1-\alpha \right\}. 
\end{equation*}

The following result shows that the separation boundary of the wild bootstrap test based on $\hat{D}^2_{\mathrm{KSD}}$ w.r.t.~the class of local alternatives, $\mathcal{P}$ cannot achieve a separation better than $n^{\frac{-2\theta}{2\theta+1}}.$

\vspace{-3mm}

\begin{theorem}[Separation boundary of wild bootstrap KSD test] \label{thm:non-optimality-KSD}
Suppose~\ref{assump:a0} holds and $n^{\frac{2\theta}{2\theta+1}}\Delta_n \to 0$ as $n \to \infty$. Then for any decay rate of $(\lambda_i)_i$, if~\ref{assump:a1} holds and $\E_{P_0}[K_0(x,Y)^2] < \infty;$ $\forall x \in \X$, we have
$$\liminf_{n \to \infty} \inf_{\PP}P_{H_1}\{\hat{D}^2_{\mathrm{KSD}}\geq q_{1-\alpha}\} < 1,$$
for any $\theta>1.$ Furthermore if $\sup_{i}\norm{\phi_i}_{\infty} < \infty$, then 
$$\liminf_{n \to \infty} \inf_{\PP}P_{H_1}\{\hat{D}^2_{\mathrm{KSD}}\geq q_{1-\alpha}\} < 1,$$
for any $\theta > 0.$
\end{theorem}
\begin{remark}
    (i) Theorem~\ref{thm:non-optimality-KSD} remains valid even if the testing threshold $q_{1-\alpha}$ is replaced with any threshold that converges in probability to the $(1-\alpha)$ quantile of the asymptotic distribution of $\hat{D}^2_{\mathrm{KSD}}$.\vspace{1mm} \\  
  \noindent  (ii) The combination of the above result with the minimax rates in Theorem \ref{thm:minimax} and the attained separation boundary of our proposed test (see Section~\ref{sec:optimal}), as demonstrated in Corollaries \ref{coro:poly} and \ref{coro:exp}, reveals the non-optimality of the test based on $\hat{D}_{\mathrm{KSD}}$. This is evident since $\inf_{\beta>1} 4\theta\beta/(4\theta\beta+1)=4\theta/(4\theta+1)>2\theta/(2\theta+1)$ and $1>2\theta/(2\theta+1)$ for any $\theta>0$.\vspace{-2mm}
\end{remark}

\subsection{Adaptation to $\lambda$ by aggregation} \label{subsec:adaptation}
Theorem~\ref{thm:sep-bound-suff-conditions} and Corollaries \ref{coro:poly}, \ref{coro:exp} demonstrate that the optimal $\lambda$ needed to achieve the minimax separation boundary depends on the unknown parameters, $\theta$ (and $\beta$ in the case of polynomial decay). To make the test practical, the idea is to devise an adaptive test by aggregating multiple 
tests constructed for various $\lambda$ taking values in a finite set, $\Lambda$. Using similar analysis as in \cite{hagrass2023Gof, twosampletest,Krishna2021Gof} and \cite{MMDagg}, the aggregate test can be shown to yield the optimal separation boundary as shown in Corollaries~\ref{coro:poly} and \ref{coro:exp}, but up to a $\log\log$ factor. In the following, we briefly introduce the aggregate test but skip stating the corresponding theoretical results since they are exactly the same as Corollaries~\ref{coro:poly}
and \ref{coro:exp} but with an additional $\log\log$ factor. 

Define $\Lambda := \{\lambda_L, 2\lambda_L, \ldots, \lambda_U\}$, where $\lambda_U = 2^b\lambda_L$ for $b \in \mathbb{N}$, such that $|\Lambda|=b+1=1+\log_2\lambda_U/\lambda_L$, with $|\Lambda|$ denoting the cardinality of $\Lambda$. Let $\lambda^*$ be the optimal $\lambda$ for achieving minimax optimality. The crux lies in selecting $\lambda_L$ and $\lambda_U$ so that an element in $\Lambda$ closely approximates $\lambda^*$ for any $\theta$ (and $\beta$ in the case of polynomially decaying eigenvalues). Define $v^* := \sup\{x \in \Lambda: x \leq \lambda^*\}$. Then, it is straightforward to verify that $v^* \asymp \lambda^*$, meaning $v^*$ also represents an optimal choice for $\lambda$ within $\Lambda$, as for $\lambda_L\leq \lambda^* \leq \lambda_U$, we have $\lambda^*/2\leq v^* \leq \lambda^*$. Motivated by this observation, we construct an adaptive test based on the union of corresponding tests across $\lambda \in \Lambda$, rejecting $H_0$ if any tests do so. Furthermore, we can also adapt across different choices of $K$, or the kernel bandwidth (or kernel parameter) of a given $K$, e.g., $K(x,y)=e^{-\Vert x-y\Vert^2_2/h}$, $h>0$. Thus, the resulting test rejects $H_0$ if $\stat^{P,h}(\mathbb{X}_n) \geq q_{1-\alpha/\cd|W|}^{\lambda,h}$ for any $(\lambda,h) \in \Lambda \times W$, where $\stat^{P,h}(\mathbb{X}_n)$ is the test statistic as defined in \eqref{eq:test-stat}.

We would like to mention that recently, \cite{schrab2023ksd} proposed an adaptive version of the KSD test of \cite{chwialkowski16} and \cite{liub16}---see Example~\ref{example:ksd}---by adapting to kernel parameter of $K$ (here, $K$ is chosen to be the Gaussian kernel with bandwidth $h$ as shown above) by aggregating over many KSD tests constructed for a given bandwidth. They obtained the separation boundary for the aggregate test over restricted Sobolev balls. However, they did not investigate the minimax separation rates for such alternative spaces, leaving open the question of the minimax optimality of their rates.

\vspace{-2mm}

\section{Interpreting the class of alternatives, $\mathcal{P}$}\label{sec:range}

As aforementioned, the alternative class, $\mathcal{P}$, involves a smoothness condition defined by $\range (\ep_{P_0}^{\theta})$, where $\ep_{P_0}$ is based on $K_0$ and $P_0$. Since $\mathcal{P}$ is abstract, in this section, we provide an interpretation for $\mathcal{P}$ by analyzing the smoothness condition through some concrete examples. Moreover, a pertinent question to understand is how the range space assumption associated with the Stein RKHS $\h_{K_0}$, corresponding to the Stein kernel $K_0$, compares to that of the base RKHS $\h_K$ corresponding to the base kernel $K$. To clarify, depending on the domain of interest, the Stein kernel will take a form that depends on the base kernel $K$ and the measure $P_0$, as illustrated in Examples \ref{example:ksd}--\ref{example:hilbert}. Recall that $\ep_{P_0} : \Lp \to \Lp$, $ f \mapsto \int K_0(\cdot,x)f(x)\,dP_0(x)$.
As previously mentioned, $\text{Ran}(\ep^\theta_{P_0})$ can generally be interpreted as an interpolation space obtained through the real interpolation of $\h_{K_0}$ (which corresponds to $\theta=1/2$) and $L^2(P_0)$ (which corresponds to $\theta=0$) at the scale $\theta$ \citep[Theorem 4.6]{Steinwart2012MercersTO}. Now, let $\T_{P_0} : \Lp \to \Lp$, defined by $$\T_{P_0} f = \int \bar{K}(\cdot,x)f(x)\,dP_0(x),$$ be the integral operator associated with the centered kernel $\bar{K}(\cdot,x) := K(\cdot,x) - \E_{P_0}K(\cdot,x)$. A natural question that follows is how $\text{Ran}(\T^\theta_{P_0})$, an interpolation space between the RKHS $\h_{\bar{K}}$ associated with $\bar{K}$ (corresponding to $\theta=1/2$) and $L^2(P_0)$, compares to $\text{Ran}(\ep^\theta_{P_0})$. We address this question in Examples \ref{Ex: uniform} and \ref{Ex: Gaussian} by providing explicit forms for those two spaces. \vspace{-2mm}

\begin{remark}
The centered kernel $\bar{K}(\cdot,x) := K(\cdot,x) - \E_{P_0}K(\cdot,x)$ is used in the definition of $\T_{P_0}$ to ensure that the functions in $\text{Ran}(\T^\theta_{P_0})$ have zero mean with respect to $P_0$. This is necessary because the range space imposes a smoothness condition on $u = dP/dP_0 - 1$, which always satisfies $\E_{P_0}u = 0$. \vspace{-2mm}
\end{remark}

First, we present a generic result which will be used in all the examples.  
\begin{proposition} \label{thm;range}
Suppose $L(x,y) = \sum_{k \in I}\lambda_k \gamma_{k}(x)\gamma_{k}(y),$ 
where $1 \cup (\gamma_k)_{k \in I}$ represents orthonormal basis with respect to $\Lp$. Then 
 $$\emph{Ran}(\mathcal{D}_{P_0}^{\theta})=\left\{ \sum_{k \in I} a_k \gamma_k(x) : \sum_{k \in I} \frac{a_k^2}{\lambda_k^{2\theta}} < \infty \right\},$$
where $\mathcal{D}_{P_0} : \Lp \to \Lp$, $f \mapsto \int L(\cdot,x)f(x)\,dP_0(x).$
\end{proposition}

Based on the above results, we now discuss a few examples. 

\begin{example}[Uniform distribution on {$[0,1]$}] \label{Ex: uniform}
 Let $P_0$ be the uniform distribution defined on $[0,1],$ and 
 \begin{equation}
 K(x,y)= a_0+ \sum_{k \neq 0} |k|^{-\beta} e^{\sqrt{-1}2\pi kx}e^{-\sqrt{-1}2\pi ky},\,\,a_0 \geq 0,\,\,\beta>2.\label{eq:kernel-fourier}    
 \end{equation}
Then 
$$R_1:=\text{Ran}(\T_{P_0}^{\theta})=\left\{ \sum_{k \neq 0} a_k e^{\sqrt{-1}2 \pi k x } : \sum_{k \neq 0} a_k^2 k^{2\theta \beta} < \infty \right\}$$
and 
$$R_2:=\text{Ran}(\ep_{P_0}^{\theta})=\left\{ \sum_{k \neq 0} a_k e^{\sqrt{-1}2 \pi k x } : \sum_{k \neq 0} a_k^2 k^{2\theta (\beta-2)} < \infty \right\},$$ where $K_0(x,y)=\nabla_x\nabla_y K(x,y)$. 
See Section~\ref{proof:ex-uniform} for the derivation of $R_1$ and $R_2.$
\begin{remark}
   We would like to highlight that since the null distribution $P_0$ is defined on $[0,1] \subset \mathbb{R}$ rather than on the entire $\mathbb{R}$, extra care is needed in defining the Stein operator (because of the lack of differentiability of $p_0$) while ensuring $\mathbb{E}_{P_0}[K_0(\cdot, X)] = 0$. 
   For the base kernel defined in \eqref{eq:kernel-fourier}, it is easy to verify that $K$ is a periodic function on $\mathbb{R}$ that is at least twice differentiable, and $\E_{P_0}[K_0(\cdot, X)] = 0$ which follows from the orthogonality of the Fourier basis. However, note that the Stein kernel $K_0(x, y) = \langle \psi_{P_0}(\cdot, x), \psi_{P_0}(\cdot, y) \rangle_{\mathcal{H}_{K_0}}$,
where $\psi_{P_0}(\cdot, x) = K(\cdot, x) \nabla_x \log p_0(x) + \nabla_x K(\cdot, x),$ is not rigorously defined because of the non-differentiability of $p_0$. But by avoiding rigor and employing a formal calculation, we can write $K_0(x,y)=\langle \nabla_x K(\cdot,x),\nabla_y K(\cdot,y)\rangle_{\mathcal{H}_{K_0}}=\nabla_x\nabla_y K(x,y)$ by treating $\nabla_x \log p_0(x) = 0$ on $[0, 1]$. 
We refer the reader to \citet[Section 3.1]{NIPS2015_698d51a1} for technical details related to defining the Stein operator on $\mathcal{X}\subset\mathbb{R}^d$.
\end{remark}
This means that the functions in $R_2$ are 
$2\theta$ less smooth than those in $R_1$ since the Fourier coefficients in $R_2$ decay at a rate that is $2\theta$ orders slower than those of $R_1$. Moreover, since 
$ \sum_{k \neq 0} a_k^2 k^{2\theta (\beta-2)} \leq \sum_{k \neq 0} a_k^2 k^{2\theta \beta},$ for $\beta >2$, we have $R_1 \subseteq R_2.$ Observe that the eigenfunctions satisfy the uniform boundedness condition as $\sup_{k}|e^{\sqrt{-1}2\pi k x}| = 1$. For $\beta > 3$, it can be demonstrated that both the base kernel $K$ and the Stein kernel $K_0$ are bounded. Consequently, the entire alternative space $\mathcal{P}_{\Delta}$ associated with the Stein kernel (i.e., the one defined through the condition $u \in R_2$) encompasses the alternative space associated with the base kernel (i.e., the one defined through the condition $u \in R_1$). Moreover, the minimax separation boundary linked with the alternative space defined through the base kernel follows an order of $n^{-4\theta\beta/(4\theta\beta+1)}$. Meanwhile, the larger alternative space associated with the Stein kernel has a minimax separation boundary order of $n^{-4\theta\tilde{\beta}/(4\theta\tilde{\beta}+1)}$, which is achieved by our proposed regularized KSD test, where $\tilde{\beta}=\beta-2$. This adjustment arises from the corresponding eigenvalue decay rate of the Stein kernel, which holds an order of $|k|^{-(\beta-2)}$. An example of a kernel that follows the form in \eqref{eq:kernel-fourier} is the periodic spline kernel, represented as $K(x,y) = (-1)^{r -1}(2r) !^{-1} B_{2r}([x-y])$, where $B_r$ denotes the Bernoulli polynomial, which is generated by the generating function $te^{tx}/(e^t-1) = \sum_{r=0}^{\infty} B_r(x) t^r/r!$, and $[t]$ denotes the fractional part of $t$. Using the formula $B_{2r}(x) = (-1)^{r-1} (2r)!(2\pi)^{-2r} \sum_{k \neq 0}^{\infty} |k|^{-2r} e^{\sqrt{-1} 2\pi kx}$, it can be demonstrated that $K(x,y) = \sum_{k \neq 0} (2\pi |k|)^{-2r} e^{\sqrt{-1} 2\pi kx} e^{-\sqrt{-1} 2\pi ky}$ (see \citealp[page 21]{Wahba1990} for details).
\end{example}
\begin{remark} \label{rem:unifrom}
    (i) The $s$-order Sobolev space defined on $[0,1]$ is given by
    $$\W^{s,2}:=\left\{f(x)=\sum_{k\in\mathbb{Z}} a_k e^{i2\pi kx},\,x\in[0,1]: \sum_{k\in\mathbb{Z}}(1+k^2)^sa^2_k<\infty\right\}.$$
    Since
    $\sum_{k\neq 0} k^{2\theta \beta}a_k^2 \leq \sum_{k \in\mathbb{Z}} (1+k^2)^{\theta \beta}a_k^2$, it follows that $\W^{s,2}\subset R_1$. This means, if $u:= dP/dP_0-1 \in \W^{s,2}$, then $u \in R_1$ with $\theta=s/\beta.$

    (ii) Note that for $d\geq 2$, we can generalize the above example by using a product of kernels over each dimension, which yields $$R_1 := \left\{\sum_{k_1,\dots, k_d} a_{k_1,\dots,k_d} \gamma_{k_1,\dots,k_d} : \sum_{k_1,\dots, k_d}a_{k_1,\dots,k_d}^2 (k_1 k_2 \dots k_d)^{2\theta\beta} < \infty\right\},$$ \vspace{-2mm} and \vspace{-2mm} 

     $$R_2 := \left\{\sum_{k_1,\dots, k_d} a_{k_1,\dots,k_d} \gamma_{k_1,\dots,k_d} : \sum_{k_1,\dots, k_d}a_{k_1,\dots,k_d}^2 (k_1^2+ \dots+k_d^2)^{-2\theta} (k_1 k_2 \dots k_d)^{2\theta\beta} < \infty\right\},$$
    where $\gamma_{k_1,\dots,k_d}(x) = \prod_{j=1}^d e^{\sqrt{-1} 2 \pi k_j x_j}$, $x=(x_1,\ldots,x_d)\in\mathbb{R}^d$ and for $\beta > \frac{2}{d},$ $R_1\subset R_2$. 

        Moreover observe that $$\sum_{k_1,\dots, k_d}a_{k_1,\dots,k_d}^2(k_1 k_2 \dots k_d)^{2\theta\beta} \leq \sum_{k_1,\dots, k_d}a_{k_1,\dots,k_d}^2(1+k_1^2+\dots+k_d^2)^{d\theta\beta},$$  
        which yields that $u \in \W^{s,2}$ implies $u \in R_1$ for $\theta = s/(d \beta).$

\end{remark}

\begin{example}[Gaussian distribution with Mehler kernel on $\mathbb{R}$] \label{Ex: Gaussian}
  Let $P_0$ be a standard Gaussian distribution on $\mathbb{R}$ and $K$ be the Mehler kernel, i.e., $$K(x,y):=\frac{1}{\sqrt{1-\rho^2}}\exp\left(-\frac{\rho^2(x^2+y^2)-2 \rho xy}{2(1-\rho^2)}\right),$$ for $0<\rho<1$. Then \vspace{-4mm}
 $$R_1:=\text{Ran} (\T_{P_0}^{\theta})=\left\{\sum_{k=1}^{\infty}a_{k}\gamma_k(x) : \sum_{k=1}^{\infty}a_k^2e^{-2k\theta\log \rho} < \infty \right\},$$
    where 
    $\gamma_k(x) = \frac{H_k(x)}{\sqrt{k!}} ,$ and $H_k(x) = (-1)^ke^{x^2/2}\frac{d^k}{dx^k}e^{-x^2/2}$ is the Hermite polynomial. Furthermore, 
$$R_2 := \text{Ran} (\ep_{P_0}^{\theta})=\left\{\sum_{k=1}^{\infty}a_{k}\gamma_k(x) : \sum_{k=1}^{\infty}a_k^2k^{-2\theta}e^{-2k\theta\log \rho}< \infty\right\}.$$ 
Since $\sum_{k=1}^{\infty}a_k^2k^{-2\theta}e^{-2k\theta\log \rho} \leq \sum_{k=1}^{\infty}a_k^2e^{-2k\theta\log \rho},$ we have $R_1 \subseteq R_2.$ See Section~\ref{proof:ex-gaussian} for details.
\end{example}

\begin{remark}
(i) Let $\tilde{K}(x,y):=\phi(x-y):=\exp(-\sigma(x-y)^2),$ for $\sigma < \frac{\rho}{4(1-\rho)}$.  Then the corresponding RKHS $\mathscr{H}_{\tilde{K}}$ is the Gaussian RKHS, defined by $$\mathscr{H}_{\tilde{K}}=\left\{f\in L^2(\mathbb{R}):\int_{\mathbb{R}} |\hat{f}(\omega)|^2 e^{\frac{\omega^2}{4\sigma}}\,d\omega<\infty\right\},$$
where $\hat{f}$ is the Fourier transform of $f\in L^2(\R)$.
It can be shown that $\h_{\tilde{K}}\subset R_1$ for any $\theta\le 1/2$ (see Section~\ref{proof: remark} for details). Therefore, if $u:= dP/dP_0-1 \in \h_{\tilde{K}},$ then  $ u \in R_1\subset R_2,$ for any $\theta \leq 1/2.$ 
\vspace{1mm}\\
(ii) For $d \geq 2$, the above example can be generalized by using the product of kernels over each dimension. This yields that 
$$R_1 := \left\{\sum_{k_1,\dots, k_d} a_{k_1,\dots,k_d} \gamma_{k_1,\dots,k_d} : \sum_{k_1,\dots, k_d}a_{k_1,\dots,k_d}^2 e^{-2(k_1+\dots+k_d)\theta \log \rho} < \infty\right\},$$ \vspace{-2mm} and  \vspace{-2mm}

     $$R_2 := \left\{\sum_{k_1,\dots, k_d} a_{k_1,\dots,k_d} \gamma_{k_1,\dots,k_d} : \sum_{k_1,\dots, k_d}a_{k_1,\dots,k_d}^2 (k_1+ \dots+k_d)^{-2\theta} e^{-2(k_1+\dots+k_d)\theta \log \rho} < \infty\right\},$$  
    where $\gamma_{k_1,\dots,k_d}(x) = \prod_{j=1}^d\gamma_{k_j}(x_j),\,x=(x_1,\ldots,x_d)\in\mathbb{R}^d$. 
\end{remark}

\begin{example}[Uniform distribution on $\mathbb{S}^2$]\label{range:sphere}
    Let $P_0$ be a uniform distribution on $\X=\mathbb{S}^{2},$ where $\mathbb{S}^2$ denotes the unit sphere. Let 
    \begin{equation}
    K_0(x,y):= \sum_{k=1}^{\infty}\sum_{j=-k}^{k}\lambda_k Y^j_k(\theta_x,\phi_x)Y^j_k(\theta_y,\phi_y), \label{eq:kernel}    
    \end{equation}
     where 
    $x =(\sin \theta_x \cos \phi_x,\sin \theta_x \sin\phi_x,\cos \theta_x)$, $y =(\sin \theta_y \cos \phi_y,\sin \theta_y \sin\phi_y,\cos \theta_y)$ with \\ $0<\theta_x,\theta_y<\pi$, $0<\phi_x,\phi_y<2\pi$, and $$Y^j_{k}(\theta,\phi):= \sqrt{\frac{(2k+1)(k-j)!}{4\pi (k+j)!}}p^j_k(\cos \theta)e^{\sqrt{-1}j\phi},$$ with $p^j_k(x) := (-1)^j (1-x^2)^{\frac{j}{2}} \frac{d^j p_k(x)}{dx^j}$ and $p_k(x):= \frac{1}{k!2^k}\frac{d^k(x^2-1)^k}{dx^k}$. 
    If $\sum_{k=1}^{\infty}(2k+1) \lambda_k < \infty,$ then     
    $$\text{Ran} (\ep_{P_0}^{\theta}):=\left\{\sum_{k=1}^{\infty}\sum_{j=-k}^{k}a_{kj}Y_k^j(\theta_x,\phi_x) : \sum_{k=1}^{\infty}\sum_{j=-k}^{k} a_{kj}^2 \lambda_k^{-2\theta}< \infty\right\}.$$
\end{example}

\begin{remark}
(i) In Example~\ref{range:sphere}, we directly assume a form for the Stein kernel $K_0$, in contrast to Examples~\ref{Ex: uniform} and \ref{Ex: Gaussian}, where we first express a base kernel $K$, and then compute the corresponding Stein kernel $K_0$. Therefore, the claim in Example~\ref{range:sphere} follows directly from Proposition~\ref{thm;range}. On the other hand, suppose we choose  \eqref{eq:kernel} as the base kernel and then apply the appropriate Stein transformation as outlined in Example~\ref{example:manifold}, the transformed eigenfunctions lose orthogonality to $P_0$, making it challenging to derive an explicit expression for $\range(\ep_{P_0}^{\theta})$ using Proposition \ref{thm;range}. Conversely, in Examples~\ref{Ex: uniform} and \ref{Ex: Gaussian}, the transformation only alters the eigenvalues while preserving the eigenfunctions of the base kernel. 

\noindent (ii) For more details about kernels that take the form in \eqref{eq:kernel}, we refer the reader to \citet{unibound}. For example, \citet[Theorem 2 and 3]{unibound} provides explicit expressions for the eigenvalues corresponding to Gaussian and polynomial kernels on the sphere. \vspace{-3mm}
\end{remark}

The above examples have focused on the smoothness condition when $\mathcal{X}$ is continuous. Suppose now that $\mathcal{X}$ is a finite set with cardinality $|\X|$, on which $P_0$ and $K_0$ are supported---refer, for example, to \cite{discretedata} for a definition of the Stein kernel $K_0$ for discrete distributions. Let $\X=\{a_1,\ldots,a_{|\X|}\}$and define $P_0=\sum^{|\X|}_{i=1} w_i\delta_{a_i}$ where $w_i>0$ for all $i$. Then it is easy to verify that $L^2(P_0)$ is isomorphic to $\mathbb{R}^{|\X|}$ since $\Vert f\Vert^2_{L^2(P_0)}=\sum^{|\X|}_{i=1} w_i f^2(a_i)$. Consider $(\ep_{P_0}f)(a_j)=\sum^{|\X|}_{i=1} w_i K_0(a_j,a_i) f(a_i)$, which implies that $\ep_{P_0}$ can be identified with a $|\X|\times|\X|$ matrix. Then $u\in \text{Ran}(\ep_{P_0})$, implies that there exists $h\in\mathbb{R}^{|\X|}$ such that $u=\ep_{P_0}h=\mathbf{K}_0 \text{diag}(w_1,\ldots,w_{|\X|}) h$, where $[\mathbf{K}_0]_{ij}=K_0(a_i,a_j),\,i,j\in[|\X|]$. Therefore, if $\mathbf{K}_0$ is invertible, then it is clear that there always exists $h\in\mathbb{R}^{|\X|}$ such that $u=\ep_{P_0}h$. Hence, the range space condition is always satisfied since $\text{Ran}(\ep^\theta_{P_0})=\text{Ran}(\ep_{P_0})=\mathbb{R}^{|\X|}$ for all $\theta>0$. This means the alternative space is solely defined by the $\chi^2$ separation condition. A more interesting scenario is when $\X$ is countable, in which case $L^2(P_0)$ is isomorphic to the sequence space, $\ell^2(\X)$. We leave a detailed study of interpreting the range space condition in this setting for future work.

\section{Experiments}\label{sec:experiments}

In this section, we conduct an empirical evaluation and comparison of various goodness-of-fit tests: \vspace{-1mm}

\begin{itemize}
    \item \emph{KSD(Tikhonov)}: An adaptive spectral regularized KSD test with a Tikhonov regularizer.\vspace{.5mm}
    \item \emph{KSD(TikMax)}: An adaptive spectral regularized KSD test with a TikMax regularizer.\vspace{.5mm}
    \item \emph{KSD(Tikhonov)$^*$}: An adaptive spectral regularized KSD test with a Tikhonov regularizer based on an estimator of the regularized discrepancy in \eqref{eq:reg-discrp}. The test statistic is similar to that of \eqref{eq:test-stat} but with $\hat{\A}_P$ being replaced by $\hat{\A}_{P_0}$, where $\hat{\A}_{P_0}$ is computed based on extra samples from $P_0$.\vspace{.5mm}
    \item \emph{KSD(TikMax)$^*$}: An adaptive spectral regularized KSD test constructed similar to that of \\
    \emph{KSD(Tikhonov)$^*$} but with a TikMax regularizer.\vspace{.5mm}  
   \item \emph{KSD}: An unregularized KSD test using the median heuristic bandwidth for the kernel parameter. 
 \vspace{.5mm}
    \item \emph{MMD(Tikhonov)}: An adaptive spectral regularized MMD test (as implemented in \citealt{hagrass2023Gof}), which uses additional samples from $P_0$ to compute the test statistic.
\end{itemize}
\begin{remark}
(i) We examined the performance of the above-mentioned tests across various domains, including $\mathbb{R}^d$, $\mathbb{S}^{d-1}$, and infinite-dimensional Hilbert space. Depending on the domain $\mathcal{X}$, we apply the appropriate Stein operator to compute the kernel Stein discrepancy. Specifically, in Section \ref{RBM}, KSD refers to the variant proposed by \citet{chwialkowski16} and \citet{liub16} (see Example 2 for details). In Section \ref{spherical}, KSD denotes the directional Stein discrepancy (dKSD) for directional data, as introduced by \citet{xu2020directional} (see Example 3 for details) and in Section \ref{infdim}, KSD refers to the infinite-dimensional version described in \citet{InfiniteDimKSD} (see Example 4 for details).\vspace{1mm}\\
(ii) In Section \ref{RBM}, we also compare with KSDAgg test \citep{schrab2023ksd}, an adaptive version of the KSD test of \cite{chwialkowski16} and \cite{liub16}---see Example~\ref{example:ksd}---by aggregating tests over the kernel bandwidth $h \in W$ (see below for details). \vspace{1mm}\\
(iii) \emph{KSD(Tikhonov)$^*$} and \emph{KSD(TikMax)$^*$} are being considered to compare the performance of the regularized test when using $\hat{\A}_{P}$ versus $\hat{\A}_{P_0}$ as the estimated covariance operator. Recall that in our proposed test, we opted for $\hat{\A}_{P}$ over $\hat{\A}_{P_0}$ to create a practical test that circumvents the requirement for additional samples from $P_0$. 
\end{remark}
All these tests are conducted with both Gaussian and inverse multiquadric (IMQ) kernels defined as $K(x,y)=\exp\left(-\|x-y\|^2_2/(2h)\right)$ and $K(x,y)=\left(1+\|x-y\|^2_2/h\right)^{-0.5}$, respectively. For our proposed KSD adaptive tests, we aggregate tests jointly over $\lambda \in \Lambda$ and $h \in W$ (refer to Section \ref{subsec:adaptation} for details). Unless otherwise specified, all experiments utilize $\Lambda :=\{10^{-4}, 10^{-2.5}, 10^{-1},10^{0.5}\}$ and $W:=\{10^{-2}, 10^{-1.5}, 10^{-1}, 10^{-0.5}, 1,  10^{0.5}\} \times h_m$, where $h_m$ represents the kernel bandwidth determined using the median heuristic defined by $h_m:= \text{median}\{\|q-q'\|_2^2: q,q' \in X \cup Y\}$. Each test is repeated 500 times, and the average power is reported. For all KSD tests, we employed the i.i.d. weighted bootstrap to compute the test threshold with the number of bootstrap samples being $B=1000$. On the other hand, for the MMD(Tikhonov) test, we used the same values for $\Lambda$ and $W$ as mentioned above, along with 150 permutations to compute the test threshold (note that the MMD test is a permutation test).  


The main takeaway from the experiments could be summarized as follows:
\begin{itemize}
    \item Overall, the proposed regularized KSD tests, consistently outperform the standard unregularized KSD test.\vspace{.5mm}
    \item Generally, when independent samples from $P_0$ are readily available, MMD(Tikhonov) demonstrates the best performance, albeit at the highest computational cost. Even in this setting, the regularized KSD tests are competitive and achieve performance matching that of MMD(Tikhonov) in some cases. \vspace{.5mm}
    \item The proposed regularized versions of KSD with respect to $\hat{\A}_P$ and $\hat{\A}_{P_0}$ (KSD$^*$), exhibit similar performance in most cases, with $\mathrm{KSD}^*$ slightly outperforming the regularized KSD. However, note that $\mathrm{KSD}^*$ has the benefit of using additional samples from $P_0.$ \vspace{-3mm}
\end{itemize}



\begin{figure}[t]
\centering
\includegraphics[scale=0.35]{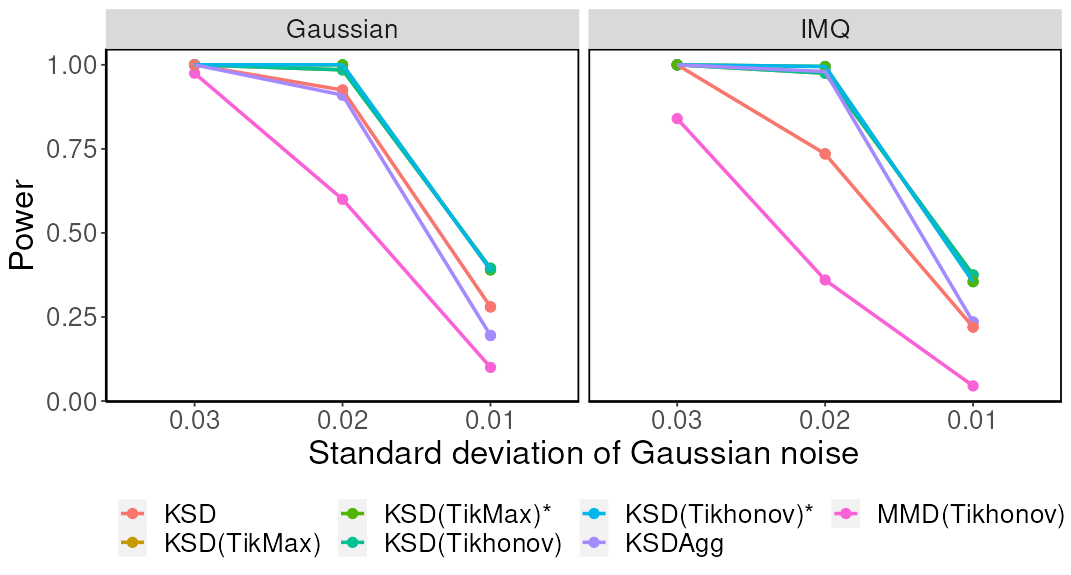}
\caption{Power of the tests for Gaussian-Bernoulli restricted Boltzmann machine with $d=50$, $n_2=100$ and $n=1000$.} \label{fig:RBM}
\vspace{-2mm}
\end{figure}
\subsection{Gaussian-Bernoulli restricted Boltzmann machine} \label{RBM}
As considered in \citet{liub16} and \citet{schrab2023ksd}, we investigate a Gaussian-Bernoulli restricted Boltzmann machine, a graphical model characterized by a binary hidden variable $h \in\{-1,1\}^{d_h}$ and a continuous observable variable $x \in \mathbb{R}^d$. These variables share a joint density function defined as
$
f(x, h)=Z^{-1}\exp \left( x^{\top} B h/2+b^{\top} x+c^{\top} h-\|x\|_2^2/2\right),
$
where $Z$ represents an unknown normalization constant. The target density is obtained as 
$\frac{dP_0}{dx}=p_0(x)=\sum_{h \in\{-1,1\}^{d_h}} f(x, h).$
We adopt the same settings and implementation as described in \citet{schrab2023ksd}, where $d=50$, and $d_h=40.$ Note that $p_0$ is known only up to a normalization constant and the score function admits a closed-form expression:
$$
\nabla \log p_0(x)=b-x+B \frac{\exp \left(2\left(B^{\top} x+c\right)\right)-1}{\exp \left(2\left(B^{\top} x+c\right)\right)+1} .
$$
Here $b$ and $c$ are drawn from a standard Gaussian, and the elements in $B$ are selected randomly from $\{-1,1\}.$ The alternate distribution $P$ is constructed similarly to $P_0$ but with each element of $B$ being perturbed by Gaussian noise, $N(0,\sigma)$ with $\sigma \in \{0.01, 0.02, 0.03\}$. Differentiating $P$ from $P_0$ becomes more challenging as $\sigma$ becomes small. We draw $n=1000$ samples from $P$ using a Gibbs sampler, using which all aforementioned KSD tests (except KSD$^*$ and MMD) are constructed with $n_2=100$ (see also Figure~\ref{fig:RBM_n2} for the effect of changing $n_2$). To construct KSD$^*$ and MMD tests, we also sample 100 and 1000 samples, respectively, from $P_0$ using a Gibbs sampler.  
Figure \ref{fig:RBM} illustrates the power of the tests considered. We note that KSD(Tikhonov)$^*$ and KSD(TikMax)$^*$ yield matching results and perform the best while being closely followed by our proposed tests KSD(Tikhonov) and KSD(TikMax), which strictly dominate the KSD test. KSDAgg test performs comparably to regularized KSD tests when using the IMQ kernel. However, its performance declines with the Gaussian kernel, highlighting the advantage of including the regularization parameter $\lambda$ rather than relying solely on aggregation over the kernel bandwidth. Surprisingly, the MMD(Tikhonov) exhibits the worst performance in this case, despite having access to extra samples from $P_0$. This unexpected outcome may be attributed to the nature of the experiment, where the samples from $P_0$ are obtained using a Gibbs sampler, which might only yield approximate samples or samples that are not independent, while the MMD(Tikhonov) requires extra independent samples from $P_0$. 
\begin{figure}[t]
\centering
\includegraphics[scale=0.35]{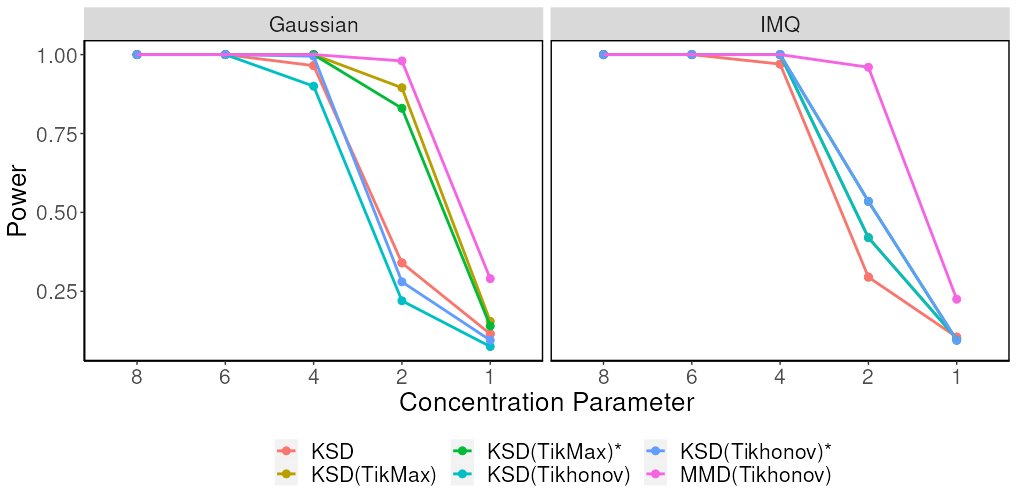}
\caption{Power for mixture of Watson distributions for different concentration parameter $k$  using $n_2=100$ and $n=500$.} \label{fig:watson}
\vspace{-2mm}
\end{figure}
\subsection{Directional data}\label{spherical}
The density for Watson distribution on $\mathbb{S}^{d-1}$ is given by $$f(x,\mu,k)=\frac{\Gamma(d/2)}{2\pi^{d/2}M(1/2,d/2,k)}\exp(k(\mu^Tx)^2),$$ where $x\in \mathbb{S}^{d-1}$, $k\geq 0$ is the concentration parameter, $\mu$ is the mean parameter, and $M$ denotes the Kummer's confluent hypergeometric function. In this experiment, we consider a mixture of two multivariate Watson distributions, representing axially symmetric distributions on a sphere. Note that $k=0$ corresponds to a uniform distribution on $\mathbb{S}^{d-1}$, which is chosen to be $P_0$ for this experiment. The alternate distributions are chosen to be a mixture of two Watson distributions with similar concentration parameters $k$ and mean parameters $\mu_1 ,\mu_2$ respectively, where $\mu_1=(1, \dots ,1) \in \mathbb{R}^d$ and $\mu_2=(-1,1 \dots ,1) \in \mathbb{R}^d$ with the first coordinate changed to $-1$, where changing $k$ creates a sequence of alternatives. The smaller the $k$ is, the more difficult it is to distinguish between $P_0$ and $P$.

We generate $n=500$ samples from $P$ (for different values of $k$) and construct all tests (except KSD$^*$ and MMD) based on them with $n_2=100$ (see also Figure~\ref{fig:watson_n2} for the effect of changing $n_2$). We use additionally generated samples of size 100 and 500 from $P_0$ to construct KSD$^*$ and MMD tests, respectively. 
Figure~\ref{fig:watson} presents the outcomes of testing against a spherical uniform distribution across various concentration parameters. With both Gaussian and IMQ kernels, we note that MMD(Tikhonov) demonstrates the most optimal performance, followed by KSD$^*$ versions, 
which can be attributed to its access to extra independent samples from the null distribution. Our proposed tests KSD(Tikhonov) and KSD(TikMax) behave similarly and better than unregularized KSD with IMQ while KSD(Tikhonov) has similar performance to that of KSD with the Gaussian kernel. Also, KSD(TikMax) outperforms both the unregularized KSD test and the KSD(Tikhonov) test. Interestingly, using the TikMax regularizer generally provides better results than using the Tikhonov regularizer for the Gaussian kernel. This observation underscores the fact that the choice of regularizer and kernel can impact performance in practice. 



\subsection{Brownian motion} \label{infdim}
In this section, we adopt the setup outlined in \citet{InfiniteDimKSD}, which considers $\X=L^2[0,1]$ and $P_0=N_C$, a Brownian motion over $[0,1]$ with covariance operator $C$ that has eigenvalues $\gamma_i=(i-1/2)^{-2}\pi^{-2}$ and eigenfunctions $e_i(t)=\sqrt{2}\sin((i-1/2)\pi t),\,t\in[0,1]$. The following alternate distributions are considered in the experiments, which are variations of those considered in \cite{InfiniteDimKSD}. 
Let $B_t$ denote the standard Brownian motion. 
\begin{itemize}
    \item $P$ follows the law of $X_t=(1+\sin(2\delta\pi t))B_t$, $\delta \in \{0.15, 0.1, 0.05, 0.01\}.$\vspace{.5mm} 
    \item $P$ is the law of the Brownian motion clipped to $k$ frequencies, $k \in [4]$, i.e.,  $\sum_{i=1}^k \sqrt{\gamma_i} \xi_i e_i$,  where $\xi_i \stackrel{i.i.d}{\sim} N(0,1)$.
\end{itemize}
Note that $\delta$ and $k$ generate a sequence of alternatives in the above cases with small $\delta$ and large $k$ making $P$ difficult to distinguish from $P_0$.

For our proposed regularized tests, we use $\Lambda=\{0.01,0.05,1\}$ and $W=\{h_m\}$, where $h_m$ denotes the median bandwidth. We use additionally generated samples of size $n_2$ and $n$ (see the captions of Figures~\ref{fig:Infdim} and \ref{fig:Infdim2} for details) from $P_0$ to construct KSD$^*$ and MMD tests.  We generated the samples using the settings and code provided in \citet{InfiniteDimKSD}. In this setup, function samples were observed at 100 points uniformly distributed across the interval [0,1], and 100 basis functions of the Brownian motion were used for the numerical approximation of $X_t$.

Figure \ref{fig:Infdim} shows the power of the first experiment for different values of $\delta$ with $n_2=35$ and $n=70$ (see also Figure~\ref{fig:inf_delta_n2} for the effect of changing $n_2$). In this experiment, the best performance is achieved by KSD(Tikhonov)$^*$ and MMD(Tikhonov), which have access to extra samples, followed by KSD(Tikhonov), with matching power to KSD(Tikhonov)$^*$ at $\delta=0.01$. In this scenario, Tikhonov regularization outperforms TikMax regularization, which can be explained by the fact that $X_t$ is a perturbed version of $B_t$ with high-frequency components being added. Thus, Tikhonov regularization is a more appropriate choice in this case.

Figure \ref{fig:Infdim2} illustrates the power of the second experiment with different values of $k$ with $n_2=100$ and $n=200$ (see also Figure~\ref{fig:inf_freq_n2} for the effect of changing $n_2$). In this experiment, once again, the best performance is achieved by MMD(Tikhonov), followed by the regularized KSD methods with Tikhonov regularization, which outperforms the unregularized KSD and KSD(TikMax). This can be explained by the fact that $X_t$ differs from $B_t$ at higher frequency components, making the Tikhonov regularizer a more appropriate choice in this case too. \vspace{-4mm}
\begin{figure}[t]
\centering
\includegraphics[scale=0.35]{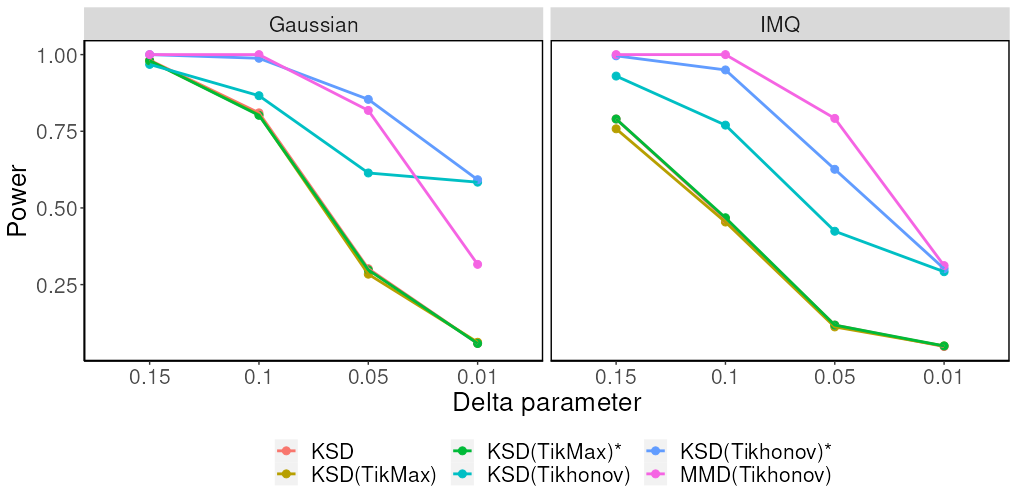}
\caption{Power for different values for $\delta$ using $n_2=35$ and $n=70$.} \label{fig:Infdim}
\vspace{0mm}
\end{figure}

\begin{figure}[H]
\centering
\includegraphics[scale=0.35]{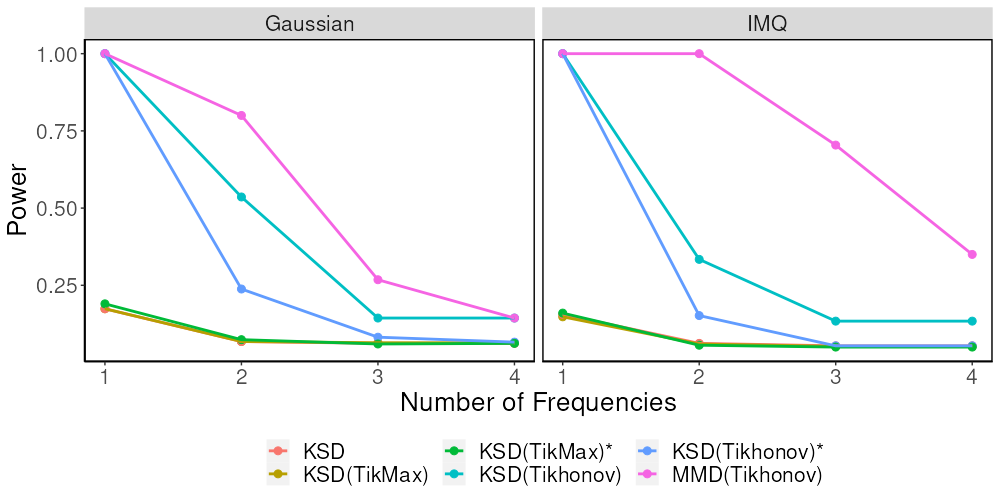}
\caption{Power for different values of $k$ using $n_2=100$ and $n=200$.} \label{fig:Infdim2}
\vspace{0mm}
\end{figure}
\vspace{-5mm}

\section{Discussion}\label{sec:discussion} 
To summarize, we studied the minimax optimality of goodness-of-fit tests across general domains, using the kernelized Stein discrepancy (KSD) framework—a versatile tool for conducting goodness-of-fit testing that can accommodate diverse data structures beyond Euclidean spaces and operate efficiently with only partial knowledge of the reference distribution. We formulated a comprehensive framework and provided an operator-theoretic representation of KSD, encompassing various existing KSD tests in the literature, each tailored to specific domains. Building on the concept of spectral regularization, we proposed a regularized KSD test and derived its separation boundary with respect to a suitable class of alternatives, using the $\chi^2$-distance as the separation metric. We demonstrated that the proposed regularized test is minimax optimal by deriving the minimax separation boundary of the considered alternative space and matching it with the separation boundary of our proposed test. This establishes the minimax optimality of the test across a wide range of the smoothness parameter $\theta$. On the other hand, we showed that the unregularized KSD test fails to achieve the minimax separation rate over the considered alternative space, thereby confirming its non-optimality. 
Furthermore, through numerical experiments, we demonstrated the superior performance of our proposed tests across various domains compared to the unregularized KSD test. 

There are several promising directions for future research. First, extending our results to handle composite null hypotheses would broaden the applicability of the proposed tests. Next, while our framework encompasses many KSD variations defined across different domains, there remains room for further investigation in terms of interpreting the smoothness condition (defined through the range space of an integral operator) for various specific domains and null distributions. Examples include discrete distributions supported on countable domain and network data, which can be addressed within the proposed general framework but require additional careful examination for any required domain-specific adjustments. Moreover, exploring alternative methods of adaptation could enhance the test performance. For instance, investigating approaches that provide tighter bounds than the conservative union bound, such as studying the asymptotic distribution of the adapted test, holds promise for further enhancing the test power.



\section{Proofs}\label{sec:proofs}
In this section, we present proof of all the main results of the paper.

\subsection{Proof of Proposition \ref{thm: computation}}\label{subsec:compute}
Define $$S_z : \h_{K_0} \to \R^{n_2},\ \ f \mapsto \frac{1}{\sqrt{n_2}}(f(Z_1),\cdot \cdot \cdot,f(Z_{n_2}))^\top$$ so that $$S_z^* : \R^{n_2} \to \h_{K_0},\ \ \alpha \mapsto \frac{1}{\sqrt{n_2}}\sum_{i=1}^{n_2} \alpha_iK_0(\cdot,Z_i).$$ Then observe that $\hat{\A}_{P}=S_z^{*}S_z$ and $n_2S_z S_z^*= [K_0(Z_i,Z_j)]_{i,j\in [n_2]} := K_{n_2}.$ Let $(\hat{\lambda}_i,\hat{\alpha}_i)_i$ be the eigensystem of $\frac{1}{n_2}K_{n_2}$, i.e., $S_z S_z^* \hat{\alpha}_i=\hat{\lambda}_i \hat{\alpha}_i$. By applying $S_z^*$ to both sides on the eigen equation, we can show that $(\hat{\lambda}_i, \hat{\psi}_i)$ are the eigensystem of $\hat{\A}_P$ where $\hat{\psi}_i:= \frac{1}{\sqrt{\hat{\lambda}_i}}S_z^* \hat{\alpha}_i.$ Thus, using $\hat{\psi}_i$ in the definition of $g_{\lambda}(\hat{\A}_P)$ yields that 
$$g_{\lambda}(\hat{\A}_P):=g_{\lambda}(0)\Id+S_z^* G S_z.$$
Define $\one_n : =(1,\stackrel{n}{\ldots},1)^{\top}$, and let $\one_n^i$ be a vector of zeros with only the $i^{th}$ entry equal one. Also define $S_x$ similar to that of $S_z$ based on samples $(X_i)^{n_1}_{i=1}$ and observe that $\sqrt{n_{1}n_{2}}S_x S_z^{*}:=K_{n_1,n_2}.$ Therefore,
\begin{align*}
\stat^P(\mathbb{X}_n)=\frac{1}{n_1(n_1-1)}\left(\footnotesize{\circled{1}}-\circled{2}\right), 
\end{align*}
where $\circled{\footnotesize{1}}:=\sum_{i,j=1}^{n_1} h(X_i,X_j),$ $\circled{\footnotesize{2}}:=\sum_{i=1}^{n_1}h(X_i,X_i),$ and  $$h(X_i,X_j)=\inner{\gSh K_0(\cdot,X_i)}{\gSh K_0(\cdot,X_j)}_{\h_{K_0}}.$$
Note that 
\begin{align*}
\circled{\footnotesize{1}} & = \inner{\gS\sum_{i=1}^{n_1}K_0(\cdot,X_i)}{\sum_{j=1}^{n_1}K_0(\cdot,X_j)}_{\h_{K_0}}=n_1\inner{\gS S_x^*\one_{n_1}}{S_x^*\one_{n_1}}_{\h_{K_0}}\\
    &= n_1\inner{S_x\gS S_x^*\one_{n_1}}{\one_{n_1}}_{2}
     = n_1\inner{g_{\lambda}(0)\frac{1}{n_1}K_{n_1}+\frac{1}{n_{1} n_{2}}K_{n_1,n_2} GK_{n_1,n_2}^\top\one_{n_1}}{\one_{n_1}}_{2} \\
    & = \one_{n_1}^\top\left(g_{\lambda}(0)K_{n_1}+\frac{1}{n_2}K_{n_1,n_2} G K_{n_1,n_2}^\top\right)\one_{n_1},
\end{align*}
and 
\begin{align*}
    \circled{\footnotesize{2}} & = \sum_{i=1}^{n_1}\inner{\gS K_0(\cdot,X_i)}{K_0(\cdot,X_i)}_{\h_{K_0}}=n_1\sum_{i=1}^{n_1}\inner{\gS S_x^*\one_{n_1}^i}{S_x^*\one_{n_1}^i}_{\h_{K_0}}\\ 
    & = n_1\sum_{i=1}^{n_1}\inner{S_x\gS S_x^*\one_{n_1}^i}{\one_{n_1}^i}_{2} \\
&=n_1\sum_{i=1}^{n_1}\inner{g_{\lambda}(0)\frac{1}{n_1}K_{n_1}+\frac{1}{n_1 n_2}K_{n_1,n_2} G K_{n_1,n_2}^\top\one_{n_1}^i}{\one_{n_1}^i}_{2}\\
    & = \text{Tr}\left(g_{\lambda}(0)K_{n_1}+\frac{1}{n_2}K_{n_1,n_2} G K_{n_1,n_2}^\top\right).
\end{align*}

\subsection{Proof of Theorem \ref{Type-I error}}\label{subsec:wildbootstrap}
 Let $h(x,y) := \inner{\gShh K_0(\cdot,x)}{\gShh K_0(\cdot,y)}_{\h_{K_0}}$. 
Then observe that $h$ is symmetric and degenerate with respect to $P_0$,  i.e., $\E_{P_0}h(x,Y)=0$ and $\E_{P_0}h(X,Y)=0$. Moreover 
\begin{align*}
    \E_{P_0} (h^2(X,Y) | \mathbb Z_{n_2}) &\leq \norm{\gShh}_{\opS}^4 \left[\E_{P_0}K_0(X,X)\right] \left[\E_{P_0}K_0(Y,Y)\right] \\ &\lesssim \frac{[\E_{P_0}K_0(X,X)^2]}{\lambda^2} < \infty, 
\end{align*}
and 
\begin{align*}
    \E_{P_0} (h(X,X) | \mathbb Z_{n_2}) \leq \norm{\gShh}_{\opS}^2 \left[\E_{P_0}K_0(X,X)\right] \lesssim \frac{1}{\lambda} [\E_{P_0}K_0(X,X)] < \infty.
\end{align*}
The result therefore follows from \citet[Theorem 3.1]{DEHLING1994392}.

\subsection{Separation boundary theorem}\label{subsec:power}
\begin{appxthm}\label{thm:sep-bound-suff-conditions}
Let $n_2 \asymp n$, $\Delta_n := \chi^2(P_n,P_0),$ $\sup_{P_n \in \PP}\norm{\ep_{P_0}^{-\theta}\U}_{\Lp} < \infty,$ with $u:=dP_n/dP_0-1$ and $\tilde{\theta}:=\min\{\theta,\xi\}$. Suppose
\begin{equation}
    \lim_{n \to \infty} \Delta_n \lambda_n^{-2\tilde{\theta}}> d_1, \ \textit{for some } d_1>0, \label{eq:one}
\end{equation}
and 
\begin{equation}
    \lim_{n \to \infty}\frac{n \Delta_n D_{\lambda_n}^{-1/2}}{1+\sqrt{C_{\lambda_n} \Delta_n}} \to \infty,\label{eq:two}
\end{equation}where $$\Cl :=
\left\{
	\begin{array}{ll}
		\Nol \sup_{i} \norm{\phi_i}^2_{\infty},   &  \ \ \sup_i \norm{\phi_i}^2_{\infty} < \infty \\
		\left(\Nol\right)^{\frac{v-2}{2(v-1)}}\left(\frac{\kappa}{\lambda}\right)^{\frac{v}{v-1}},  & \ \  \text{otherwise}
	\end{array}\right., \,\,\,
 \text{and}\,\,\, D_{\lambda} =
\left\{
	\begin{array}{ll}
		\Ntlsq ,   &  \ \ \ 4\Cl\Delta_n < 1 \\
		\lambda^{-1}, & \ \  \text{otherwise}
	\end{array}
\right..$$
 Then for $(P_n)_n \subset \PP,$ we have $$\lim_{n \to \infty} P\{\stat^{P_n}(\mathbb{X}_n) \geq q^{\lambda}_{1-\alpha}\} = 1 $$ under either of the following conditions:

\begin{enumerate}
    \item  \ref{assump:a2} holds and 
    \begin{equation}
     \lambda_n {\sqrt{n}} \to \infty.    \label{eq:three}
    \end{equation}  
    Furthermore if $4 C_{\lambda_n} \Delta_n < 1$, then \eqref{eq:three} can be replaced by 
    \begin{equation}
        \frac{n\lambda_n(\mathcal{N}_{1}(\lambda_n)+C_{\lambda_n}\sqrt{\Delta_n})^{-1}}{\log n} \to \infty. \label{eq:three(b)}
    \end{equation} 

    \item  $\sup_{i}\norm{\phi_i}_{\infty} < \infty$ holds and 
    \begin{equation}
         \frac{n\lambda_n}{\log n} \to \infty. \label{eq:four}
    \end{equation}

    Furthermore if $4\Cl \Delta_n < 1$ and $\inf_{n>0} \E_{P_n}\phi_1^2 >0,$ then \eqref{eq:four} can be replaced by \begin{equation}
        \frac{n}{\mathcal{N}_{1}(\lambda_n) \log \mathcal{N}_{1}(\lambda_n)} \to \infty. \label{eq:four-(b)}
    \end{equation} 
\end{enumerate}
\end{appxthm}

\begin{proof}
We will show that under the conditions stated in the theorem, for any $0<\delta<1,$ we have $$\lim_{n \to \infty} P\{\stat^{P_n}(\mathbb{X}_n) \geq q^{\lambda}_{1-\alpha}\} > 1-3\delta,$$ and thus the desired result holds. Let $$\gamma := \frac{\log \frac{2}{\alpha}}{n_1\sqrt{\delta}} \left(\norm{\W}_{\opS} \sqrt{\zeta} + \norm{\W}_{\opS}^2 \sqrt{D_{\lambda}}+\zeta\right),$$ and $$\gamma_1:=\frac{1}{\sqrt{\delta}}\left(\frac{\norm{\W}_{\opS}\sqrt{\zeta}}{\sqrt{n_1}}+\frac{\norm{\W}_{\opS}^2 \sqrt{\D_{\lambda}}}{n_1}\right),$$
where $\zeta := \E \left(\stat^{P_n}(\mathbb{X}_n) | \mathbb Z_{n_2}\right)= \norm{\gSh \Psi_P}_{\h_{K_0}}^2$ and $\W=\hat{\A}_{P,\lambda}^{-1/2}\A_{P,\lambda}^{1/2}$. Observe that 

\begin{align*}
& P\left\{\stat^{P_n}(\mathbb{X}_n) \geq q^{\lambda}_{1-\alpha}\right\} \\& \geq P\left\{ \left\{\stat^{P_n}(\mathbb{X}_n) \geq \zeta - \sqrt{\frac{\text{Var}(\stat^{P_n} \mid \mathbb{Z}_{n_2})}{\delta}}\right\} \cap \left\{q^{\lambda}_{1-\alpha} \leq \zeta - \sqrt{\frac{\text{Var}(\stat^{P_n} \mid \mathbb{Z}_{n_2})}{\delta}}\right\} \right\} \\
& \geq 1 - P\left\{\stat^{P_n}(\mathbb{X}_n) \leq \zeta - \sqrt{\frac{\text{Var}(\stat^{P_n} \mid \mathbb{Z}_{n_2})}{\delta}}\right\} - P\left\{q^{\lambda}_{1-\alpha} \geq \zeta - \sqrt{\frac{\text{Var}(\stat^{P_n} \mid \mathbb{Z}_{n_2})}{\delta}}\right\} \\
& \geq 1 - P\left\{\left|\stat^{P_n}(\mathbb{X}_n) - \zeta\right| \geq \sqrt{\frac{\text{Var}(\stat^{P_n} \mid \mathbb{Z}_{n_2})}{\delta}}\right\} - P\left\{q^{\lambda}_{1-\alpha} \geq \zeta - \sqrt{\frac{\text{Var}(\stat^{P_n} \mid \mathbb{Z}_{n_2})}{\delta}}\right\} \\
& \geq 1- \delta - P\left\{q^{\lambda}_{1-\alpha} \geq \zeta - \sqrt{\frac{\text{Var}(\stat^{P_n} \mid \mathbb{Z}_{n_2})}{\delta}}\right\}, 
\end{align*}
where in the last step we have invoked Chebyshev's inequality. Thus by Lemmas \ref{lem:bound-var} and \ref{lem:bound BS quantile}, we can deduce that it is sufficient to show
$$P\left(\gamma \leq \zeta - c_1\gamma_1\right) \geq 1-2\delta, $$ for some constant $c_1>0.$ Observe that 
\begin{align*}
    &P\left(\gamma \leq \zeta - c_1\gamma_1\right) = P\left(\zeta \geq \gamma + c_1\gamma_1\right) \\ 
    &\geq P\left(\{\zeta \geq \gamma\} \cap \left\{\zeta \geq \frac{c_1\norm{\W}_{\opS}\sqrt{\zeta}}{\sqrt{\delta n_1}}\right\} \cap \left\{\zeta \geq \frac{c_1\norm{\W}^2_{\opS}\sqrt{\D_{\lambda}}}{\sqrt{\delta} n_1}\right\}\right) \\ 
    &\stackrel{(*)}{\geq} P\left(\{\zeta \geq \gamma\} \cap \left\{ \frac{c_2\norm{\W^{-1}}_{\opS}^{-2}\Delta_n}{1+\sqrt{\Cl}\sqrt{\Delta_n}} \geq \frac{\norm{\W}^2_{\opS}}{\delta n_1}\right\}\cap \left\{\norm{\W^{-1}}^{2}_{\opS} \leq \frac{3}{2}\right\} \right.\\
    &\qquad\qquad\left.\cap \left\{ \frac{c_2\norm{\W^{-1}}_{\opS}^{-2}\Delta_n}{1+\sqrt{\Cl}\sqrt{\Delta_n}} \geq \frac{\norm{\W}^2_{\opS}\sqrt{\D_{\lambda}}}{\sqrt{\delta} n_1}\right\} \cap \left\{\norm{\W}^2_{\opS} \leq 2\right\}\right)\\
    & \stackrel{(**)}{\geq} P\left(\left\{\zeta \geq \gamma\right\} \cap \left\{\norm{\W^{-1}}^{2}_{\opS} \leq \frac{3}{2}\right\} \cap \left\{\norm{\W}^2_{\opS} \leq 2\right\} \right) \\
    &\stackrel{(\dag)}{\geq} P\left(\left\{\norm{\W^{-1}}^{2}_{\opS} \leq \frac{3}{2}\right\} \cap \left\{\norm{\W}^2_{\opS} \leq 2\right\}  \right) \\
    &\geq 1- P\left(\norm{\W}^2_{\opS} \geq 2\right)-P\left(\norm{\W^{-1}}^{2}_{\opS} \geq \frac{3}{2}\right) \stackrel{(\ddag)}{\geq} 1-2\delta,
\end{align*}
where $c_2>0$, $(*)$ follows from Lemma \ref{Lem: bound zeta}$(i)$ under the assumption that $u \in \range (\ep_{P_0}^{\theta}),$ \\ $\sup_{P \in \PP}\norm{\ep_{P_0}^{-\theta}\U}_{\Lp} < \infty,$  and  for some $c_3>0,$
\begin{equation}
    \Delta_n \geq c_3 \lambda^{2\tilde{\theta}}. \label{eq:cond1}
\end{equation} 
$(**)$ and $(\dag)$ follows under the dominant condition that 
\begin{equation}
    \frac{\Delta_n}{1+\sqrt{\Cl}\sqrt{\Delta_n}} \geq \frac{c_4(\delta) \sqrt{\D_{\lambda}}}{n}, \label{eq:cond2}
\end{equation}
for some constant $c_4(\delta)\asymp \delta^{-1}$. $(\ddag)$ follows from Lemma \ref{lem:bound-op-norm-1}$(i)$ assuming 
\begin{equation}
    \lambda \geq \frac{c_5(\delta)}{\sqrt{n}}, \label{eq:lambda-bd1}
\end{equation}
where $c_5(\delta) \asymp \log(\frac{2}{\delta}),$
and if $4\Cl \Delta_n < 1$, then using Lemma \ref{lem:bound-op-norm-1}$(ii),$ \eqref{eq:lambda-bd1} can be replaced by
\begin{equation}
    \lambda \geq \frac{c_6(\delta)\log n}{n} \left(\Nol+\Cl \sqrt{\Delta_n}\right), \label{eq:lambda-bd2}
\end{equation}
for some constant $c_6(\delta) \asymp \log(\frac{2}{\delta}).$  When $\sup_{i}\norm{\phi_i}^2_{\opS}\leq C < \infty$ (which implies that the kernel is bounded) then using Lemma \ref{lem:bound-op-norm-1}$(iii),$
\eqref{eq:lambda-bd1} can be replaced by
\begin{equation}
    \lambda \geq \frac{c_7(\delta)\log{n}}{n}, \label{eq:lambda-bd3}
\end{equation}
where $c_7(\delta) \asymp \log(\frac{1}{\delta}).$ On the other hand, when $\sup_{i}\norm{\phi_i}^2_{\op}\leq C < \infty$ and $4\Cl \Delta_n < 1$, define
$$\gamma := \frac{\log \frac{2}{\alpha}}{n_1\sqrt{\delta}} \left(\norm{\V}_{\op} \sqrt{\zeta} + \norm{\V}_{\op}^2 \sqrt{\D_{\lambda}}+\zeta\right),$$ and $$\gamma_1:=\frac{1}{\sqrt{\delta}}\left(\frac{\norm{\V}_{\opS}\sqrt{\zeta}}{\sqrt{n_1}}+\frac{\norm{\V}_{\opS}^2 \sqrt{\D_{\lambda}}}{n_1}\right),$$ where $\V := \hat{\A}_{P,\lambda}^{-1/2}\A_{P_0,\lambda}^{1/2}.$ 
Then by Lemma \ref{lem:bound-var}$(ii)$, Lemma \ref{Lem: bound zeta}$(ii)$ and Lemma \ref{lem:bound-op-norm-2} and  exactly using the same approach as above, we can show that $$P\left(\gamma \leq \zeta - c_8\gamma_1\right) \geq 1-2\delta, $$ for some constant $c_8>0,$
under the same conditions \eqref{eq:cond1}, \eqref{eq:cond2} and replacing \eqref{eq:lambda-bd1} with 
\begin{equation}
    n \geq c_9(\delta) \Nol \log \Nol, \label{eq:lambda-bd4}
\end{equation}
for some constant $c_9(\delta)\asymp\log (\frac{1}{\delta}),$ and using $\log\frac{x}{\delta} \leq \log\frac{1}{\delta}\log x$ for large enough $x \in \R,$ and assuming $\E_{P_n}\phi_1^2 >0.$
Finally observe that the conditions \eqref{eq:one}--\eqref{eq:four-(b)} given in the theorem statement implies \eqref{eq:cond1}--\eqref{eq:lambda-bd4} respectively for any $0<\delta<1.$ 
    
\end{proof}

\subsection{General version of Corollories \ref{coro:poly} and \ref{coro:exp}}\label{sec:general-cor}
In this section, we present a general version of Corollories \ref{coro:poly} and \ref{coro:exp} under the weaker assumption of \ref{assump:a2}.
\begin{appxcoro} \label{coro:general:poly}
Suppose $\lambda_i \lesssim i^{-\beta},$ $\beta>1$. Let $\lambda_n \asymp \Delta_n^{\frac{1}{2\tilde{\theta}}},$ with $\Delta_n \to 0$. Then for $(P_n)_n \subset \PP,$ we have  
$$\lim_{n \to \infty}P\{\stat^{P_n} \geq q^{\lambda}_{1-\alpha}\} = 1,$$
if one of the following holds:
\begin{enumerate}
    \item [(i)] \ref{assump:a2} holds and $\Delta_n B_n^{-1} \to \infty,$ where 
$$B_{n}  =
\left\{
	\begin{array}{ll}
	n^{\frac{-4\tilde{\theta} \beta}{4\tilde{\theta} \beta+1}},  &  \ \ \Tilde{\theta}> b_1 \\

 \left(\frac{\log n}{ n}\right)^{\frac{2\tilde{\theta}\beta}{\beta+1}},  & \ \ b_2<\Tilde{\theta} \leq b_1\\
		 \left(\frac{\log n}{n}\right)^{\frac{4\tilde{\theta}\beta}{2\beta+\frac{v}{v-1}(1-\frac{2}{v}+2\beta)-2\tilde{\theta}\beta}}, & \ \  \frac{v-2}{4\beta(v-1)}+\frac{v}{2(v-1)}< \Tilde{\theta} \leq  b_2 \\
      n^{\frac{-8\tilde{\theta} \beta}{2(1+\frac{v}{v-1})\beta+4\tilde{\theta}\beta+\frac{v-2}{v-1}}}, & \ \ \frac{2v-3}{2v-2}-\frac{v-2}{4\beta(v-1)}\leq \tilde{\theta} \leq \frac{v-2}{4\beta(v-1)}+\frac{v}{2(v-1)} \\
      n^{-\tilde{\theta}}, & \ \ \text{otherwise}
	\end{array}
\right.,$$
with \hspace{-0.1mm}$b_1=  \max\left\{\frac{1}{3}+\frac{2\beta v-1}{6\beta(v-1)}, \frac{v-2}{4\beta(v-1)}+\frac{v}{2(v-1)}\right\},$  $b_2 = \min\{\frac{v}{v-1}-\frac{v}{2(v-1)\beta},b_1\},$ and $\tilde{\theta}=\min\{\theta,\xi\}.$
\item [(ii)] $\sup_{i}\norm{\phi_i}_{\infty} < \infty$ and $\Delta_n B_n^{-1} \to \infty,$ where
$$B_{n}  =
\left\{
	\begin{array}{ll}
	n^{\frac{-4\tilde{\theta} \beta}{4\tilde{\theta} \beta+1}},  &  \ \ \tilde{\theta} > \frac{1}{2\beta} \\
	n^{\frac{-4\tilde{\theta} \beta}{2\tilde{\theta}\beta+\beta+1}}	, & \ \  \frac{1}{2}-\frac{1}{2\beta} < \tilde{\theta} \leq \frac{1}{2\beta} \\
    \left(\frac{\log n}{n}\right)^{2\tilde{\theta}}, & \ \ \tilde{\theta}\le\min\{\frac{1}{2}-\frac{1}{2\beta},\frac{1}{2\beta}\}
	\end{array}
\right..$$
\end{enumerate}

\end{appxcoro}

\begin{proof} 

We will verify that the conditions of Theorem \ref{thm:sep-bound-suff-conditions} are satisfied. If $\lambda_i \lesssim i^{-\beta},$ for $\beta >1$ then by \citet[Lemma B.9]{kpca}, we have $\Nol \lesssim \lambda^{\frac{-1}{\beta}}$ and $\Ntl \leq \mathcal{N}^{1/2}_{1}(\lambda) \lesssim \lambda^{\frac{-1}{2\beta}}$. Since
$\lambda_n \asymp \Delta_n^{\frac{1}{2\tilde{\theta}}}$,  \eqref{eq:one} is directly satisfied.\vspace{1.5mm}\\ \emph{(i)}  
In this case $$\Cl = \frac{(\Nol)^{\frac{v-2}{2(v-1)}}}{\lambda^{\frac{v}{v-1}}}=\lambda^{\frac{-v}{v-1}\left(\frac{1}{2\beta}-\frac{1}{v\beta}+1\right)} \asymp \Delta_n^{{\frac{-v}{2\tilde{\theta}(v-1)}\left(\frac{1}{2\beta}-\frac{1}{v\beta}+1\right)}}.$$ 
Observe that  
$$4C_\lambda \Delta_n \asymp \Delta_n^{1-{\frac{v}{2\tilde{\theta}(v-1)}\left(\frac{1}{2\beta}-\frac{1}{v\beta}+1\right)}} \stackrel{(*)}{<}1,$$
for $n$ large enough, where $(*)$ holds since $\Delta_n \to 0$, and using $1-{\frac{v}{2\tilde{\theta}(v-1)}\left(\frac{1}{2\beta}-\frac{1}{v\beta}+1\right)}>0,$ which is equivalent to $\tilde{\theta} > \frac{v-2}{4\beta(v-1)}+\frac{v}{2(v-1)}.$  Thus it remains to verify \eqref{eq:two} and \eqref{eq:three(b)} for this range of $\tilde{\theta}$. Observe that in this range $D_{\lambda}=\Ntlsq \lesssim \lambda^{\frac{-1}{\beta}} \asymp \Delta_n^{\frac{-1}{2\tilde{\theta}\beta}},$ thus \eqref{eq:two} is equivalent to $\Delta_n^{\frac{4\tilde{\theta}\beta +1}{4\tilde{\theta}\beta}} n \to \infty$, which in turn is equivalent to
\begin{equation}
\Delta_n n^{\frac{4\tilde
{\theta}\beta}{4\tilde
{\theta}\beta+1}} \to \infty. \label{poly-cond1}    
\end{equation}

\eqref{eq:three(b)} is implied if $\frac{n\lambda_n}{\mathcal{N}_1(\lambda_n) \log n} \to \infty$ and $\frac{n\lambda_n}{C_{\lambda_n} \sqrt{\Delta_n} \log n} \to \infty,$ where we used that $\frac{1}{a+b} \leq \frac{1}{a} + \frac{1}{b},$ for $a,b>0.$ Then after simplifying the expressions,  the first condition $\frac{n\lambda_n}{\mathcal{N}_1(\lambda_n) \log n} \to \infty$ is implied when 

\begin{equation}
  \Delta_n \left(\frac{n}{\log n}\right)^{\frac{2\tilde{\theta}\beta}{\beta+1}} \to \infty,  \label{poly-cond2}
\end{equation} 
and the second condition 
$\frac{n\lambda_n }{C_{\lambda_n} \sqrt{\Delta_n}\log n} \to \infty$
is implied if $\Delta_n^{\frac{2\beta +\frac{v}{v-1}(1-\frac{2}{v}+2\beta)-2\tilde{\theta}\beta}{4\tilde{\theta}\beta}}\left(\frac{n}{\log n}\right) \to \infty.$ This is directly satisfied  if $2\beta +\frac{v}{v-1}(1-\frac{2}{v}+2\beta)-2\tilde{\theta}\beta < 0$ 
which is equivalent to $\tilde{\theta} > 1+\frac{v}{v+1}+\frac{v-2}{(v-1)2\beta}.$ On the other hand, if $\tilde{\theta} \le 1+\frac{v}{v+1}+\frac{v-2}{(v-1)2\beta},$ then the second condition $\frac{n\lambda_n }{C_{\lambda_n} \sqrt{\Delta_n}\log n} \to \infty$
is implied if 
\begin{equation}
   \Delta_n \left(\frac{n}{\log n}\right)^{\frac{4\tilde{\theta}\beta}{2\beta +\frac{v}{v-1}(1-\frac{2}{v}+2\beta)-2\tilde{\theta}\beta}} \to \infty.\label{poly-cond3}
\end{equation} 
Since $\max\left\{1+\frac{v}{v+1}+\frac{v-2}{(v-1)2\beta}, \frac{v-2}{4\beta(v-1)}+\frac{v}{2(v-1)}\right\}=1+\frac{v}{v+1}+\frac{v-2}{(v-1)2\beta}$ for all $v>2$, in the case of 
$\tilde{\theta}> 1+\frac{v}{v+1}+\frac{v-2}{(v-1)2\beta},$ the conditions on $\Delta_n$ to be satisfied are \eqref{poly-cond1} and \eqref{poly-cond2}, and in the case $\frac{v-2}{4\beta(v-1)}+\frac{v}{2(v-1)}<\tilde{\theta}\leq 1+\frac{v}{v+1}+\frac{v-2}{(v-1)2\beta}$, then \eqref{poly-cond3} has also to be satisfied in addition to \eqref{poly-cond1} and \eqref{poly-cond2}. 
On the other hand, if $\tilde{\theta}\le \frac{v-2}{(v-1)4\beta}+\frac{v}{2(v-1)},$ then it is easy to verify that \eqref{eq:two} (with $D_{\lambda} = \lambda^{-1}$) and \eqref{eq:three} are implied if 
\begin{equation}
  \Delta_n n^{\frac{4\tilde{\theta}}{4\tilde{\theta} +1 }}\to \infty, \label{poly-cond4} 
\end{equation}

\begin{equation}
  \Delta_n n^{\frac{8\tilde{\theta} \beta}{2(1+\frac{v}{v-1})\beta+4\tilde{\theta}\beta+\frac{v-2}{v-1}}} \to \infty, \label{poly-cond5}  
\end{equation}
and 

\begin{equation}
\Delta_n  n^{\tilde{\theta}}\to \infty. \label{poly-cond6}    
\end{equation}
Observe that when
$\frac{4\tilde{\theta} \beta}{4\tilde{\theta}\beta+1} \leq \frac{2\tilde{\theta} \beta}{\beta+1} \Leftrightarrow \frac{1}{2}+\frac{1}{4\beta} \leq \tilde{\theta},$ and $\frac{4\tilde{\theta}\beta}{4\tilde{\theta}\beta+1} \leq \frac{4\tilde{\theta}\beta}{2\beta+\frac{v}{v-1}(1-\frac{2}{v}+2\beta)-2\tilde{\theta}\beta} \Leftrightarrow \frac{1}{3}+\frac{2\beta v-1}{6\beta(v-1)}
\leq \tilde{\theta}$ then \eqref{poly-cond1} dominates \eqref{poly-cond2} and \eqref{poly-cond3} for large enough $n$. 
When $\frac{4\Tilde{\theta}\beta}{2\beta + \frac{v}{v-1}(1-\frac{2}{v}+2\beta)-2\Tilde{\theta}\beta} < \frac{4\Tilde{\theta}\beta}{\beta+1} \Leftrightarrow \Tilde{\theta} < \frac{v}{v-1}-\frac{v}{2(v-1)\beta}$, then \eqref{poly-cond3} dominates \eqref{poly-cond2} . If $\frac{8\tilde{\theta} \beta}{2(1+\frac{v}{v-1})\beta+4\tilde{\theta}\beta+\frac{v-2}{v-1}} \leq \tilde{\theta} \Leftrightarrow \frac{2v-3}{2v-2}-\frac{v-2}{4\beta(v-1)}\leq \tilde{\theta}$, and $\frac{8\tilde{\theta} \beta}{2(1+\frac{v}{v-1})\beta+4\tilde{\theta}\beta+\frac{v-2}{v-1}} \leq \frac{4\tilde{\theta}}{4\tilde{\theta}+1} \Leftrightarrow \tilde{\theta} < \frac{v}{2(v-1)}+\frac{v-2}{4\beta(v-1)}$, then \eqref{poly-cond5} dominates both \eqref{poly-cond4} and \eqref{poly-cond6} for large enough $n$. In other ranges \eqref{poly-cond6} dominates. Thus putting things together yields that $$\Delta_n B_n^{-1} \to \infty,$$ with $B_n$ as mentioned in Corollary \ref{coro:general:poly}\emph{(i)}.

\underline{\emph{(ii)} $\sup_{i}\norm{\phi_i}_{\infty} < \infty$:}  In this case, $\Cl \asymp \Nol$ and we need to verify equations \eqref{eq:one}, \eqref{eq:two} and \eqref{eq:four}, where \eqref{eq:four} can be replaced by \eqref{eq:four-(b)} when $4\Cl \Delta_n < 1$. When $\tilde{\theta}>\frac{1}{2\beta}$, it can be verified that $4\Cl \Delta_n < 1$ since  $\Delta_n  \to 0$. Then the required conditions are implied when 
$$\Delta_n \min\left\{n^{\frac{4\tilde{\theta}\beta}{4\tilde{\theta}\beta+1}}, \left(\frac{n}{\log n}\right)^{2\tilde{\theta}\beta}\right\} \to \infty.$$ On the other hand, if $\tilde{\theta} \leq \frac{1}{2\beta}$, then the required conditions are implied when 
$$\Delta_n \min\left\{n^{\frac{4\tilde{\theta}}{4\tilde{\theta}+1}}, n^{\frac{4\tilde{\theta}\beta}{2\tilde{\theta}\beta+\beta+1}}, \left(\frac{n}{\log n}\right)^{2\tilde{\theta}}\right\} \to \infty.$$
Putting the conditions together yields $$\Delta_n B_n^{-1} \to \infty,$$ where $B_n$ is mentioned in Corollary \ref{coro:general:poly}\emph{(ii)}.
\end{proof}

\begin{appxcoro}  \label{coro:general:exp}  
   Suppose $\lambda_i \lesssim e^{-\tau i},$ $\tau>0$. Let $\lambda_n \asymp \Delta_n^{\frac{1}{2\tilde{\theta}}}$ with $\Delta_n \to 0$. Then for any $(P_n)_n \subset \PP,$ we have  
$$\lim_{n \to \infty}P\{\stat^{P_n} \geq q^{\lambda}_{1-\alpha}\} =1,$$
if one of the following holds:
\begin{enumerate}
    \item [(i)]  \ref{assump:a2} holds and $\Delta_n B_n^{-1} \to \infty,$ where 
$$B_{n}  =
\left\{
	\begin{array}{ll}
	\frac{\sqrt{\log n}}{n},  &  \ \ \Tilde{\theta}> \frac{2v-1}{3(v-1)} \\
		\left(\frac{(\log n)^2}{n}\right)^{\frac{2\tilde{\theta}}{\frac{2v-1}{v-1}-\tilde{\theta}}} , & \ \  \frac{v}{2(v-1)} \leq \tilde{\theta} \leq \frac{2v-1}{3(v-1)} \\
      \left(\frac{\log n}{n}\right)^{\frac{4\tilde{\theta}(v-1)}{2\tilde{\theta}(v-1)+2v-1}}, & \ \ \frac{2v-3}{2(v-1)} \leq \tilde{\theta} \leq \frac{v}{2(v-1)} \\
      n^{-\tilde{\theta}}, & \ \ \text{otherwise}
	\end{array}
\right.,$$
where $\tilde{\theta}=\min\{\theta,\xi\}.$
\item[(ii)]  $\sup_{i}\norm{\phi_i}_{\infty} < \infty$ and $\Delta_n B_n^{-1} \to \infty,$ where
$$B_{n}  = \frac{\sqrt{\log n}}{n}.$$
\end{enumerate} 
\end{appxcoro}
\begin{proof}
    We will verify that the conditions of Theorem \ref{thm:sep-bound-suff-conditions} are satisfied. If $\lambda_i \asymp e^{-\tau i},$ then \citep[Lemma B.9]{kpca} yields $\Nol \lesssim \log \frac{1}{\lambda}$ and $\Ntl \leq \mathcal{N}^{1/2}_{1}(\lambda) \lesssim \sqrt{\log \frac{1}{\lambda}}$. Since
$\lambda_n \asymp \Delta_n^{1/2\tilde{\theta}}$,  \eqref{eq:one} is directly satisfied. 

\emph{(i)}  In this case $\Cl \asymp \frac{(\log\lambda^{-1})^{\frac{v-2}{2(v-1)}}}{\lambda^{\frac{v}{v-1}}}$.\vspace{1mm}

\underline{\textit{For $\tilde{\theta}> \frac{v}{2(v-1)}$:}} In this case, it can be verified that $4C_\lambda \Delta_n< 1$ since $\Delta_n \to 0$. Thus we only need to verify \eqref{eq:two} and \eqref{eq:three(b)}. Equation \eqref{eq:two} is implied when 
$$\Delta_n \frac{n}{\sqrt{\log n}} \to \infty.$$
\eqref{eq:three(b)} is implied if 
$$\Delta_n \left(\frac{n}{(\log n)^2}\right)^{2\tilde{\theta}} \to \infty,$$
and $ \frac{n\lambda_n}{\Cl\sqrt{\Delta_n}\log n} \to \infty$
 which is satisfied for large enough $n$ when $\tilde{\theta} \geq \frac{2v-1}{v-1}.$

 \underline{\textit{For} $\frac{v}{2(v-1)}\leq \tilde{\theta} \leq \frac{2v-1}{v-1}$:} In addition to the conditions in the previous case we also have 

 $$\Delta_n \left(\frac{n}{(\log n)^2}\right)^{\frac{2\tilde{\theta}}{\frac{2v-1}{v-1}-\tilde{\theta}}} \to \infty.$$

\underline{\textit{For} $\tilde{\theta} \leq \frac{v}{2(v-1)}$:} In this case, we need to verify equations \eqref{eq:two} (with $D_{\lambda} = \lambda^{-1}$) and \eqref{eq:three} which are implied if 
$$\Delta_n \min\left\{n^{\frac{4\tilde{\theta}}{4\tilde{\theta}+1}}, \left(\frac{n}{\log n}\right)^{\frac{4\tilde{\theta}(v-1)}{2\tilde{\theta}(v-1)+2v-1}},n^{\tilde{\theta}}\right\} \to \infty.$$
Putting the conditions together while using 
$$\frac{2\tilde{\theta}}{\frac{2v-1}{v-1}-\tilde{\theta}} \geq 1 \Leftrightarrow \tilde{\theta} \geq \frac{2v-1}{3(v-1)},$$
and 
$$\frac{4\tilde{\theta}(v-1)}{2\tilde{\theta}(v-1)+2v-1} > \tilde{\theta} \Leftrightarrow \frac{2v-3}{2(v-1)} > \tilde{\theta},$$
yields the desired result.\vspace{1mm}\\
\underline{\emph{(ii)} $\sup_{i}\norm{\phi_i}_{\infty} < \infty$:} In this case $\Cl \asymp \Nol \asymp \log \frac{1}{\lambda}$, and it can also be verified that for large enough $n$, $4\Cl \Delta_n < 1$ is implied when $\Delta_n \to 0.$ Thus it only remains to verify equations \eqref{eq:two} and \eqref{eq:four-(b)} which are both implied when 
$\Delta_n \frac{n}{\sqrt{\log n}} \to \infty.$
\end{proof}

\subsection{Proof of Theorem \ref{thm:minimax}}\label{subsec:minimax}
For completeness, we briefly outline the main argument in Ingster's method for the lower bound. Let \(\phi\) be any \(\alpha\)-level test, and suppose \(\mathbb{E}_{P_0^n} \left[\left(L^{-1} \sum_{k=1}^L dP_k^n/dP_0^n\right)^2\right] = M_n^2 < \infty\). Then, as shown in \citet{Ingester1},
\begin{align*}
\inf_{\phi \in \Phi_{\alpha}} R_\Delta(\phi) + \alpha &= \inf_{\phi \in \Phi_{\alpha}} \sup_{P \in \mathcal{P}_\Delta} \mathbb{E}_{P^n}(1 - \phi) + \alpha \\ & \geq  1 - \mathrm{TV}\left(\frac{1}{L}\sum_{k=1}^L P_k^n,P_0^n\right)  
\geq 1 - \frac{1}{2} M_n \left(\sqrt{M_n^2 + 4} - M_n\right),
\end{align*}
where $\mathrm{TV}(P,Q)$ represents the total variation between the measures $P$ and $Q$. 
Thus, \(\lim_{n \to \infty} M_n < \infty\) implies that \(\lim_{n \to \infty} M_n (\sqrt{M_n^2 + 4} - M_n) < 2\), which in turn implies that \(\lim_{n \to \infty} \inf_{\phi \in \Phi_{\alpha}} R_\Delta(\phi) + \alpha > 0\). Therefore, there exists \(\tilde{\alpha} > 0\) such that for any \(\alpha \leq \tilde{\alpha}\),
\[
\lim_{n \to \infty} \inf_{\phi \in \Phi_{\alpha}} R_\Delta(\phi) > 0,
\]
which means there is no test $\phi\in\Phi_\alpha$ that is consistent over $\mathcal{P}_\Delta$. This means, it is sufficient to find a set of distributions $\{P_k\}_{k=1}^L \subset \PP$ such that
\begin{equation*}
    \lim_{n \to \infty}\E_{P_0^n}\left[\left(\frac{1}{L}\sum_{k=1}^L\frac{dP_k^n}{dP_0^n}\right)^2\right] < \infty.
\end{equation*}
Then the proof exactly follows the proof of Theorem 2 of \citet{hagrass2023Gof} by replacing $\T$ with $\ep_{P_0}$ with one main difference, which is bounding $$\left|\frac{dP_k}{dP_0}(x)-1\right|=|u_{n,k}(x)|=\left|a_n\sum_{i=1}^{B_n}\epsilon_{ki}\phi_{i}(x)\right|$$ in the case when $\sup_{i} \norm{\phi_i}_{\infty}$ \emph{is not finite}, since in this paper we do not assume that $K_0$ is bounded. To this end, in the following, we obtain a bound on $|u_{n,k}(x)|$ without assuming the boundedness of $K_0$. Let $C_n=\floor{\sqrt{B_n}}$ and $a_n=\sqrt{\frac{\Delta_n}{C_n}}.$ Then
\begin{align*}
    |u_{n,k}(x)| &= \left|\inner{K_0(\cdot,x)}{a_n\sum_{i=1}^{B_n}\epsilon_{ki}\phi_{i}}_{\h_{K_0}}\right| = \left|\inner{K_0(\cdot,x)}{a_n\sum_{i=1}^{B_n}\epsilon_{ki}\lambda_{i}^{-1}\id^*\tilde{\phi}_{i}}_{\h_{K_0}}\right|\\
    &= \left|\E_{P_0}\left[K_0(x,Y) a_n \sum_{i=1}^{B_n}\epsilon_{ki}\lambda_{i}^{-1} \tilde{\phi}_i(Y)\right]\right| \\
    & \leq \left(a_n^2\sum_{i,j}^{B_n}\epsilon_{ki} \epsilon_{kj} \lambda_i^{-1} \lambda_j^{-1}\E_{P_0}[\tilde{\phi}_{i}(Y)\tilde{\phi}_{j}(Y)]\right)^{1/2}\left(\E_{P_0}[K(x,Y)^2]\right)^{1/2} \\
    &= a_n \left(\sum_{i=1}^{B_n} \epsilon_{ki}^2\lambda_{i}^{-2}\right)^{1/2}\left(\E_{P_0}[K(x,Y)^2]\right)^{1/2}.
\end{align*}
Thus when $\lambda_i \asymp i^{-\beta}$, the above bound yields 
$$|u_{n,k}(x)| \lesssim \sqrt{\Delta_n} B_n^{\beta} \lesssim \Delta_n^{\frac{1}{2}-\frac{
1
}{2\theta}},$$
where $B_n \asymp \Delta_n^{\frac{-1}{2\theta\beta}}$, which implies $|u_{n,k}(x)| < 1$ for large enough $n$ and $\theta >1.$ For the case $\lambda_i \asymp e^{-\tau i}$, the bound yields that 
$$|u_{n,k}(x)| \lesssim a_n C_n^{1/2} e^{\tau B_n} \lesssim 1, $$  where we used $B_n \asymp\tau^{-1}\log(\Delta_n^{-1/2}).$ This means for large $n$, $|u_{n,k}(x)| < 1$. Moreover, under \ref{assump:a1}, we have 
\begin{align*}
    \E_{P_k}\left[K_0(X,X)\right] =  \E_{P_0}\left[K_0(X,X)(u_{n,k}(X)+1)\right] \leq 2\E_{P_0}\left[K_0(X,X)\right] < \infty,
\end{align*}
and 
\begin{align*}
\E_{P_k}\left[K_0(X,X)^r\right] =  \E_{P_0}\left[K_0(X,X)^r(u_{n,k}(X)+1)\right] \leq 2\E_{P_0}\left[K_0(X,X)^r\right] \leq 2cr!\kappa^r.
\end{align*}
Thus $P_k \in \PP$ for any $k \in \{1,\dots,L\}$.
 
\subsection{Proof of Theorem \ref{thm:non-optimality-KSD}}\label{subsec:non-optimal-ksd}
Observe that for any $P \in \PP_{\Delta},$ we have  
$$ \inf_{P \in \PP_{\Delta}}P_{H_1}\{\hat{D}^2_{\mathrm{KSD}}\geq q_{1-\alpha}\} \leq P_{H_1}\{\hat{D}^2_{\mathrm{KSD}}\geq q_{1-\alpha}\}.$$
Then it is sufficient to find a sequence of probability distributions $\{P_n\}_n$, such that $P_n \in \PP_{\Delta_n},$ for each $n$, and $$\lim_{n \to \infty} P_{n}\{\hat{D}^2_{\mathrm{KSD}}\geq q_{1-\alpha}\} < 1.$$
First we will show that if $\D^2_{\mathrm{KSD}}(P_n,P_0) = \norm{\Psi_{P_{n}}}^2_{\h_{K_0}}= o(n^{-1})$, then the desired result holds, where $\Psi_P:=\E_{P}K_0(\cdot,X)$. 
Observe that
\begin{align*}
  &\hat{D}^2_{\mathrm{KSD}} = \frac{1}{n(n-1)} \sum_{i \neq j} K_0(X_i,X_j) \\
  &= \frac{1}{n(n-1)}\sum_{i \neq j}\inner{K_0(\cdot,X_i)-\Psi_{P_n}}{K_0(\cdot,X_j)-\Psi_{P_n}}_{\h_{K_0}}+\frac{2}{n} \sum_{i=1}^n \inner{K_0(\cdot,X_i)}{\Psi_{P_n}}_{\h_{K_0}}\\
  &\qquad\qquad-\norm{\Psi_{P_n}}^2_{\h_{K_0}} \\
  &= \frac{1}{n(n-1)}\sum_{i \neq j}\inner{K_0(\cdot,X_i)-\Psi_{P_n}}{K_0(\cdot,X_j)-\Psi_{P_n}}_{\h_{K_0}}+\frac{2}{n} \sum_{i=1}^n \inner{K_0(\cdot,X_i)-\Psi_{P_n}}{\Psi_{P_n}}_{\h_{K_0}}\\
  &\qquad\qquad+\norm{\Psi_{P_n}}^2_{\h_{K_0}} \\
  &:= \circled{\tiny{1}}+\circled{\tiny{2}}+\norm{\Psi_{P_n}}^2_{\h_{K_0}}.
\end{align*}
Then under $H_1$, i.e., $(X_i)_i\stackrel{i.i.d.}{\sim} P_n\in\mathcal{P}$, using standard arguments for the convergence of U-statistics (see \citealp[Section 5.5.2]{Serfling1980}) yields that 
 $$n\, \circled{\tiny{1}}\, {\rightsquigarrow} \sum_{i=1}^{\infty} \lambda_i (a_i^2-1):=S,$$
where $a_i \stackrel{i.i.d}{\sim} \mathcal{N}(0,1).$
Next, note that
$$n\,\circled{\tiny{2}} = 2\sqrt{n}\mathcal{V}_n \frac{1}{\sqrt{n}} \sum_{i=1}^{n}\inner{K_0(\cdot,X_i)-\Psi_{P_n}}{\frac{\Psi_{P_n}}{\mathcal{V}_n}}_{\h_{K_0}},$$
where $\mathcal{V}_n:=\sqrt{\norm{\mathcal{A}^{1/2}_{P_n}\Psi_{P_n}}^2_{\h_{K_0}}-\norm{\Psi_{P_n}}^4_{\h_{K_0}}}$ and $\mathcal{A}_{P_n}:=\int_\mathcal{X} K_0(\cdot,x)\otimes_{\h_{K_0}} K_0(\cdot,x)\,dP_n(x)$. The central limit theorem yields that $\frac{1}{\sqrt{n}} \sum_{i=1}^{n}\inner{K_0(\cdot,X_i)-\Psi_{P_n}}{\frac{\Psi_{P_n}}{\mathcal{V}_n}}_{\h_{K_0}} \rightsquigarrow Z,$
where $Z \sim \mathcal{N}(0,1).$ Thus, it follows that if $\norm{\Psi_{P_{n}}}^2_{\h_{K_0}}= o(n^{-1})$ and $\chi^2(P_n,P_0)=O(1)$ as $n\rightarrow\infty$, then 
\begin{align*}
n\mathcal{V}^2_n&=n\norm{\mathcal{A}^{1/2}_{P_n}\Psi_{P_n}}^2_{\h_{K_0}}-n\norm{\Psi_{P_n}}^4_{\h_{K_0}}\le n\Vert\mathcal{A}_{P_n}\Vert_{\opS}\norm{\Psi_{P_{n}}}^2_{\h_{K_0}}\\
& \le n\mathbb{E}_{P_n}K_0(X,X)\norm{\Psi_{P_{n}}}^2_{\h_{K_0}}\\
&\le n\left(\mathbb{E}_{P_0}K_0(X,X)+\sqrt{\mathbb{E}_{P_0} K_0^2(X,X)} \sqrt{\chi^2(P_n,P_0)}\right)\norm{\Psi_{P_{n}}}^2_{\h_{K_0}}\rightarrow 0
\end{align*}
as $n\rightarrow\infty$, implying
$$n\hat{D}^2_{\mathrm{KSD}} \rightsquigarrow S , $$ thereby yielding 

$$P_{H_1}\{\hat{D}^2_{\mathrm{KSD}}\geq q_{1-\alpha}\}= P_{H_1}\{n\hat{D}^2_{\mathrm{KSD}}\geq nq_{1-\alpha}\}\to P\{S \geq \tilde{d}\}<1,$$ 
where $\tilde{d}=\lim_{n \to \infty} n q_{1-\alpha}$, and the last inequality follows using the facts that $q_{1-\alpha}>0$ (by the symmetry of the wild bootstrap distribution) $S$ has a non-zero probability of being negative. Thus it remains to show that we can construct a sequence of probability distributions $\{P_n\}_{n} \in \PP_{\Delta_n}$ such that $\norm{\Psi_{P_{n}}}^2_{\h_{K_0}}= o(n^{-1}).$ 
Let $h$ be a strictly decreasing continuous function on $\mathbb{R}^{+}$, and suppose $\lambda_i =h(i)$. Let $k_n=\floor{ h^{-1}\left(\Delta_{n}^{1/2\theta}\right)}$ and define 
$$u_n := \lambda_{k_n}^{\theta}\phi_{k_n}.$$
Recall that $\id \phi_{k_n}=[\phi_{k_n}]_{\sim}=\tilde{\phi}_{k_n}$. Thus $\E_{P_0} \phi^2_{k_n} = \E_{P_0} \tilde{\phi}^2_{k_n}=1,$ and $$\E_{P_0} \phi_{k_n} = \E_{P_0} \tilde{\phi}_{k_n}=\lambda_{k_n}^{-1}\inner{1}{\id \id^* \tilde{\phi}_{k_n}}_{\Lp}=\inner{\id^* 1}{\phi_{k_n}}_{\h_{K_0}}=\inner{\E_{P_0}K_0(\cdot,x)}{\phi_{k_n}}_{\h_{K_0}}=0.$$  Then observe that $u_n \in \id u_n = [u_n]_{\sim} = \lambda_{k_n}^{\theta}\tilde{\phi}_{k_n}  \in \range(\ep_{P_0}^{\theta})$ and $\norm{u_n}_{\Lp}^2 = \chi^2(P_n,P_0) \geq \Delta_n.$
Thus it remains to ensure there exists $P_n$ such that $u_n=\frac{dP_n}{dP_0}-1$, for which it is sufficient to bound $|u_n(x)|$ as follows:
\vspace{1mm}\\
\noindent{\textit{Case 1:}} \underline{\ref{assump:a1} holds and $\theta>1$}\vspace{1mm}\\
In this case, 
\begin{align*}
    |u_n(x)| &= \lambda^\theta_{k_n}|\inner{K_0(\cdot,x)}{\phi_{k_n}}_{\h_{K_0}}| 
    = \lambda_{k_n}^{\theta-1}|\E_{P_0}(\tilde{\phi}_{k_n}(Y)K_0(x,Y))|\\
    & \leq \lambda_{k_n}^{\theta-1} \left(\E_{P_0}\tilde{\phi}_{k_n}^2(Y)\right)^{1/2} \left(\E_{P_0}K_0^2(x,Y)\right)^{1/2} \lesssim \lambda_{k_n}^{\theta-1},
\end{align*}
which implies that $|u_n(x)|<1$ for large enough $n$. Then to ensure that $P_n \in \PP$, it remains to show $P_n \in \mathcal{S}$. This is indeed the case since
\begin{align*}
\E_{P_n}\left[K_0(X,X)\right] =  \E_{P_0}\left[K_0(X,X)(u_n(X)+1)\right] \stackrel{(*)}{\leq} 2\E_{P_0}\left[K_0(X,X)\right] < \infty,
\end{align*}
and 
\begin{align*}
\E_{P_n}\left[K_0(X,X)^r\right] =  \E_{P_0}\left[K_0(X,X)^r(u_n(X)+1)\right]\stackrel{(*)}{\leq} 2\E_{P_0}\left[K_0(X,X)^r\right] \stackrel{(\dag)}{\leq} 2cr!\kappa^r,
\end{align*}
where $(*)$ follows using $|u_n(x)|<1$ for large $n$, and $(\dag)$ follows by \ref{assump:a1}. \vspace{1mm}\\
\textit{Case 2: \underline{$\sup_{i}\norm{\phi_i}_{\infty} < \infty$}}\vspace{1mm}\\
In this case, 
$$|u_n(x)| \leq \sup_k\norm{\phi_k}_{\infty} \lambda_{k_n}^{\theta}\leq 1,$$
for $n$ large enough. Then to ensure that $P_n \in \PP$, it remains to show $P_n \in \mathcal{S}$ which follows similarly to the previous case.

Thus in both the cases, $D^2_{\mathrm{KSD}}(P_n,P_0)=\norm{\Psi_{P_n}}^2_{\h_{K_0}}=\norm{\ep_{P_0}^{1/2}u_n}^2_{\Lp} \sim \Delta_n^{\frac{2\theta+1}{2\theta}}=o(n^{-1})$.

\subsection{Proof of Proposition~\ref{thm;range}}
Observe that $$\mathcal{D}_{P_0}=\sum_{k \in I} \lambda_k \gamma_k \ltens \gamma_k$$
since 
\begin{align*}
   \mathcal{D}_{P_0}f&=\int K_0(\cdot,x) f(x)\,dP_0(x)=\int \sum_{k\in I}\lambda_k \gamma_k\gamma_k(x) f(x)\,dP_0(x)\\&=\sum_{k\in I} \lambda_k\gamma_k\langle \gamma_k,f\rangle_{L^2(P_0)},\,\,f\in L^2(P_0). 
\end{align*}

Therefore, for $f \in \Lp$,  $$ \mathcal{D}^{\theta}_{P_0}f= \sum_{k \in I} \lambda_k^{\theta} b_k \gamma_k,$$
where $b_k := \inner{f}{\gamma_k}_{\Lp},$ with $\sum_{k \in I} b_k^2 < \infty.$ Thus any function belonging to $ \range(\mathcal{D}_{P_0}^{\theta})$ can be expressed as $\sum_{k \in I} a_k \gamma_k,$ with $(a_k)_{k\in I}$ satisfying $\sum_{k \in I} a^2_k \lambda_k^{-2\theta} < \infty.$

\subsection{Proof of Example \ref{Ex: uniform}}\label{proof:ex-uniform}
Observe that
$\bar{K}(x,y) : =  \sum_{k \neq 0} |k|^{-\beta} e^{\sqrt{-1}2\pi kx}e^{-\sqrt{-1}2\pi ky}$. Since $1 \cup (e^{\sqrt{-1}2\pi k \cdot})_{k \in I}$ form an orthonormal basis in $L^2([0,1])$, the result follow by applying Proposition \ref{thm;range}.
\vspace{2mm}\\
For $K_0(x,y)=\nabla_{x}\nabla_{y}K(x,y),$
we have 
$$K_0(x,y) = \sum_{k \neq 0 } \lambda_k \frac{d \gamma_k(x)}{dx}\frac{d \overline{\gamma_k(y)}}{dy},$$ where $\lambda_k:=|k|^{-\beta}$, and $\gamma_k(x) := e^{\sqrt{-1}2\pi k x}.$ Then observe that $\frac{d\gamma_k(x)}{dx} = \sqrt{-1}2\pi k \gamma_k(x),$ thus 

$$K_0(x,y) = (2\pi)^2 \sum_{k} \tilde{\lambda}_k \gamma_k(x)\gamma_k(y),$$ with $\tilde{\lambda}_k:= k^2 \lambda_k= |k|^{-(\beta-2)}.$ Thus we have $\E_{P_0}K_0(\cdot,X)=0$ and the result follows by applying Proposition~\ref{thm;range}.

\subsection{Proof of Example \ref{Ex: Gaussian}}\label{proof:ex-gaussian}
By Mehler's formula \citep{Kibble_1945}, we have 
\begin{equation}\label{Eq:mehler}K(x,y) = \sum_{k=0}^{\infty} \rho^k \gamma_k(x)\gamma_k(y),
\end{equation}
where 
$\gamma_k(x) = \frac{1}{\sqrt{k!}} H_k(x),$ and $H_k(x) = (-1)^ke^{x^2/2}\frac{d^k}{dx^k}e^{-x^2/2}.$
Observe that  $\gamma_0(x)=1$ and $1 \cup (\gamma_k)_{k=1}^{\infty}$ form an orthonormal basis with respect to $L^2(P_0)$ where $p_0(x) := \frac{1}{\sqrt{2\pi}}e^{-\frac{x^2}{2}}.$

Then $$\bar{K}(x,y) = \sum_{k=0}^{\infty} \rho^k \tilde{\gamma_k}(x)\tilde{\gamma_k}(y)= \sum_{k=1}^{\infty} \rho^k \gamma_k(x)\gamma_k(y),$$
where $\tilde{\gamma}_k(x)=\gamma_k(x)-\E_{P_0}\gamma_k(X).$ Thus the result follow by applying Theorem \ref{thm;range}. On the other hand, 
\begin{align*}
    K_0(x,y) &= \frac{d \log p_0(x)}{dx}\frac{d \log p_0(y)}{dy}K(x,y)+ \frac{d \log p_0(x)}{dx} \frac{\partial K(x,y)}{\partial y} + \frac{d \log p_0(y)}{dy} \frac{\partial K(x,y)}{\partial x} \\ &\qquad\qquad + \frac{\partial^2 K(x,y)}{\partial x \partial y}.
\end{align*}
Observe that $$\frac{d \log p_0(x)}{dx}=-x,\,\, \frac{\partial K(x,y)}{\partial y} = \sum_{n=0}^{\infty} \frac{k\rho^k}{k!} H_{k-1}(y)H_k(x),$$ and $$\frac{\partial^2 K(x,y)}{\partial x \partial y} = \sum_{n=0}^{\infty}\frac{k^2\rho^k}{k!}H_{k-1}(y)H_{k-1}(x),$$ where we used that $\frac{d H_{k}(x)}{dx}=k H_{n-1}(x).$ Then by using $H_{k+1}(x)=xH_{k}(x)-kH_{k-1}(x),$ \ it can be shown that 
$$K_0(x,y) = \frac{1}{\rho} \sum_{k=1}^{\infty} \lambda_k \gamma_k(x) \gamma_k(y),$$
where $\lambda_k = k \rho^k = k e^{k\log \rho}.$  Thus the result follows by applying Proposition~\ref{thm;range}.

\subsection{Proof of Remark \ref{Ex: Gaussian} (i)} \label{proof: remark}
Observe that 
\begin{align*}
K(x,y)&=\frac{1}{\sqrt{1-\rho^2}}\exp\left(-\frac{\rho^2(x^2+y^2)-2 \rho xy}{2(1-\rho^2)}\right)\\ & = \frac{1}{\sqrt{1-\rho^2}} \exp\left(\frac{-\rho}{2(1-\rho^2)}(x-y)^2\right) \exp\left(\frac{\rho}{2(1+\rho)} x^2\right) \exp\left(\frac{\rho}{2(1+\rho)} y^2\right).    
\end{align*}
Let $\bar{K}(x,y)=\exp\left(\frac{-\rho}{2(1-\rho^2)
}(x-y)^2\right).$  An application of the Mehler formula in \eqref{Eq:mehler} yields
\begin{align*}
\bar{K}(x,y)&=\sqrt{1-\rho^2}\sum^\infty_{k=0} \rho^k \gamma_k(x)\gamma_k(y)\exp\left(\frac{-\rho}{2(1+\rho)} x^2\right) \exp\left(\frac{-\rho}{2(1+\rho)} y^2\right)\\
&=\sum_{k=0}^{\infty} \lambda_k \psi_k(x)\psi_k(y),
\end{align*}
where $\lambda_k=(1-\rho)\rho^k,$ and $\psi_k(x)=\left(\frac{1+\rho}{1-\rho}\right)^{1/4}\gamma_k(x) \exp\left(\frac{-\rho}{2(1+\rho)}x^2\right).$ Note that  $(\psi_k)_{k=0}^{\infty}$ form an orthonormal basis with respect  $L^2(Q)$ with $Q$ being a Gaussian distribution with mean zero and variance $\frac{1+\rho}{1-\rho}.$ Then the RKHS associated with $\bar{K}$ can be defined as  
$$\h_{\bar{K}} := \left\{\sum_{k=0}^{\infty}a_k\psi_k(x) : \sum_{k=0}^{\infty} a_k^2 \lambda_k^{-1} < \infty \right\}.$$ Hence for any $f \in \h_{\bar{K}},$ the function $\bar{f}(x):= \exp\left(\frac{\rho}{2(1+\rho)}x^2\right)f(x)-a_0 \in R_1$ for any $\theta \leq \frac{1}{2}.$  Therefore,
$$S :=\left\{\exp\left(\frac{\rho}{2(1+\rho)}x^2\right)g_f(x)-a_0  :  g_f \in \h_{\bar{K}} \right\}\subseteq R_1,$$ for any $\theta \leq \frac{1}{2}.$

Let $\Tilde{K}(x,y) = \exp(-\sigma(x-y)^2).$ Then we will show that $\h_{\tilde{K}} \subseteq S$. To this end, we need to show that for any $f \in \h_{\tilde{K}}$ with $\E_{P_0}f=0,$ $\exists g_f \in \h_{\bar{K}}$ such that $f(x)= \exp\left(\frac{\rho}{2(1+\rho)}x^2\right)g_f(x)-a_0.$ Thus it is sufficient to show that $$g_f(x):=\exp\left(\frac{-\rho}{2(1+\rho)}x^2\right)\left(f(x)+a_0\right)\in \h_{\bar{K}}.$$ 
Observe that $$g_f(x)=\exp\left(\frac{-\rho}{2(1+\rho)}x^2\right)f(x)+a_0\exp\left(\frac{-\rho}{2(1+\rho)}x^2\right) =: g_1(x) + g_2(x).$$ We will show $g_1, g_2 \in \h_{\bar{K}}$ and therefore the claim follows that $\h_{\tilde{K}}\subseteq R_1$. For that matter, note that 

\begin{align*}
&|\hat{g}_1(\omega)|\\&=\left|\int_{-\infty}^{\infty} \hat{f}(t)\exp\left(\frac{-(\omega-t)^2(1+\rho)}{2\rho}\right) \, dt \right| \\
& \leq \int_{-\infty}^{\infty} |\hat{f}(t)| \exp\left(\frac{-(\omega-t)^2(1+\rho)}{2\rho}\right) \exp\left(\frac{-t^2}{8\sigma}\right)\exp\left(\frac{t^2}{8\sigma}\right) \, dt \\
& \leq \left(\int_{-\infty}^{\infty} |\hat{f}(t)|^2 \exp\left(\frac{t^2}{4\sigma}\right) \, dt \right)^{1/2} \left(\int_{-\infty}^{\infty} \exp\left(\frac{-(\omega-t)^2(1+\rho)}{\rho}\right) \exp\left(\frac{-t^2}{4\sigma}\right) \, dt \right)^{1/2} \\
& =\Vert f\Vert_{\h_{\tilde{K}}} \left(\int_{-\infty}^{\infty} \exp\left(\frac{-(\omega-t)^2(1+\rho)}{\rho}\right) \exp\left(\frac{-t^2}{4\sigma}\right) \, dt \right)^{1/2} \\
&\lesssim \sqrt{\exp \left(\frac{-(1+\rho)}{4\sigma(1+\rho)+\rho}\omega^2\right)}.
\end{align*}
Then 
\begin{align*}
    \norm{g_1}_{\h_{\bar{K}}}^2 &= \int_{-\infty}^{\infty} |\hat{g}_1(\omega)|^2 \exp\left(\frac{\omega^2(1-\rho^2)}{2\rho}\right) \, d\omega \\ & \lesssim  \int_{-\infty}^{\infty} \exp\left(-\omega^2\left(\frac{1+\rho}{4\sigma(1+\rho)+\rho}-\frac{1-\rho^2}{2\rho}\right)\right) \, d\omega < \infty,
\end{align*}
where the last inequality follows using $\sigma < \frac{\rho}{4(1-\rho)}$ to ensure the exponent argument is negative. Next, observe that $\hat{g}_2(\omega)=a_0 \exp\left(\frac{-\omega^2(1+\rho)}{2\rho}\right).$ Thus 
\begin{align*}
    \norm{g_2}_{\h_{\bar{K}}}^2 &= \int_{-\infty}^{\infty} |\hat{g}_2(\omega)|^2 \exp\left(\frac{\omega^2(1-\rho^2)}{2\rho}\right) \, d\omega  \\ 
    & = a^2_0 \int_{-\infty}^{\infty} \exp\left(-\omega^2 \left(\frac{1+\rho}{2\rho}-\frac{1-\rho^2}{2\rho}\right)\right) \,d\omega < \infty.
\end{align*}

\section*{Acknowledgments}
The authors thank the Editor-In-Chief, the  Associate Editor and the reviewers for their constructive comments that helped to improve the presentation. BKS thanks Zolt\'{a}n Szab\'{o} and Florian Kalinke for their comments on the proof of Theorem~\ref{thm:non-optimality-KSD}. OH and BKS are partially supported by the National Science Foundation (NSF) CAREER award DMS-1945396. KB is supported by the NSF grant DMS-2053918.

\bibliographystyle{plainnat} 
\bibliography{main} 

\section{Technical results}
In the following, we present technical results that are used to prove the main results of the paper. Unless
specified otherwise, the notation used in this section matches that of the main paper.
\begin{appxlem} \label{Lem:AP-AP0}
    Let $\h_K$ be an RKHS associated with kernel $K$. Define $\A_Q:\h_K \to \h_K$, $$\A_Q:=\int_{\X} K(\cdot,x)\hhtens K(\cdot,x) \, dQ(x),$$ where $Q$ is a probability measure that can be either $P_0$ or $P,$ $$\Nol:= \emph{Tr}(\SgL\A_{P_0}\SgL),\  \text{and}  \hspace{2mm} \Ntl:= \norm{\SgL\A_{P_0}\SgL}_{\hs},$$ 
    where $\A_{P,\lambda}:= \A_{P} + \lambda \Id.$ Suppose $$\int_{\X} K(x,x) \, dQ(x) < \infty, \ \text{and} \int_{\X} [K(x,x)]^v \, dP_0(x) < \kappa^v,$$ for some $\kappa>0$ and $v\geq2$. Then the following hold:
    \begin{enumerate}
    \item $\norm{\SgLP\A_{P}\SgLP}_{\hs}^2 \leq \frac{\emph{Tr}(\A_P)}{\lambda};$ 
     \item If $4\Cl\norm{u}_{\Lp}^2 \leq 1,$  then 
     \begin{align*}
       &(a) \ \norm{\SgLP\A_{P}\SgLP}_{\hs}^2 \leq 8 \Ntlsq +2, \\
       &(b) \ \emph{Tr}(\SgLP \A_{P} \SgLP)  \leq 2\Nol + 2\Cl\norm{u}_{\Lp}, 
     \end{align*}
where $u:=\frac{dP}{dP_0}-1$ and $$\Cl =
\left\{
	\begin{array}{ll}
		\Nol \sup_{i} \norm{\phi_i}^2_{\infty},   &  \ \ \sup_i \norm{\phi_i}^2_{\infty} < \infty \\
		\left(\Nol\right)^{\frac{v-2}{2(v-1)}}\left(\frac{\kappa}{\lambda}\right)^{\frac{v}{v-1}},  & \ \  \text{otherwise}
	\end{array}
\right.,$$
with $\phi_i:=\frac{\alpha_i}{\sqrt{\lambda_i}}$ and $(\lambda_i,\alpha_i)_{i}$ being the  orthonormal eigensystem of $\A_{P_0}.$
    \end{enumerate}
\end{appxlem}

\begin{proof}
Observe that $\A_{P_0}=\id^*\id$, where $\id : \h_{K} \to \Lp$, $f \mapsto [f]_{\sim}$ and $\id^* : \Lp \to \h_{K}$, $f \mapsto \int K(\cdot,x)f(x)\,dP_0(x)$.  \\ Then it can verified that $\ep_{P_0}:=\id\id^* : \Lp \to \Lp$, $f \mapsto \int K(\cdot,x)f(x)\,dP_0(x).$
Note that \begin{align*}
    \text{Tr}(\ep_{P_0})&=\text{Tr}(\A_{P_0}) = \int_{x} \text{Tr}(K(\cdot,x)\hhtens K(\cdot,x) )\, dP_0(x) = \int_{x} \inner{K(\cdot,x)}{K(\cdot,x)}_{\h_K}\, dP_0(x) \\
    & =\int_{\X} K(x,x) \, dP_0(x) < \infty,
\end{align*}
which implies $\ep_{P_0}$ and $\A_{P_0}$ (similarly $\A_{P}$) are trace class operators. Let $(\lambda_i,\Tilde{\phi_i})_i$ be the eigenvalues and eigenfunctions of $\ep_{P_0}.$ Since $\ep_{P_0}=\id\id^*$ and $\A_{P_0}=\id^*\id$, we have $\id\id^*\Tilde{\phi}_i=\lambda_i\Tilde{\phi}_i$ which implies $\id^*\id\left(\id^*\Tilde{\phi}_i\right)=\lambda_i\left(\id^*\Tilde{\phi}_i\right)$, i.e., $\A_{P_0}\alpha_i=\lambda_i\alpha_i$, where $\alpha_i:=\id^*\Tilde{\phi}_i/\sqrt{\lambda_i}$. Note that $(\alpha_i)_i$ form an orthonormal system in $\h_K$. Define $\phi_i=\frac{\id^*\Tilde{\phi_i}}{\lambda_i}$, thus $\alpha_i=\sqrt{\lambda_i}\phi_i.$ Based on these notations, we present the proofs of the desired results as follows:

\emph{(i)} Observe that
\begin{align*}
    \norm{\SgLP\A_{P}\SgLP}_{\hs}^2 \leq \norm{\SgLP\A_{P}\SgLP}_{\op} \text{Tr}\left(\SgLP\A_{P}\SgLP\right)  \leq \frac{\text{Tr}(\A_P)}{\lambda}. 
\end{align*}

\emph{(ii)-(a)} Let $$\B = \int_{\X} K(\cdot,x) \hhtens K(\cdot,x) \ u(x) \, dP_0(x).$$
Then it is clear that $\A_P=\A_{P_0}+\B$ and we have 
\begin{align*}
    &\norm{\SgLP \A_{P} \SgLP}_{\hs}^2 \leq \norm{\SgLP \A_{P_0,\lambda}^{1/2}}_{\op}^4 \norm{\SgL \A_{P} \SgL}_{\hs}^2 \\
    & \leq \norm{\A_{P_0,\lambda}^{1/2} \A_{P,\lambda}^{-1}\A_{P_0,\lambda}^{1/2} }_{\op}^2 \norm{\SgL \A_{P} \SgL}_{\hs}^2 \\
    & \leq \norm{\left(\Id-\SgL(\A_{P_0}-\A_{P})\SgL\right)^{-1}}_{\op}^2 \norm{\SgL \A_{P} \SgL}_{\hs}^2 \\
    & \stackrel{(*)}{\leq} \left(1-\norm{\SgL \B \SgL}_{\op}\right)^{-2}\left( 2\Ntlsq + 2\norm{\SgL\B\SgL}_{\hs}^2\right) \\
    &\stackrel{(\dag)}{\leq} \left(1-\sqrt{\Cl}\norm{u}_{\Lp}\right)^{-2}\left(2 \Ntlsq + 2\Cl \norm{u}_{\Lp}^2\right) \\
    & \leq 8\Ntlsq+2,
\end{align*}
where $(*)$ follows when $\norm{\SgL\B\SgL}_{\op} < 1.$ By using $$\norm{\SgL\B\SgL}_{\op}\leq \norm{\SgL\B\SgL}_{\hs},$$ $(\dag)$ follows if
\begin{equation}
 \norm{\SgL\B\SgL}_{\hs} \leq \sqrt{\Cl}\norm{u}_{\Lp}, \label{eq:hs-bound} 
\end{equation}
which we show 
below for the two cases of $\Cl.$ \vspace{1mm}\\
\emph{Case 1:} $\sup_{i} \norm{\phi_i}_{\infty}^2 < \infty$. Then we have, 
\begin{align*}
    &\norm{\SgL \B \SgL}_{\hs}^2 = \text{Tr}\left(\SL^{-1} \B \SL^{-1} \B\right) \\
    &= \hspace{-1mm}\int_{\X} \int_{\X} \text{Tr}\left[\SL^{-1}K(\cdot,x)\hhtens K(\cdot,x) \SL^{-1} K(\cdot,y)\hhtens K(\cdot,y) \right] u(x) u(y) \, d\PQ(y)\,d\PQ(x)\\
    &\stackrel{(\ddagger)}{=} \sum_{i,j} \left(\frac{\lambda_i}{\lambda_i+\lambda}\right) \left(\frac{\lambda_j}{\lambda_j+\lambda}\right) \left[\int_{\X}\phi_i(x)\phi_j(x) u(x) \, d\PQ(x)\right]^2 \\
    & \leq \sum_{i,j} \frac{\lambda_i}{\lambda_i+\lambda} \inner{\phi_j}{\phi_i u}_{\Lp}^2 
     \leq \sum_{i} \frac{\lambda_i}{\lambda_i+\lambda} \norm{\phi_i u}_{\Lp}^2 
     \leq  \Nol \sup_i \norm{\phi_i}_{\infty}^2 \norm{u}_{\Lp}^2,
\end{align*}
where $(\ddagger)$ follows from
\begin{align}
&\text{Tr}\left[\SL^{-1}K(\cdot,x)\hhtens K(\cdot,x) \SL^{-1} K(\cdot,y)\hhtens K(\cdot,y) \right] \nonumber\\
    &= \inner{K(\cdot,y)}{\SL^{-1}K(\cdot,x)}_{\h_K}^2 \nonumber \\
    &= \left\langle K(\cdot,y),\sum_i \frac{1}{\lambda_i+\lambda} \inner{K(\cdot,x)}{\sqrt{\lambda_i}\phi_i}_{\h_K}\sqrt{\lambda_i}\phi_i\right.\nonumber\\
    &\qquad\qquad\left.+\frac{1}{\lambda}\left(K(\cdot,x)-\sum_{i}\inner{K(\cdot,x)}{\sqrt{\lambda_i}\phi_i}_{\h_K}\sqrt{\lambda_i}\phi_i\right)\right\rangle^2_{\h_K} \nonumber \\
    & = \left[\sum_i \frac{\lambda_i}{\lambda_i+\lambda}\phi_i(x)\phi_i(y)+\frac{1}{\lambda}\left(\inner{K(\cdot,x)}{K(\cdot,y)}_{\h_K}-\sum_i\lambda_i\phi_i(x)\phi_i(y)\right) \right]^2 \nonumber \\ &\stackrel{(\star)}{=}\left[\sum_i \frac{\lambda_i}{\lambda_i+\lambda}\phi_i(x)\phi_i(y)\right]^2, \label{eq:(i)}
\end{align}
where in $(\star)$ we used $\inner{K(\cdot,x)}{K(\cdot,y)}_{\h}=\sum_i\lambda_i\phi_i(x)\phi_i(y)$ which follows by Mercer's theorem (see \citealt[Lemma 2.6]{Steinwart2012MercersTO}). 

\noindent \emph{Case 2:} Suppose $\sup_{i} \norm{\phi_i}_{\infty}^2$ is not finite. Then 
\begin{align*}
     &\norm{\SgL \B \SgL}_{\hs}^2 = \text{Tr}\left(\SL^{-1} \B \SL^{-1} \B\right) \\
    &= \hspace{-1mm}\int_{\X} \int_{\X} \text{Tr}\left[\SL^{-1}K(\cdot,x)\hhtens K(\cdot,x) \SL^{-1} K(\cdot,y)\hhtens K(\cdot,y) \right] u(x) u(y) \, d\PQ(y)\,d\PQ(x) \\
    &= \int_{\X}\int_{\X} \inner{K(\cdot,y)}{\SL^{-1}K(\cdot,x)}_{\h_K}^2 u(x)u(y) \, dP_0(x) \,dP_0(y) \\
    & \leq \left(\int_{\X}\int_{\X} \inner{K(\cdot,y)}{\SL^{-1}K(\cdot,x)}_{\h_K}^4 \, dP_0(x) \,dP_0(y)\right)^{1/2} \norm{u}_{\Lp}^2 \\ 
    &=  \left(\int_{\X}\int_{\X}\inner{K(\cdot,y)}{\SL^{-1}K(\cdot,x)}_{\h_K}^{\frac{2v}{v-1}} \inner{K(\cdot,y)}{\SL^{-1}K(\cdot,x)}_{\h_K}^{\frac{2(v-2)}{v-1}}dP_0(x) \,dP_0(y)\right)^{1/2} \\ &\qquad\qquad\times\norm{u}_{\Lp}^2 \\
    &\leq \hspace{-1mm} \left( \int_{\X}\int_{\X} 
 \hspace{-1mm}\norm{\SL^{-1} K(\cdot,x)}_{\h_K}^{\frac{2v}{v-1}} \hspace{-1mm} \norm{ K(\cdot,y)}_{\h_K}^{\frac{2v}{v-1}}\hspace{-1mm} \inner{K(\cdot,y)}{\SL^{-1}K(\cdot,x)}_{\h_K}^{\frac{2(v-2)}{v-1}} \hspace{-2mm}dP_0(x) \,dP_0(y)\right)^{1/2}\\
    &\qquad\qquad\times\norm{u}_{\Lp}^2 \\
    &\leq  \frac{\norm{u}_{\Lp}^2}{\lambda^{\frac{v}{v-1}}} \left( \int_{\X}\int_{\X}[K(x,x)K(y,y)]^{\frac{v}{v-1}} \inner{K(\cdot,y)}{\SL^{-1}K(\cdot,x)}_{\h_K}^{\frac{2(v-2)}{v-1}}\hspace{-2mm}dP_0(x) \,dP_0(y)\right)^{1/2}\\
&\leq \frac{\norm{u}_{\Lp}^2}{\lambda^{\frac{v}{v-1}}} \left( \int_{\X}\int_{\X}[K(x,x)K(y,y)]^{v} dP_0(x) \,dP_0(y)\right)^{\frac{1}{2(v-1)}}\\
&\qquad\qquad\times\left(\E_{P_0}\inner{K(\cdot,Y)}{\SL^{-1}K(\cdot,X)}_{\h_K}^{2}\right)^{\frac{v-2}{2(v-1)}}\\
&\leq \norm{u}_{\Lp}^2 \left(\frac{\kappa}{\lambda}\right)^{\frac{v}{v-1}} \\ &\qquad\qquad\times \left(\E_{P_0}\text{Tr}\left[\SL^{-1}K(\cdot,X)\hhtens K(\cdot,X) \SL^{-1} K(\cdot,Y)\hhtens K(\cdot,Y) \right] \right)^{\frac{v-2}{2(v-1)}} \\ 
&\leq \norm{u}_{\Lp}^2 \left(\frac{\kappa}{\lambda}\right)^{\frac{v}{v-1}} \left(\text{Tr}({\SgL \A_{P_0} \SgL})\right)^{\frac{v-2}{2(v-1)}} =  \Cl \norm{u}_{\Lp}^2.
\end{align*}

\emph{(ii)-(b)} Note that
\begin{align*}
    \text{Tr}(\SgLP \A_{P} \SgLP) &\leq \norm{\SgLP \A_{P_0,\lambda}^{1/2}}_{\op}^2 \text{Tr}(\SgL \A_{P} \SgL) \\
    & \leq \left(1-\norm{\SgL \B \SgL}_{\op}\right)^{-1}\hspace{-2mm}\left( \Nol + \text{Tr}(\SgL\B\SgL)\right) \\
    &\stackrel{(*)}{\leq} \left(1-\sqrt{\Cl}\norm{u}_{\Lp}\right)^{-1}\left( \Nol + \Cl\norm{u}_{\Lp}\right) \\
    & \stackrel{(\dag)}{\leq} 2\Nol + 2\Cl\norm{u}_{\Lp},
\end{align*}
where $(*)$ follows using $\norm{\SgL\B\SgL}_{\op}  \leq \sqrt{\Cl}\norm{u}_{\Lp}$ as in \emph{(ii)-(a)}, and if $$\text{Tr}(\SgL \B \SgL) \leq \Cl \norm{u}_{\Lp}$$ holds, which we show below.

\noindent\emph{Case 1:} $\sup_{i} \norm{\phi_i}_{\infty}^2 < \infty$. Then we have, 

\begin{align*}
    \text{Tr}(\SgL \B \SgL) &= \int_{\X}\text{Tr}(\SgL K(\cdot,x)\hhtens K(\cdot,x) \SgL) u(x) \, dP_0(x) \\
    &\stackrel{(*)}{=} \sum_{i}\frac{\lambda_i}{\lambda_i+\lambda}\int_{\X}\phi_i^2(x) u(x) \, dP_0(x) \\
    & \leq \Nol \norm{u}_{\Lp} \sup_{i} \norm{\phi_i}_{\infty}^2=\Cl \norm{u}_{\Lp},
\end{align*}
where $(*)$ follows similar to the proof of \eqref{eq:(i)}.\\
\emph{Case 2:} Suppose $\sup_{i} \norm{\phi_i}_{\infty}^2$ is not finite. Then 
\begin{align*}
    &\text{Tr}(\SgL \B \SgL) \\ &= \int_{\X} \inner{\SL^{-1}K(\cdot,x)}{K(\cdot,x)}_{\h_K} u(x) \, dP_0(x)\\
    &\leq \norm{u}_{\Lp} \left(\int_{\X} \inner{\SL^{-1}K(\cdot,x)}{K(\cdot,x)}_{\h_K}^2 \, dP_0(x)\right)^{1/2} \\ 
    & \leq \norm{u}_{\Lp}\left(\int_{\X}\inner{\SL^{-1}K(\cdot,x)}{K(\cdot,x)}_{\h_K}^v \, dP_0(x)\right)^{\frac{1}{2(v-1)} } \\&\qquad\qquad\times\left(\E_{P_0}\inner{\SL^{-1}K(\cdot,x)}{K(\cdot,x)}_{\h_K}\right)^{\frac{v-2}{2(v-1)}} \\
    &\leq \norm{u}_{\Lp}\left(\frac{\kappa}{\lambda}\right)^{\frac{v}{2(v-1)}} \left(\text{Tr}(\SgL\A_{P_0}\SgL)\right)^{\frac{v-2}{2(v-1)}} \leq  \Cl \norm{u}_{\Lp}
\end{align*}
and the result follows.
\end{proof}

\begin{appxlem} \label{Lem: SP-SP0}
    Let $\statt^{P}:= \norm{g^{1/2}_{\lambda}(\A_{P}) \Psi_P}_{\h_{K_0}}^2$ and $\statt^{P_0}:= \norm{\SgL \Psi_P}_{\h_{K_0}}^2,$ where $\A_{P_0,\lambda}=\A_{P_0}+\lambda \Id$ and $g_\lambda$ satisfies $(E_1)$--$(E_4)$. Then 
    $$\statt^{P} \geq \frac{C_4 \ \statt^{P_0}}{1+\sqrt{\Cl}\norm{u}_{\Lp}},$$
 where $\Cl$ is defined in Lemma \ref{Lem:AP-AP0} with $K=K_0$. 
\end{appxlem}

\begin{proof}
    By applying Lemma \ref{lemma: bounds for g}(ii), we have \begin{align*}
        \statt^{P} &\stackrel{(*)}{\geq} \norm{\SgLP\gsP}_{\opS}^{-2}\norm{\SgLP \Psi_P}_{\h_{K_0}}^2  \geq C_4 \norm{\SgL\A_{P,\lambda}^{1/2}}_{\opS}^{-2} \statt^{P_0},
    \end{align*}
    where $(*)$ follows using $$\norm{\SgLP \Psi_P}_{\h_{K_0}}^2=\norm{\SgLP g^{-1/2}_{\lambda}(\A_{P}) g^{1/2}_{\lambda}(\A_{P})\Psi_P}_{\h_{K_0}}^2 \leq \norm{\SgLP\gsP}_{\opS}^{2} \statt^{P} $$
Next we upper bound $\norm{\SgL\A_{P,\lambda}^{1/2}}_{\opS}^{2}$ as shown below.  Let $\B= \A_P-\A_{P_0}$. Then 
\begin{align*}
\norm{\SgL\A_{P,\lambda}^{1/2}}_{\opS}^{2} &=\norm{\SgL \A_{P,\lambda} \SgL}_{\opS}=\norm{\Id + \SgL \B \SgL}_{\opS} \\
    & \leq 1 + \norm{\SgL \B \SgL}_{\opS} \leq 1 + \norm{\SgL \B \SgL}_{\hsS}\\
    &\stackrel{(*)}{\leq} 1+\sqrt{\Cl}\norm{u}_{\Lp},
\end{align*}
where $(*)$ follows from \eqref{eq:hs-bound}.
\end{proof}

\begin{appxlem} \label{lem:bound-var}
 Let $\zeta= \norm{\gSh \Psi_P}_{\h_{K_0}}^2,$ $\W=\hat{\A}_{P,\lambda}^{-1/2}\A_{P,\lambda}^{1/2}$ and $\V=\hat{\A}_{P,\lambda}^{-1/2}\A_{P_0,\lambda}^{1/2}$, where $g_\lambda$ satisfies $(E_1)$--$(E_4)$. Then
    \begin{enumerate}
    \item $$\E\left((\stat^{P}-\zeta)^2 | \mathbb Z_{n_2}\right) \leq c\left(\frac{\norm{\W}_{\opS}^{2}\zeta 
 \  }{n_1}+\frac{\norm{\W}_{\opS}^4 D_{\lambda}}{n_1^2}\right),$$
 where $c>0$ is a constant, $D_{\lambda} =
\left\{
	\begin{array}{ll}
		\Ntlsq ,   &  \ \ \ 4\Cl\norm{u}_{\Lp}^2 \leq 1 \\
		\lambda^{-1}, & \ \  \text{otherwise}
	\end{array}
\right.,$ and $\Cl$ is defined in Lemma \ref{Lem:AP-AP0} with $K=K_0$;
\item $$\E\left((\stat^{P}-\zeta)^2 | \mathbb Z_{n_2}\right) \leq \tilde{c}\left(\frac{\norm{\V}_{\opS}^{2}\zeta 
 \  }{n_1}+\frac{\norm{\V}_{\opS}^4 \Ntlsq}{n_1^2}\right),$$ if $4\Cl \norm{u}^2_{\Lp} < 1,$
 where $\tilde{c}>0$ is a constant.    \end{enumerate}
\end{appxlem}

\begin{proof}
    Let $a(x)=\B\SgLP(K_0(\cdot,x)-\Psi_P)$ and $\B=\gSh\A_{P,\lambda}^{1/2}$. Then  
    \begin{align*}
        \stat^P &= \frac{1}{n_1(n_1-1)}\sum_{i \neq j}\inner{a(X_i)}{a(X_j)}_{\h_{K_0}} + \frac{2}{n_1} \sum_{i=1}^{n_1} \inner{a(X_i)}{\B\SgLP \Psi_P}_{\h_{K_0}} + \zeta.
    \end{align*}
Thus we have $\stat^P-\zeta={\circled{\tiny{1}}}+{\circled{\tiny{2}}},$ where ${\circled{\tiny{1}}}:= \frac{1}{n_1(n_1-1)}\sum_{i \neq j}\inner{a(X_i)}{a(X_j)}_{\h_{K_0}}$ and \\${\circled{\tiny{2}}}:= \frac{2}{n_1} \sum_{i=1}^{n_1}\inner{a(X_i)}{\B\SgLP \Psi_P}_{\h_{K_0}}.$ 
Next, we bound each term as follows: \\
\emph{(i)} 
\begin{align*}
    &\E\left(\circled{\tiny{1}}^2|\mathbb Z_{n_2}\right) \\&\stackrel{(*)}{=} \frac{2}{n_1^2(n_1-1)^2}\sum_{i \neq j} \E \inner{a(X_i)}{a(X_j)}_{\h_{K_0}}^2 \\
    &\leq \frac{4}{n_1^2} \norm{\B\SgLP \E_{P}[(K_0(\cdot,X)-\Psi_P)\htens(K_0(\cdot,X)-\Psi_P)]\SgLP\B^*}^2_{\hsS} \\
    &\stackrel{(\dag)}{\leq} \frac{4}{n_1^2}(C_1+C_2)^2 \norm{\W}_{\opS}^4 \norm{\SgLP \A_P \SgLP}_{\hsS}^2, 
\end{align*}
where $(*)$ follows from \citep[Lemma A.3$(ii)$]{twosampletest} and in $(\dag)$ we used Lemma \ref{lemma: bounds for g}(i) and that 
\begin{align*}
 \E_{P}((K_0(\cdot,X_i)-\Psi_P)\htens(K_0(\cdot,X_i)-\Psi_P)) = \A_P - \Psi_P \htens \Psi_P \preccurlyeq \A_P.   
\end{align*}
For term \circled{\tiny{2}} we have 
\begin{align*}
&\E\left(\circled{\tiny{2}}^2|\mathbb Z_{n_2}\right) \leq \frac{4}{n_1^2} \sum_{i=1}^{n_1} \E\inner{a(X_i)}{\B\SgLP \Psi_P}_{\h_{K_0}}^2 \\ 
    &=  \frac{4}{n_1^2} \sum_{i=1}^{n_1}\E \inner{ a(X_i)\htens a(X_i)}{ \B\SgLP \Psi_P \htens \Psi_P \SgLP\B^*}_{\hsS} \\
    & \leq \frac{4}{n_1} \text{Tr}\left(\SgLP\A_P\SgLP\B^*\B \SgLP \Psi_P \htens \Psi_P \SgLP \B^*\B\right) \\
    & \leq \frac{4}{n_1} \norm{\SgLP\A_P\SgLP}_{\opS} (C_1+C_2)\norm{\W}_{\opS}^2 \\ &\qquad\qquad \times \text{Tr}\left(\B\SgLP \Psi_P \htens \Psi_P \SgLP\B^* \right) \\
    &= \frac{4}{n_1} (C_1+C_2)\norm{\W}_{\op}^2 \zeta.
\end{align*}
Using the above bounds and Lemma \ref{Lem:AP-AP0} yields the desired result.  

\noindent \emph{(ii)} Consider
\begin{align*}
    &\E\left(\circled{\tiny{1}}^2|\mathbb Z_{n_2}\right) =\frac{2}{n_1^2(n_1-1)^2}\sum_{i \neq j} \E \inner{a(X_i)}{a(X_j)}_{\h_{K_0}}^2 \\
    &\leq \frac{4}{n_1^2} \norm{\gSh \E_{P}[(K_0(\cdot,X)-\Psi_P)\htens(K_0(\cdot,X)-\Psi_P)]\gSh}^2_{\hsS} \\
    &\leq \frac{4}{n_1^2}(C_1+C_2)^2 \norm{\V}_{\opS}^4 \norm{\SgL \A_P \SgL}_{\hsS}^2 \\
    &\leq \frac{4\norm{\V}_{\opS}^4(C_1+C_2)^{2}}{n_1^2}  \left(2\norm{\SgL \A_{P_0} \SgL}_{\hsS}^2\right.\\
    &\qquad\qquad\qquad\qquad\left.+2\norm{\SgL (\A_P-A_{P_0})\SgL}_{\hsS}^2 \right) \\
    &\stackrel{(*)}{\leq} \frac{4\norm{\V}_{\opS}^4(C_1+C_2)^{2}}{n_1^2}  \left(2\Ntlsq+2\Cl \norm{u}_{\Lp}^2\right) \\
   & \lesssim \frac{4\norm{\V}_{\opS}^4\Ntlsq}{n_1^2},  
\end{align*}
where $(*)$ follow from \eqref{eq:hs-bound} and  
\begin{align*}
&\E\left(\circled{\tiny{2}}^2|\mathbb Z_{n_2}\right) 
    \leq \frac{4}{n_1} \text{Tr}\left(\SgLP\A_P\SgLP\B^*\B \SgLP \Psi_P \htens \Psi_P \SgLP \B^*\B\right) \\
    & \leq \frac{4}{n_1} \norm{\SgL\A_P\SgL}_{\opS} (C_1+C_2)\norm{\V}_{\opS}^2\\&\qquad\qquad\times\text{Tr}\left(\B\SgLP \Psi_P \htens \Psi_P \SgLP\B^* \right) \\
    &\leq \frac{4\norm{\V}_{\op}^2 \zeta(C_1+C_2)}{n_1} \\&\qquad\qquad \times \left(\norm{\SgL\A_{P_0}\SgL}_{\opS}+\norm{\SgL(\A_P-A_{P_0})\SgL}_{\opS}\right) \\
    &\stackrel{(\dag)}{\leq}\frac{4\norm{\V}_{\op}^2 \zeta(C_1+C_2)}{n_1} \left(1+\sqrt{\Cl}\norm{u}_{\Lp}\right) \lesssim \frac{\norm{\V}_{\op}^2 \zeta}{n_1},
\end{align*}
where we used \eqref{eq:hs-bound} in ($\dag$).
\end{proof}

\begin{appxlem} \label{lem:bound-op-norm-1}
Suppose there exist some $c>0$ and $\kappa>0$ such that for any $r>2,$ $$\int_{\X} [K(x,x)]^r \, dP(x) \leq c r!  \kappa^r,$$ and one of the following hold: 
\begin{enumerate}
    \item $\lambda \geq \frac{16\kappa}{\sqrt{n}}\log \frac{2}{\delta} \max\{\sqrt{c},1\}.$ 
    \item \hspace{-1mm}$4\Cl\norm{u}^2_{\Lp} < 1,$ $n>4,$ $\lambda \geq \frac{d(\log \frac{2}{\delta})(\log n)}{n} \left(\Nol+\Cl\norm{u}_{\Lp}\right),$ and \\ $\E_{P_0}[K(X,X)^v] < \tilde{\kappa}^v,$ for some $d>0$, $v>2$ and $\Cl$ is defined in Lemma \ref{Lem:AP-AP0} (with $\kappa$ replaced by $\tilde{\kappa}$).
    \item $\sup_{x}K(x,x) \leq \kappa$ and $\frac{d_1\log \frac{d_2 n}{\delta}}{n}\leq \lambda \leq \norm{\A_P}_{\op},$ for some $d_1>0$ and $d_2>0.$
\end{enumerate}
Then for $0<\delta\leq 1$,
\begin{enumerate}
    \item $P\left(\norm{\SgLP(\A_P-\hat{\A}_P)\SgLP}_{\op}\leq \frac{1}{2}\right)\geq 1-\delta$;
    \item $P\left(\norm{\A_{P,\lambda}^{1/2}\hat{\A}_{P,\lambda}^{-1/2}}_{\op}\leq \sqrt{2}\right) \geq 1-\delta$;
    \item $P\left(\norm{\A_{P,\lambda}^{-1/2}\hat{\A}_{P,\lambda}^{1/2}}_{\op}\leq \sqrt{\frac{3}{2}}\right) \geq 1-\delta,$
\end{enumerate}
where $\A_{P}$ is defined in Lemma \ref{Lem:AP-AP0}.
\end{appxlem}
\begin{proof}
    Let $\xi_i = \SgLP \left((K(\cdot,Z_i)\hhtens K(\cdot,Z_i))-\A_P\right)\SgLP$. \\Then $\frac{1}{n}\sum_{i=1}^n \xi_i = \SgLP (\hat{\A}_P- \A_P) \SgLP$ and $\E_P(\xi_i)=0.$
    The main idea is to apply Bernstein's inequality for Hilbert-valued random elements as stated in \citet[Theorem D.1]{kpca}. \\
    Let $\theta^2= 2c^{1-q}\left(\frac{d_q}{\lambda}\right)^{2-q}\left[\text{Tr}(\SgLP \A_P \SgLP)\right]^{q},$ $B=\frac{d_q}{\lambda}$ and $d_q=\frac{2\kappa}{1-q}$, for some $0\leq q< 1.$ Then we bound $\E\norm{\xi_1}_{\hs}^r$ as follows. 
    \begin{align*}
    &\E\norm{\xi_1}_{\hs}^r \stackrel{(*)}{\leq} \E \left(\norm{\SgLP\left(K(\cdot,Z_1) \hhtens K(\cdot,Z_1)\right) \SgLP}^{2}_{\hs}\right)^\frac{r}{2} \\
        &= \E \left(\inner{\SgLP K(\cdot,Z_1)}{\SgLP K(\cdot,Z_1)}^2_{\h_K}\right)^\frac{r}{2} \\
        &= \E \left(\text{Tr}\left(\SLP^{-1} K(\cdot,Z_1)\hhtens K(\cdot,Z_1) \SLP^{-1} K(\cdot,Z_1)\hhtens K(\cdot,Z_1)\right) \right)^{\frac{r}{2}-q + q} \\ 
        &\stackrel{(\dag)}{\leq}  \left[\E\left(\inner{\SLP^{-1} K(\cdot,Z_1)}{K(\cdot,Z_1)}_{\h_K}^2\right)^{\frac{2(r/2-q)}{1-q}}\right]^{\frac{1-q}{2}}  \\ & \quad \times\left[\E\left(\text{Tr}\left(\SLP^{-1} K(\cdot,Z_1)\hhtens K(\cdot,Z_1) \SLP^{-1} K(\cdot,Z_1)\hhtens K(\cdot,Z_1)\right)\right)^{\frac{2q}{1+q}}\right]^{\frac{1+q}{2}} \\
        & \leq \frac{1}{\lambda^{r-2q}} \left[\E (K(Z_1,Z_1))^{\frac{2(r-2q)}{1-q}}\right]^{\frac{1-q}{2}}  \\ & \times \left[\E\left(\norm{\SLP^{-1/2}K(\cdot,Z_1)}_{\h_{K}}^2 \text{Tr}\left( \SLP^{-1} K(\cdot,Z_1)\hhtens K(\cdot,Z_1)\right)\right)^{\frac{2q}{1+q}}\right]^{\frac{1+q}{2}} \\
        &\stackrel{(\ddag)}{\leq} \frac{1}{\lambda^{r-q}} \left[\E (K(Z_1,Z_1))^{\frac{2(r-2q)}{1-q}}\right]^{\frac{1-q}{2}} \left[\E (K(Z_1,Z_1))^{\frac{2q}{1-q}}\right]^{\frac{1-q}{2}} \left[\text{Tr}(\SgLP \A_P \SgLP)\right]^{q}\\ & \stackrel{(**)}{\leq} \frac{r!}{2}\theta^2 B^{r-2},
    \end{align*}
where $(*)$ follows since $\A_P$ is positive definite operator. ($\dag$) and ($\ddag$) follow by applying Holder's inequality. $(**)$ follows using 
the condition $\E_{P}[K(Z,Z)^r] \leq c r! \kappa^r$ and that for any $k>1,\ l>0$, we have $\left((lk)!\right)^{\frac{1}{k}} \leq l! k^l.$ 

Thus using $\norm{\frac{1}{n}\sum_{i=1}^n \xi_i}_{\op}\leq \norm{\frac{1}{n}\sum_{i=1}^n \xi_i}_{\hs}$ and applying \citet[Theorem D.1]{kpca} yields that with probability at least $1-\delta$, 
$$\norm{ \SgLP (\hat{\A}_P- \A_P) \SgLP}_{\op} \leq \frac{d_q \log\frac{2}{\delta}}{n \lambda}+\sqrt{\frac{4d_q^{2-q}c^{1-q}\left[\text{Tr}(\SgLP \A_P \SgLP)\right]^{q} \log \frac{2}{\delta}}{n\lambda^{2-q}}},$$ holds for any $0<q\leq 1.$ Setting $q=0$, yields that 
$$\norm{ \SgLP (\hat{\A}_P- \A_P) \SgLP}_{\op} \leq \frac{2\kappa \log\frac{2}{\delta}}{n \lambda}+\sqrt{\frac{16\kappa^2 c \log \frac{2}{\delta}}{n\lambda^{2}}},$$
which under condition $(i)$ implies the desired bound in $(a).$ 

Suppose $q\neq 0$ and  let $\epsilon=1-q$, then we have 

\begin{align*}
 &\norm{ \SgLP (\hat{\A}_P- \A_P) \SgLP}_{\op} \\&\qquad\leq \frac{2\kappa \log\frac{2}{\delta}}{\epsilon n \lambda}+\sqrt{\frac{4(2\kappa)^{1+\epsilon}c^{\epsilon}\left[2\Nol+2\Cl\norm{u}_{\Lp}\right]^{1-\epsilon} \log \frac{2}{\delta}}{n\epsilon^{1+\epsilon}\lambda^{1+\epsilon}}},   
\end{align*}

where we used Lemma~\ref{Lem:AP-AP0}\emph{(ii)-}$b$. The above bound implies the desired result in $(a)$ when $$\lambda \geq \frac{256\kappa\max\{c,1\}\left(\Nol+\Cl\norm{u}_{\Lp}\right)\log \frac{2}{\delta}}{\epsilon n^{\frac{1}{1+\epsilon}}},$$
then by setting $\epsilon=\frac{1}{\log n}$ and using $n^\frac{\log n}{1+\log n} \geq \frac{n}{e}$ yields the condition on $\lambda$ in $(ii).$ 

For \emph{$(iii)$}, when the kernel is bounded, we use Tropp's inequality for operators as stated in \citet[Theorem D.2]{kpca} instead of Bernstein's inequality. 
Observe that $$\norm{\xi_i(Z)}_{\op}= \norm{\SgLP K(\cdot,Z)}_{\h_K} \leq \frac{\kappa}{\lambda}$$
and
\begin{align*}
     \E\left(\xi-\E(\xi)\right)^2 &\preccurlyeq \E(\xi^2) \\ &=\E\left(\A_{P,\lambda}^{-1/2}(K(\cdot,Z_i) \hhtens K(\cdot,Z_i))\A_{P,\lambda}^{-1}(K(\cdot,Z_i) \hhtens K(\cdot,Z_i))\A_{P,\lambda}^{-1/2}\right) \\ 
     &\preccurlyeq \sup_Z \norm{\xi(Z)}_{\op} \E(\A_{P,\lambda}^{-1/2}(K(\cdot,Z_i) \hhtens K(\cdot,Z_i))\A_{P,\lambda}^{-1/2}) \\
     & \preccurlyeq \frac{\kappa}{\lambda} \A_{P,\lambda}^{-1/2} \A_{P} \A_{P,\lambda}^{-1/2} := S,
\end{align*}
where
$$\norm{S}_{\op} \leq \frac{\kappa}{\lambda}:= \sigma^2,$$
and 
\begin{align*}
    d:= \frac{\text{Tr}(S)}{\norm{S}_{\op}} &= \frac{\text{Tr}( \A_{P,\lambda}^{-1/2} \A_{P} \A_{P,\lambda}^{-1/2})}{\norm{ \A_{P,\lambda}^{-1/2} \A_{P} \A_{P,\lambda}^{-1/2}}_{\op}} 
     \leq \frac{\kappa(\norm{\A_P}+\lambda)}{\norm{\A_P}_{\op}\lambda} \leq \frac{2\kappa}{\lambda}.
\end{align*}
Thus applying \citet[Theoerm D.2]{kpca} yields with probability at least $1-\delta$ that 
$$\norm{ \SgLP (\hat{\A}_P- \A_P) \SgLP}_{\op} \leq \frac{\beta \kappa}{\lambda n}+\sqrt{\frac{3\beta\kappa}{\lambda n}},$$
where $\beta := \frac{2}{3}\log \frac{4d}{\delta}.$ Then the desired result in $(a)$ follows under the condition on $\lambda$ in \emph{(iii)}.

Finally, it remains to show that $(b)$ and $(c)$ follows from $(a)$. To this end, define $B_n:=\SgLP(\A_P-\hat{\A}_P)\SgLP,$ and observe that
\begin{align*}
\norm{\A_{P,\lambda}^{1/2}\hat{\A}_{P,\lambda}^{-1/2}}_{\op} &= \norm{\A_{P,\lambda}^{1/2}\hat{\A}_{P,\lambda}^{-1}\A_{P,\lambda}^{1/2}}^{1/2}_{\op} \\
&=\norm{(\Id-B_n)^{-1}}^{1/2}_{\op} \leq (1-\norm{B_n}_{\op})^{-1/2},
\end{align*}
and 
\begin{align*}
\norm{\A_{P,\lambda}^{-1/2}\hat{\A}_{P,\lambda}^{1/2}}_{\op} = \norm{\Id-B_n}_{\op}^{1/2} \leq (1+\norm{B_n}_{\op})^{1/2},
\end{align*}
the desired results follow.
\end{proof}

\begin{appxlem} \label{lem:bound-op-norm-2}
    Let $\phi_i$ be as defined in Lemma \ref{Lem:AP-AP0} with $\sup_{i}\norm{\phi_i}^2_{\op}\leq C < \infty$. If $$4C\Nol\norm{u}_{\Lp}^2 < 1,$$ $n> \frac{200C\Nol}{3}\log\left(\frac{12\Nol}{v_p \delta}\right),$ and $\lambda \leq \norm{\A_{P_0}}_{\op},$ where $v_p := \E_{P}\phi_1^2 >0.$ 
    Then for $0<\delta\leq 1$,
\begin{enumerate}
    \item $P\left(\norm{\SgL(\A_P-\hat{\A}_P)\SgL}_{\op}\leq \frac{1}{4}\right)\geq 1-\delta$;
    \item $P\left(\norm{\A_{P_0,\lambda}^{1/2}\hat{\A}_{P,\lambda}^{-1/2}}_{\op}\leq 2\right) \geq 1-\delta$;
    \item $P\left(\norm{\A_{P_0,\lambda}^{-1/2}\hat{\A}_{P,\lambda}^{1/2}}_{\op}\leq \frac{\sqrt{7}}{2}\right) \geq 1-\delta,$
\end{enumerate}
where $$\hat{\A}_{P}=\frac{1}{n}\sum_{i=1}^n K(\cdot,Z_i)\hhtens K(\cdot,Z_i), \ \ \ Z_i \stackrel{i.i.d}{\sim} P_0.$$
\end{appxlem}
\begin{proof}
    Let $\xi_i:= \A_{P_0,\lambda}^{-1/2}(K(\cdot,Z_i) \hhtens K(\cdot,Z_i))\A_{P_0,\lambda}^{-1/2}.$ Then $\E(\xi_i)=\A_{P_0,\lambda}^{-1/2}\A_{P}\A_{P_0,\lambda}^{-1/2},$ and 
    $$\frac{1}{n}\sum_{i=1}^n \xi_i-\E(\xi)=\A_{P_0,\lambda}^{-1/2}(\hat{\A}_{P}-\A_P)\A_{P_0,\lambda}^{-1/2}.$$
Observe that 
\begin{align*}
    &\norm{\xi(Z)}_{\op} 
    = \norm{\A_{P_0,\lambda}^{-1/2}(K(\cdot,Z) \hhtens K(\cdot,Z))\A_{P_0,\lambda}^{-1/2}}_{\op}  \\
    &=\norm{\A_{P_0,\lambda}^{-1/2}K(\cdot,Z)}_{\h_K}^2 = \inner{K(\cdot,Z)}{\A_{P_0,\lambda}^{-1}K(\cdot,Z)}_{\h_K} \\
    &= \left\langle K(\cdot,Z),\sum_{i\geq1}\frac{1}{\lambda_i+\lambda}\inner{K(\cdot,Z)}{\sqrt{\lambda_i}\phi_i}_{\h_K}\sqrt{\lambda_i}\phi_i\right.\\
    &\qquad\qquad\left.+\frac{1}{\lambda}(K(\cdot,Z)-\sum_{i}\inner{K(\cdot,Z)}{\sqrt{\lambda_i}\phi_i}_{h_K}\sqrt{\lambda_i}\phi_i\right\rangle_{\h_K} \\
    &=\sum_{i}\frac{\lambda_i}{\lambda_i+\lambda}\phi_{i}^2(Z)+\frac{1}{\lambda}\left(K(Z,Z)-\sum_{i}\lambda_i\phi_i(Z)\phi_i(Z)\right) \\
    &\stackrel{(*)}{=} \sum_{i}\frac{\lambda_i}{\lambda_i+\lambda}\phi_{i}^2(Z) \leq \sup_{i}\norm{\phi_i}_{\op}^2 \Nol \leq C\Nol,
\end{align*}
     where $(*)$ follows by Mercer's theorem (see \citealt[Lemma 2.6]{Steinwart2012MercersTO}). Also, we have 
\begin{align*}
    \E\left(\xi-\E(\xi)\right)^2 &\preccurlyeq \E(\xi^2) \\&=\E\left(\A_{P_0,\lambda}^{-1/2}(K(\cdot,Z_i) \hhtens K(\cdot,Z_i))\A_{P_0,\lambda}^{-1}(K(\cdot,Z_i) \hhtens K(\cdot,Z_i))\A_{P_0,\lambda}^{-1/2}\right) \\ 
     &\preccurlyeq \sup_Z \norm{\xi(Z)}_{\op} \E(\A_{P_0,\lambda}^{-1/2}(K(\cdot,Z_i) \hhtens K(\cdot,Z_i))\A_{P_0,\lambda}^{-1/2}) \\
     & \preccurlyeq C \Nol \A_{P_0,\lambda}^{-1/2} \A_{P} \A_{P_0,\lambda}^{-1/2} := S
\end{align*}
so that
\begin{align*}
    \norm{S}_{\op} &= C\Nol \norm{\A_{P_0,\lambda}^{-1/2} \A_{P} \A_{P_0,\lambda}^{-1/2}}_{\op} 
    \stackrel{(\dag)}{\leq} C\Nol \left(1+\sqrt{C\Nol}\norm{u}_{\Lp}\right) \\&:= \sigma^2,
\end{align*}
where in $(\dag)$ we used the bound in \eqref{eq:hs-bound}. Define 
\begin{align*}
    d:= \frac{\text{Tr}(S)}{\norm{S}_{\op}} &= \frac{\text{Tr}(\A_{P_0,\lambda}^{-1/2} \A_{P} \A_{P_0,\lambda}^{-1/2})}{\norm{\A_{P_0,\lambda}^{-1/2} \A_{P} \A_{P_0,\lambda}^{-1/2}}_{\op}} 
    \stackrel{(**)}{\leq} \frac{\Nol+C\Nol\norm{u}_{\Lp}}{\norm{\A_{P_0,\lambda}^{-1/2} \A_{P} \A_{P_0,\lambda}^{-1/2}}_{\op}} \\
    & \stackrel{(\ddag)}{\leq} \frac{\left(\Nol+C\Nol\norm{u}_{\Lp}\right)\left(\norm{\A_{P_0}}_{\op}+\lambda\right)}{v_p \norm{\A_{P_0}}_{\op}} \\ &\leq \frac{2}{v_p} \Nol (1+C\norm{u}_{\Lp}),
\end{align*}
where the last inequality follows by using $\lambda \leq \norm{\A_{P_0}}_{\op}$. $(**)$ follows using Lemma \ref{Lem:AP-AP0}\emph{(ii)} and in $(\ddag)$ we used $$\norm{\A_{P_0,\lambda}^{-1/2} \A_{P} \A_{P_0,\lambda}^{-1/2}}_{\op} \geq \frac{v_p\norm{\A_{P_0}}_{\op}}{\norm{\A_{P_0}}_{\op}+\lambda},$$ which is proven below.
\begin{align*}
&\norm{\A_{P_0,\lambda}^{-1/2} \A_{P} \A_{P_0,\lambda}^{-1/2}}_{\op} = \sup_{\norm{f}_{\h_K}\leq 1} \norm{\A_{P_0,\lambda}^{-1/2} \A_{P} \A_{P_0,\lambda}^{-1/2} f}_{\h_K} \\
    & \geq \norm{\A_{P_0,\lambda}^{-1/2} \A_{P} \A_{P_0,\lambda}^{-1/2}\psi_1}_{\h_K} 
    = \norm{\sum_{i}\frac{\sqrt{\lambda_1}}{\sqrt{\lambda_1+\lambda}}\frac{\sqrt{\lambda_i}}{\sqrt{\lambda_i+\lambda}}(\E_P \phi_i \phi_1)\psi_i}_{\h_K} \\
    & = \sqrt{\frac{\lambda_1}{\lambda_1+\lambda}\sum_i \frac{\lambda_i}{\lambda_i+\lambda}\left(\E_P\phi_i \phi_1\right)^2} 
    \geq \frac{\lambda_1}{\lambda_1+\lambda}\E_P\phi_1^2 = \frac{v_p\norm{\A_{P_0}}_{\op}}{\norm{\A_{P_0}}_{\op}+\lambda},   
\end{align*}
where $\psi_i=\sqrt{\lambda_i} \phi_i$ are orthonormal eigenfunctions of $\A_{P_0}$. Finally, we apply Tropp's inequality for operators as stated in \citet[Theorem D.2]{kpca}, which yields that with probability at least $1-\delta$, 
$$\norm{\SgL(\A_P-\hat{\A}_P)\SgL}_{\op} \leq \frac{C\beta\Nol}{n}+\sqrt{\frac{3C\beta\Nol(1+\sqrt{C\Nol}\norm{u}_{\Lp})}{n}},$$ 
where $\beta := \frac{2}{3}\log \frac{4d}{\delta}.$

Thus, the desired result in
\emph{(i)} follows under the conditions  $4C\Nol\norm{u}_{\Lp}^2 < 1$ and \\ $n> \frac{200C\Nol}{3}\log\left(\frac{12\Nol}{v_p \delta}\right)$. Then it remains to show that \emph{(ii)} and \emph{(iii)} follows from \emph{(i)}. 
For \emph{(ii)} observe that  
\begin{align*}
\norm{\A_{P_0,\lambda}^{1/2}\hat{\A}_{P,\lambda}^{-1/2}}^2_{\op} &= \norm{\A_{P_0,\lambda}^{1/2}\hat{\A}_{P,\lambda}^{-1}\A_{P_0,\lambda}^{1/2}}_{\op} \\ 
&= \norm{\Id- \A_{P_0,\lambda}^{-1/2}(\A_{P_0}-\hat{\A}_{P})\A_{P_0,\lambda}^{-1/2}}_{\op} \\
& \leq \left(1-\norm{\A_{P_0,\lambda}^{-1/2}(\A_{P_0}-\hat{\A}_{P})\A_{P_0,\lambda}^{-1/2}}_{\op}\right)^{-1} \leq 4,
\end{align*}
where the last inequality follows using 
\begin{align*}
  &\norm{\A_{P_0,\lambda}^{-1/2}(\A_{P_0}-\hat{\A}_{P})\A_{P_0,\lambda}^{-1/2}}_{\op} \\
  &\leq \norm{\A_{P_0,\lambda}^{-1/2}(\A_{P}-\hat{\A}_{P})\A_{P_0,\lambda}^{-1/2}}_{\op} + \norm{\A_{P_0,\lambda}^{-1/2}(\A_{P_0}-\A_{P})\A_{P_0,\lambda}^{-1/2}}_{\op} \\
  & \stackrel{(*)}{\leq} \frac{1}{4} + \sqrt{C\Nol}\norm{u}_{\Lp} < \frac{3}{4}.
\end{align*}
For \emph{(iii)} we have 
\begin{align*}
    &\norm{\A_{P_0,\lambda}^{-1/2}\hat{\A}_{P,\lambda}^{1/2}}_{\op}^2 \\&=\norm{\A_{P_0,\lambda}^{-1/2}\hat{\A}_{P,\lambda}\A_{P_0,\lambda}^{-1/2}}_{\op} \\
    & = \norm{\A_{P_0,\lambda}^{-1/2}(\hat{\A}_{P,\lambda}-\A_{P,\lambda}+\A_{P,\lambda})\A_{P_0,\lambda}^{-1/2}}_{\op} \\
    & = \norm{\A_{P_0,\lambda}^{-1/2}(\hat{\A}_{P}-\A_{P})\A_{P_0,\lambda}^{-1/2}+ \A_{P_0,\lambda}^{-1/2}\A_{P,\lambda}\A_{P_0,\lambda}^{-1/2}}_{\op} \\ 
    & = \norm{\A_{P_0,\lambda}^{-1/2}(\hat{\A}_{P}-\A_{P})\A_{P_0,\lambda}^{-1/2}+ \A_{P_0,\lambda}^{-1/2}(\A_{P}-\A_{P_0})\A_{P_0,\lambda}^{-1/2}+\Id}_{\op} \\
    & \leq 1 + \norm{\A_{P_0,\lambda}^{-1/2}(\hat{\A}_{P}-\A_{P})\A_{P_0,\lambda}^{-1/2}}_{\op} + \norm{\A_{P_0,\lambda}^{-1/2}(\A_{P}-\A_{P_0})\A_{P_0,\lambda}^{-1/2}}_{\op} \\
    & \stackrel{(**)}{\leq} 1+ \frac{1}{4} + \sqrt{C\Nol}\norm{u}_{\Lp} \leq \frac{7}{4},
\end{align*}
where $(*)$ and $(**)$ follow from \eqref{eq:hs-bound}.
\end{proof}

\begin{appxlem} \label{Lem: bound zeta}
    Let $\zeta= \norm{\gSh \Psi_P}_{\h_{K_0}}^2,$ $\W=\hat{\A}_{P,\lambda}^{-1/2}\A_{P,\lambda}^{1/2}$, and $\V=\hat{\A}_{P,\lambda}^{-1/2}\A_{P_0,\lambda}^{1/2}$, where $g_\lambda$ satisfies $(E_1)$--$(E_4)$. Suppose  $u \in \emph{Ran} (\ep_{P_0}^{\theta})$, $\theta>0$ and $$\norm{\U}_{\Lp}^2 \geq  d_1\norm{\ep_{P_0}}_{\opl}^{2\max(\theta-\xi,0)}\lambda^{2 \Tilde{\theta}} \norm{\ep_{P_0}^{-\theta}\U}_{\Lp}^2,$$  where $\Tilde{\theta}= \min(\theta,\xi)$, for some constant $d_1>0$. Then 
\begin{enumerate}
    \item $\zeta \geq d_2\norm{\W^{-1}}_{\opS}^{-2} \frac{\norm{u}^2_{\Lp}}{1+\sqrt{\Cl}\norm{u}_{\Lp}}$;
    \item $\zeta \geq d_3\norm{\V^{-1}}_{\opS}^{-2} \norm{u}^2_{\Lp},$
\end{enumerate}
where $\Cl$ is defined in Lemma \ref{Lem:AP-AP0} with $K=K_0$ and $d_2,d_3$ are positive constants.
\end{appxlem}
\begin{proof}
\emph{(i)}: 
By applying Lemma \ref{lemma: bounds for g}(ii) we have 
    \begin{align*}
        \zeta &\stackrel{(*)}{\geq} \norm{\hat{\A}_{P,\lambda}^{-1/2} g^{-1/2}_{\lambda}(\hat{\A}_{P})}_{\opS}^{-2} \norm{\hat{\A}_{P,\lambda}^{-1/2} \Psi_P}_{\h_{K_0}}^2 
         \geq C_4 \norm{\W^{-1}}_{\opS}^{-2} \norm{\A_{P,\lambda}^{-1/2} \Psi_P}_{\h_{K_0}}^2 \\
        & \stackrel{(\dag)}{\geq}  C_4^2\norm{\W^{-1}}_{\opS}^{-2} \frac{\norm{\A_{P_0,\lambda}^{-1/2} \Psi_P}_{\h_{K_0}}^2}{1+\sqrt{\Cl}\norm{u}_{\Lp}},
    \end{align*}

    where $(*)$ follows using $$\norm{\hat{\A}_{P,\lambda}^{-1/2} \Psi_P}_{\h_{K_0}}^2=\norm{\hat{\A}_{P,\lambda}^{-1/2} g^{-1/2}_{\lambda}(\hat{\A}_{P}) g^{1/2}_{\lambda}(\hat{\A}_{P})\Psi_P}_{\h_{K_0}}^2 \leq \norm{\hat{\A}_{P,\lambda}^{-1/2}g^{-1/2}_{\lambda}(\hat{\A}_{P})}_{\opS}^{2} \zeta,$$
    and $(\dag)$ follows from Lemma \ref{Lem: SP-SP0}. For the bound in \emph{(ii)} we have,

\begin{align*}
   \zeta &\geq \norm{\hat{\A}_{P,\lambda}^{-1/2} g^{-1/2}_{\lambda}(\hat{\A}_{P})}_{\opS}^{-2} \norm{\hat{\A}_{P,\lambda}^{-1/2} \Psi_P}_{\h_{K_0}}^2   \geq C_4 \norm{\V^{-1}}_{\opS}^{-2}\norm{\A_{P_0,\lambda}^{-1/2} \Psi_P}_{\h_{K_0}}^2.
\end{align*}
Thus, it remains to show $$\norm{\A_{P_0,\lambda}^{-1/2} \Psi_P}_{\h_{K_0}}^2 \geq \tilde{d} \norm{u}_{\Lp}^2,$$ for some $\tilde{d}>0.$ For that matter, consider
\begin{align*}
   4 B_3 \norm{\A_{P_0,\lambda}^{-1/2} \Psi_P}_{\h_{K_0}}^2& = 4 B_3\inner{\ep_{P_0} g_{\lambda}(\ep_{P_0}) u}{u}_{\Lp}\\
    &= \norm{\ep_{P_0} g_{\lambda}(\ep_{P_0}) u}_{\Lp}^2+4B_3^2\norm{\U}_{\Lp}^2-\norm{\ep_{P_0} g_{\lambda}(\ep_{P_0}) u-2B_3\U}_{\Lp}^2.
\end{align*}
Since $u \in \range(\ep_{P_0}^{\theta})$, there exists $f\in \Lp$ such that $u= \ep_{P_0}^{\theta}f$. Therefore, we have 
$$\norm{\ep_{P_0} g_{\lambda}(\ep_{P_0}) u}_{\Lp}^2 = \sum_{i} \lambda_i^{2\theta+2}\gl^{2}(\lambda_i) \inner{f}{\Tilde{\phi}_i}_{\Lp}^2,$$ and $$\norm{\ep_{P_0} g_{\lambda}(\ep_{P_0}) u-2B_3 u}_{\Lp}^2 = \sum_{i} \lambda_i^{2\theta}(\lambda_i\gl(\lambda_i)-2B_3)^2 \inner{f}{\Tilde{\phi}_i}_{\Lp}^2,$$ where $(\lambda_i,\tilde{\phi}_i)_i$ are the eigenvalues and eigenfunctions of $\ep_{P_0}$. Using these expressions we have 
\begin{align*}
    \norm{\ep_{P_0} g_{\lambda}(\ep_{P_0}) u}_{\Lp}^2-\norm{\ep_{P_0} g_{\lambda}(\ep_{P_0}) u-2B_3\U}_{\Lp}^2
    = \sum_{i} 4B_3 \lambda_i^{2\theta}(\lambda_i\gl(\lambda_i)-B_3)\inner{f}{\Tilde{\phi}_i}_{\Lp}^2.
\end{align*}
Thus 
\begin{align*}
    4B_3\norm{\A_{P_0,\lambda}^{-1/2} \Psi_P}_{\h_{K_0}}^2 & = 4B_3^2\norm{u}_{\Lp}^2+\sum_{i} 4B_3 \lambda_i^{2\theta}\left(\lambda_i\gl(\lambda_i)-B_3\right)\inner{f}{\Tilde{\phi}_i}_{\Lp}^2 \\
    & \geq 4B_3^2\norm{u}_{\Lp}^2 - \sum_{\{i:\lambda_{i}\gl(\lambda_i)< B_3\}} 4B_3 \lambda_i^{2\theta}\left(B_3-\lambda_i\gl(\lambda_i)\right)\inner{f}{\Tilde{\phi}_i}_{\Lp}^2.
\end{align*}
When $\theta \leq \xi$, by Assumption $(E_3)$, we have
\begin{equation*}
    \sup_{\{i:\lambda_{i}\gl(\lambda_i)< B_3\}} \lambda_i^{2\theta}\left(B_3-\lambda_i\gl(\lambda_i)\right) \leq C_3 \lambda^{2\theta}. 
\end{equation*}
On the other hand, for $\theta > \xi$,
\begin{align*}
    &\sup_{\{i:\lambda_{i}\gl(\lambda_i)< B_3\}} \lambda_i^{2\theta}\left(B_3-\lambda_i\gl(\lambda_i)\right)\\
    &\le \sup_{\{i:\lambda_{i}\gl(\lambda_i)< B_3\}} \lambda_i^{2\theta-2\xi}\sup_{\{i:\lambda_{i}\gl(\lambda_i)< B_3\}}\lambda_i^{2\xi}\left(B_3-\lambda_i\gl(\lambda_i)\right)\\ &\stackrel{(*)}{\le} C_3\norm{\ep_{P_0}}^{2\theta-2\xi}_{\opl} \lambda^{2\xi},
\end{align*}
where $(*)$ follows by Assumption $(E_3)$. Therefore we can conclude that
\begin{align*}
    \norm{\A_{P_0,\lambda}^{-1/2} \Psi_P}_{\h_{K_0}}^2 \geq B_3\norm{u}^2_{\Lp}-C_3\norm{\ep_{P_0}}_{\opl}^{2\max(\theta-\xi,0)}\lambda^{2 \Tilde{\theta}} \norm{\ep_{P_0}^{-\theta}u}_{\Lp}^2 \stackrel{(\dag)}{\geq} \frac{B_3}{2}\norm{u}^2_{\Lp},
\end{align*}
where we used $\norm{u}_{\Lp}^2 \geq  \frac{2C_3}{B_3} \norm{\ep_{P_0}}_{\opl}^{2\max(\theta-\xi,0)}\lambda^{2 \Tilde{\theta}} \norm{\ep_{P_0}^{-\theta}u}_{\Lp}^2$ in $(\dag)$.
\end{proof}

\begin{appxlem} \label{lem:bound BS quantile}
    For any $\delta >0$, there exists $\tilde{C}$ such that 
    $$P_{H_1}(q_{1-\alpha}^{\lambda} \leq \tilde{C}\gamma) \geq 1-\delta,$$
    where $\gamma := \frac{\log \frac{2}{\alpha}}{n_1\sqrt{\delta}} \left(\norm{\W}_{\opS} \sqrt{\zeta} + \norm{\W}_{\opS}^2 \sqrt{D_{\lambda}}+\zeta\right)$. Furthermore the operator $\W$ can be replaced by $\V$ when $4\Cl\norm{u}_{\Lp}^2 \leq 1$,
    where $\W$, $\V,$ $D_{\lambda}$ and $\zeta$ are defined in Lemma \ref{lem:bound-var} and $\Cl$ is defined in Lemma \ref{Lem:AP-AP0}. 
\end{appxlem}
\begin{proof}
    Let $h(x,y):= \inner{\gSh K_0(\cdot,x)}{\gSh K_0(\cdot,y)}_{\h_{K_0}}$. Conditioned on the samples $\mathbb{X}_n$, it follows from 
\citet[Corollary 3.2.6]{delapena} that
$$P_{\epsilon}\left(\left|\sum_{i \neq j} \epsilon_i \epsilon_j h(X_i,X_j)\right| \geq t \right) \leq 2\exp\left(\frac{-at}{\sqrt{\sum_{i,j}h^2(X_i,X_j)}}\right),$$ for some constant $a>0$. This implies that 
\begin{equation}
  q_{1-\alpha}^{\lambda} \leq  \frac{\tilde{C}\log \frac{2}{\alpha}}{n_1(n_1-1)} \sqrt{\sum_{i,j}h^2(X_i,X_j)},  \label{eq:BS-quantile-bound}
\end{equation} 
almost surely. Let $I:=\frac{1}{n_1(n_1-1)}\sqrt{\sum_{i,j}h^2(X_i,X_j)}$. Then similar to the proof of Lemma \ref{lem:bound-var}, we can show that 
$$\E\left(I^2 | \mathbb Z_{n_2}\right) \lesssim \frac{1}{n_1^2}\left(\norm{\W}_{\opS}^2 \zeta+\norm{\W}_{\opS}^4 D_{\lambda} + \zeta^2\right),$$ where $\W$ can be replaced by $\V$ when $4\Cl \norm{u}^2_{\Lp} \leq 1.$  Thus using \eqref{eq:BS-quantile-bound} and Markov's inequality, we obtain the desired result.
\end{proof}

 \begin{appxlem} \label{lemma: bounds for g}
Let $\mathcal{C}$ be any compact, self-adjoint operator defined on a separable Hilbert space, $H$ and $(\tau_i,\alpha_i)_i$ are the eigenvalues and eigenfunctions of $\mathcal{C}$. Then for $g_{\lambda}$ that satisfies $(E_1)$--$(E_4)$ the following hold.
\begin{enumerate}[label=(\roman*)]
    \item $\norm{g_{\lambda}^{1/2}(\mathcal{C})(\mathcal{C}+\lambda\Id)^{1/2}}_{\op} \leq (C_1+C_2)^{1/2}$;
    \item $\norm{(\mathcal{C}+\lambda \Id)^{-1/2} g_{\lambda}^{-1/2}(\mathcal{C})}_{\op} \leq C_4^{-1/2}$;

\end{enumerate}
\end{appxlem}
\begin{proof}

\emph{(i)} 
\begin{align*}
    \norm{g_{\lambda}^{1/2}(\mathcal{C})(\mathcal{C}+\lambda\Id)^{1/2}}_{\op}& = \norm{g_{\lambda}^{1/2}(\mathcal{C})(\mathcal{C}+\lambda\Id)g_{\lambda}^{1/2}(\mathcal{C})}_{\op}^{1/2} \\
    &= \sup_{i}\left| g_{\lambda}(\tau_i)(\tau_i+\lambda) \right|^{1/2} \stackrel{(\dag)}{\leq} (C_1+C_2)^{1/2},
\end{align*}

where $(\dag)$ follows from Assumptions $(E_1)$ and $(E_2)$.\vspace{1mm} \\
\emph{(ii)} 
\begin{align*}
    \norm{(\mathcal{C}+\lambda \Id)^{-1/2} g_{\lambda}^{-1/2}(\mathcal{C})}_{\op}& =\norm{(\mathcal{C}+\lambda \Id)^{-1/2} g_{\lambda}^{-1}(\mathcal{C})(\mathcal{C}+\lambda \Id)^{-1/2}}_{\op}^{1/2} \\&= \sup_i\left|(\lambda+\tau_i)g_{\lambda}(\tau_i)\right|^{-1/2} \\ & \leq \left|\inf_i (\lambda+\tau_i)g_{\lambda}(\tau_i)\right|^{-1/2} \stackrel{(\ddag)}{\leq} C_4^{-1/2}, 
\end{align*}
where $(\ddag)$ follows from Assumption $(E_4)$.
\end{proof}

\section{Additional Experiments}
Recall that, given $n$ samples, we split them into two groups of sizes $n_1$ and $n_2$, where $n = n_1 + n_2$. Here, $n_2$ represents the number of samples from $P$ used to estimate the covariance operator. In this section, we provide additional experiments highlighting the impact of $n_2$ on the test power. Figures \ref{fig:RBM_n2}-\ref{fig:inf_freq_n2} show the test power of the proposed regularized KSD test for the experiments corresponding to Figures \ref{fig:RBM}-\ref{fig:Infdim2} in Section \ref{sec:experiments}, with different choices of \( n_2 \). For each experiment, we display the power of either KSD(Tikhonov) or KSD(TikMax), selecting the one that performs the best. Based on these figures, the overall takeaway is that test performance is generally not very sensitive to the choice of \( n_2 \). However, if \( n_2 \) is chosen too large (so that \( n_1 \) becomes too small), performance deteriorates. Similarly, when \( n_2 \) is too small, the performance also declines due to insufficient samples to estimate the covariance operator accurately. These results suggest that the best performance is achieved when \( n_2 \) is reasonably small relative to \( n_1 \) but not too small. For example, choosing \( n_2 \approx n/5 \) appears to yield good performance across all experiments we conducted. Alternatively, if training data is available, \( n_2 \) can be selected to maximize the power observed in the training set.

\begin{figure}[H]
\centering
\includegraphics[scale=0.4]{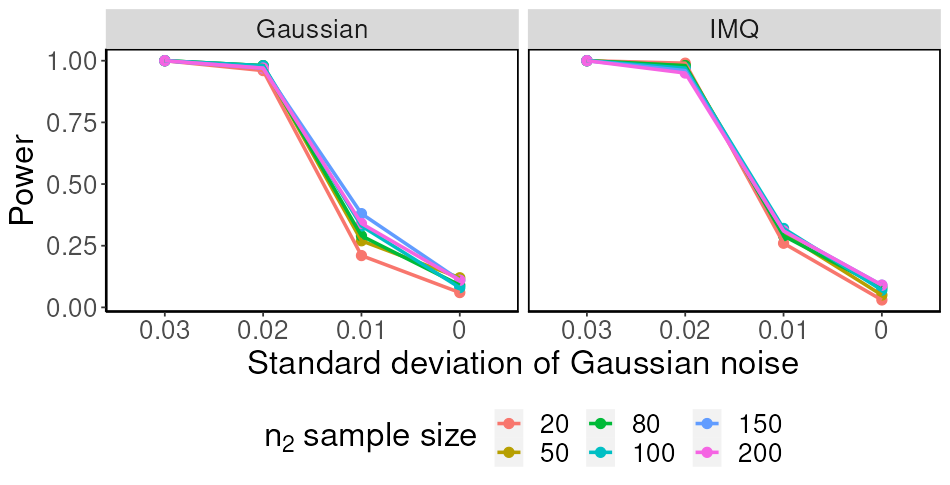}
\caption{Power of the KSD(Tikhonov) test with varying choices of $n_2$ for Gaussian-Bernoulli restricted Boltzmann machine with $d=50$ and a total sample size of $n=1000$.} \label{fig:RBM_n2}
\vspace{-2mm}
\end{figure}

\begin{figure}[H]
\centering
\includegraphics[scale=0.4]{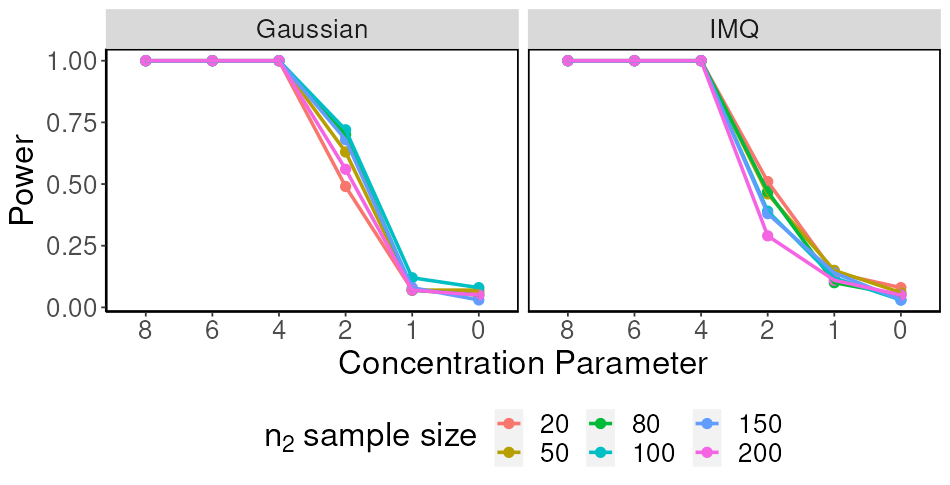}
\caption{Power of the KSD(TikMax) test with varying choices of $n_2$ for mixture of Watson distributions with a total sample size of $n=500$.} \label{fig:watson_n2}
\vspace{-2mm}
\end{figure}

\begin{figure}[H]
\centering
\includegraphics[scale=0.4]{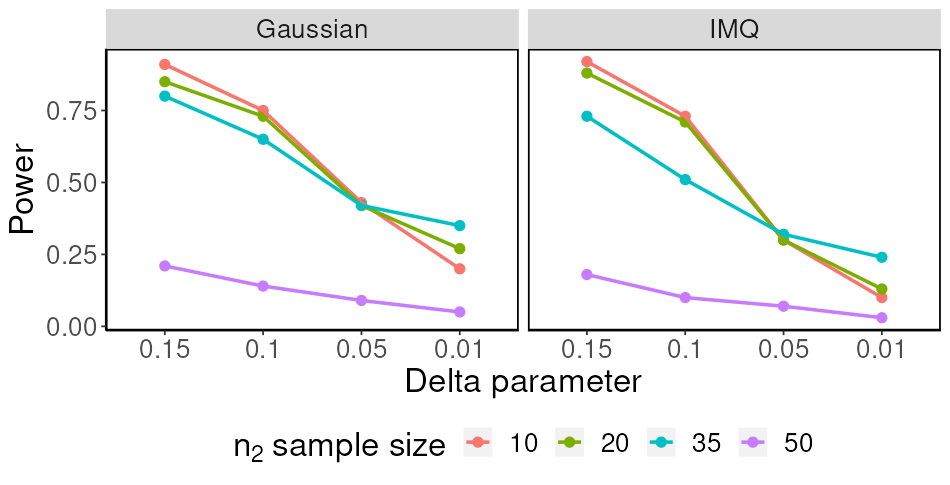}
\caption{Power of the KSD(Tikhonov) test with varying choices of $n_2$ for Brownian motion ($\delta$-perturbation experiment) with a total sample size of $n=70$.} \label{fig:inf_delta_n2}
\vspace{-2mm}
\end{figure}

\begin{figure}[H]
\centering
\includegraphics[scale=0.4]{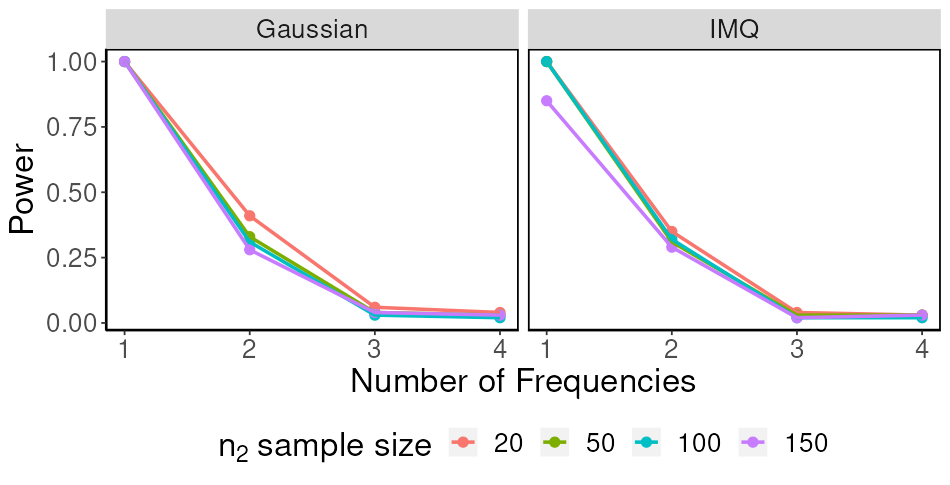}
\caption{Power of the KSD(Tikhonov) test with varying choices of $n_2$ for Brownian motion (truncated frequency experiment) with a total sample size of $n=200$.} \label{fig:inf_freq_n2}
\vspace{-2mm}
\end{figure}

\end{document}